\documentclass[a4paper,11pt,reqno]{amsart}
\usepackage{amsmath,amssymb,amsthm,mathrsfs,amsfonts,bbm,dsfont}
\usepackage[all]{xy}
\usepackage{latexsym}
\usepackage{titletoc}
\newtheorem{theorem}{Theorem}[section]
\newtheorem{corollary}{Corollary}[section]
\newtheorem{lemma}{Lemma}[section]
\newtheorem{proposition}{Proposition}[section]
\newtheorem{notation}{Notation}[section]
\theoremstyle{definition}

\theoremstyle{remark}
\newtheorem{remark}{Remark}[section]
\numberwithin{equation}{section}

\begin{document}
 \title{On the Yamabe Problem on contact Riemannian Manifolds}
 \date{}
 \author{Feifan Wu} \thanks{Department of Mathematics, Zhejiang University, Hangzhou 310027, P. R. China, Email: wesleywufeifan08@sina.com.} \author{Wei Wang} \thanks{Department of Mathematics, Zhejiang University, Hangzhou 310027, P. R. China, Email: wwang@zju.edu.cn.} \thanks{
Supported by National Nature Science Foundation in China (No.
11171298)}
 \maketitle
 \begin{abstract}
 Contact Riemannian manifolds, whose complex structures are not necessarily integrable, are generalization of pseudohermitian manifolds in CR geometry. The Tanaka-Webster-Tanno connection plays the role of the Tanaka-Webster connection of a pseudohermitian manifold. Conformal transformations and the Yamabe problem are also defined naturally in this setting. By constructing the special frames and the normal coordinates on a contact Riemannian manifold, we prove that if the complex structure is not integrable, its Yamabe invariant on a contact Riemannian manifold is always less than the Yamabe invariant of the Heisenberg group. So the Yamabe problem on a contact Riemannian manifold is always solvable.
 \end{abstract}
\tableofcontents
 \section{Introduction}
 The Yamabe Problem in Riemannian Geometry was completely solved during 70's-80's (cf. \cite{A1}, \cite{JP1}, \cite{S1}, \cite{TR1} and references therein). For the analogous CR Yamabe problem, Jerison and Lee proved in \cite{JL2} \cite{JL4} that there is a numerical CR invariant $\lambda(M)$ called the Yamabe invariant, and for any compact, strictly pseudoconvex $(2n+1)$-dimensional CR manifold $M$, it is always less than or equal to the Yamabe invariant $\lambda(\mathscr{H}^n)$ of the Heisenberg group $\mathscr{H}^n$. Furthermore if $\lambda(M)<\lambda(\mathscr{H}^n)$, then $M$ admits a pseudohermitian structure with constant scalar curvature. In \cite{JL1}, Jerison and Lee proved that $\lambda(M)<\lambda(\mathscr{H}^n)$ holds if $n>2$ and $M$ is not locally CR equivalent to $S^{2n+1}$. The remaining case was solved by Gamara and Yacoub in \cite{GA1} \cite{GY1}. The purpose of this paper is to solve the Yamabe problem on general contact Riemannian manifolds. This problem has already been studied by Zhang \cite{ZH1} by using the contact Yamabe flow.

 A $(2n+1)$-dimensional manifold $(M,\theta)$ is called a {\it contact manifold} if it has a real 1-form $\theta$ such that $\theta{\wedge}d\theta^n\ne0$ everywhere on $M$. We call such $\theta$ a {\it contact form}. There exists a unique vector field $T$, {\it the Reeb vector field}, such that $\theta(T)=1$ and $T\lrcorner{d\theta}=0$. It's known that given a contact manifold $(M,\theta)$, there are a Riemannian metric $h$ and a $(1,1)$-tensor field $J$ on $M$ such that
 \begin{equation}\label{eq2.1}
  \begin{aligned}
  &h(X,T)=\theta(X),\\
  &J^2=-Id+\theta\otimes{T},\\
  &d\theta(X,Y)=h(X,JY),
  \end{aligned}
 \end{equation}
 for any vector field $X,Y$ (cf. p. 278 in \cite{BD1} or p. 351 in \cite{T1}). Such a metric $h$ is said to be associated with $\theta$, and $J$ is called an {\it almost complex structure}. Given a   contact form $\theta$, once $h$ is fixed, $J$ is uniquely determined, and vise versa. $(M,\theta,h,J)$ is called a {\it contact Riemannian manifold}. $HM:=Ker(\theta)$ is called the {\it horizontal subbudle} of the tangent bundle $TM$. On a contact Riemannian manifold, there exists a distinguished connection called the {\it Tanaka-Webster-Tanno connection} $\nabla$ (or {\it TWT connection} briefly). In the CR case, this connection is exactly the Tanaka-Webster connection (cf. \cite{TA1} and \cite{W1}). Tanno constructed this connection for general contact Riemannian manifolds in \cite{T1}. Since there is no obstruction to the existence of the almost complex structure $J$, contact Riemannian structures exist naturally on any contact manifold and analysis on it has potential applications to the geometry of contact manifolds (cf. e.g. \cite{Pe1}, \cite{Se1} and \cite{WW1}).

 Let $\mathbb{C}TM$ be the complexification of $TM$. $\mathbb{C}TM$ has a unique subbuddle $T^{(1,0)}M$ such that $JX=iX$ for any $X\in{\Gamma(T^{(1,0)}M)}$. Here and in the following, $\Gamma(S)$ denotes the space of all smooth sections of a vector bundle $S$. Set $T^{(0,1)}M=\overline{T^{(1,0)}M}$. Then for any $X\in{T^{(0,1)}M}$, $JX=-iX$. $J$ is called {\it integrable} if $[\Gamma(T^{(1,0)}M),\Gamma(T^{(1,0)}M)]{\subset}\Gamma(T^{(1,0)}M).$
 If $J$ is integrable, $J$ is called a {\it CR structure} and $(M,\theta,h,J)$ is called a {\it pseudohermitian manifold}. By \cite{T1}, the integrable condition holds if and only if the Tanno tensor $Q=\nabla{J}=0$. In general, a contact Riemman manifold is not a CR manifold.

  Under conformal transformations of a contact Riemannian manifold, which is given by
  $$\widehat{\theta}=f\theta,$$
  for some positive function $f$, we have the transformation formulae
  \[(\theta,J,T,h)\to(\widehat{\theta},\widehat{J},\widehat{T},\widehat{h}),\]
  with $\widehat{J},\widehat{T},\widehat{h}$ given by
  \begin{equation}\label{eqA.1}
  \begin{aligned}
  &\widehat{T}=\frac{1}{f}(T+\zeta),\\
  &\widehat{h}=fh-f\big(\theta\otimes\omega+\omega\otimes\theta\big)+f(f-1+||\zeta||^2)\theta\otimes\theta,\\
  &\widehat{J}=J+\frac{1}{2f}\theta\otimes\big(\nabla{f}-T(f)T\big),
  \end{aligned}
  \end{equation}
   (cf. (12) in \cite{BD1} or Lemma 9.1 in \cite{T1}), where $\zeta=\frac{1}{2f}J\nabla{f}$ and $\omega$ satisfies $\omega(X)=h(X,\zeta)$ for $X\in{TM}$.

  The contact Riemannian Yamabe problem is that given a compact contact Riemannian manifold $(M,\theta,h,J)$, find $\widehat{\theta}$ conformal to $\theta$ such that its scalar curvature is constant. It is known that (cf. p. 337 in \cite{BD1}) if we write the conformal transformation $\widehat{\theta}=f\theta$ with $f=u^{2/n}$, the scalar curvature $\widehat{R}$ of the TWT connection transforms as
  \begin{equation}\label{eq1.1}
  b_n\Delta_{\theta}u+Ru=\widehat{R}u^{(n+2)/n},{\quad}b_n=2+\frac{2}{n},
  \end{equation}
  where $\Delta_{\theta}$ is the sub-laplacian. \eqref{eq1.1} is the contact Riemannian Yamabe equation. The {\it Yamabe functional} is defined as
  \begin{equation}\label{eq1.2}
  \mathscr{Y}_{\theta,h}(u)=\frac{\int_M(p{|du|}_{H}^2+Ru^2)dV_{\theta}}{(\int_Mu^{p}dV_{\theta})^{2/p}},
  {\quad}p=b_n=2+\frac{2}{n},
  \end{equation}
where ${|du|}_{H}$ is the norm of the horizontal part of $du$ and $dV_{\theta}$ is the volume form. The solutions to the Yamabe problem are critical points of the Yamabe functional $\mathscr{Y}_{\theta,h}$. The {\it Yamabe invariant} is defined by
\begin{equation}\label{eq1.2a}
\lambda(M)=\inf\limits_{u}\mathscr{Y}_{\theta,h}(u).
\end{equation}
Equivalently,
\begin{equation}\label{eq1.3}
  \lambda(M)=inf\{A_{\theta,h}(u):B_{\theta,h}(u)=1\},
  \end{equation}
  where $A_{\theta,h}(u)=\int_M(b_n|du|_{H}^2+Ru^2)dV_{\theta}$, $B_{\theta,h}(u)=\int_M|u|^pdV_{\theta}$.

  Our main result in this paper is
  \begin{theorem}\label{thm1.1}
  Suppose $(M,\theta,h,J)$ is a compact contact manifold of dimension $2n+1$, $n\ge2$. If the almost complex structure $J$ is not integrable, then $\lambda(M)<\lambda(\mathscr{H}^n)$.
  \end{theorem}
  Here we require $n\ge2$ because the almost complex structure on a $3$-dimensional contact Riemannian manifold is automatically integrable.
  \begin{corollary}\label{cor1.1}
  If the contact Riemannian manifold $M$ is not integrable, then the infimum \eqref{eq1.2a} is attained by a positive $C^{\infty}$ solution to \eqref{eq1.1}. Thus the contact form $\widehat{\theta}=u^{p-2}\theta$ has constant scalar curvature $R\equiv\lambda(M)$.
  \end{corollary}
  It is well-known that the function
  \begin{equation}\label{eq5.2}
  \Phi(z,t)=\frac{1}{|w+i|^n},{\quad}w=t+i{|z|}^2,{\quad}z\in\mathbb{C}^{2n},{\quad}t{\in}\mathbb{R},
  \end{equation}
  is an extremal for the Yamabe functional on the Heisenberg group (cf. \cite{JL4}). For each $\varepsilon>0$, $\Phi^{\varepsilon}:=\varepsilon^{-n}\delta_{\frac{1}{\varepsilon}}^*\Phi=\varepsilon^n{|w+i\varepsilon^2|}^{-n}$ is also an extremal. As in CR case \cite{JL1}, we use the test function
  \[f^{\varepsilon}(z,t)=\psi(w)\Phi^{\varepsilon}(z,t),\]
  to calculate the asymptotic expansion for $\mathscr{Y}_{\theta,h}(f^{\varepsilon})$ as $\varepsilon{\to}0$, where $\psi\in{C_0^\infty}(M)$ is supported in the set $\{|w|<2\kappa\}$ and $\psi(w)=1$ for $|w|<\kappa$ for $\kappa>0$.

  To solve the CR Yamabe problem, Jerision and Lee constructed the pseudohermitian normal coordinates by parabolic geodesics and parabolic exponential map in \cite{JL1}. On a contact Riemannian manifold, in Section 2, we also have the parabolic geodesics analogous to the CR case. For a fixed point $q$, the parabolic geodesics induce a natural map from $T_qM$ to $M$, called the parabolic exponential map. The Reeb vector $T$ is automatically parallel along the parabolic geodesics. Choosing a basis $\{W_{\alpha;q}\}_{\alpha=1}^{n}$ of the complex vector space $T_q^{(1,0)}M$ and its conjugation $\{W_{\bar{\alpha};q}\} $ of $T_q^{(0,1)}M$, and extending them by parallel translation along the parabolic geodesics, together with $W_0=T$, we get a special frame in a neighborhood of $q$. The normal coordinates is the coordinates with respect to this special frame.

  In the CR case, the complex structure is preserved under the parallel translation, and so $T^{(1,0)}M$ and $T^{(0,1)}M$ are preserved. But on a general contact Riemannian manifold, the complex structure is not preserved under the parallel translation. Namely, the special frame $\{W_{\alpha}\}_{\alpha=1}^{n}$ is not a $T^{(1,0)}M$   frame even if it is a basis of $T_q^{(1,0)}M$   at point $q$. This is our main difficulty.

  With the normal coordinates, following the method in \cite{JL1}, the asymptotic expansion of $\mathscr{Y}_{\theta,h}(f^{\varepsilon})$ can be calculated explicitly by using certain invariants at the origin. These invariants are constructed by the curvature,  Webster torsion  and  Tanno tensors. In the CR case, besides the first term, the first nonzero term of the Yamabe functional $\mathscr{Y}_{\theta,h}(f^{\varepsilon})$ is $O(\varepsilon^4)$. Because our frame $\{W_a\}$ is not holomorphic or anti-holomorphic, our expansion of $\mathscr{Y}_{\theta,h}(f^{\varepsilon})$ is much more complicated than that in the CR case. Notably, we have to expand the almost complex structure $J$ asymptotically near $q$. While in the CR case, $J$ is constant. But fortunately, if the Tanno tensor is nonvanishing at point $q$, the second-order term of the Yamabe functional $\mathscr{Y}_{\theta,h}(f^{\varepsilon})$ is already nonzero. This makes the calculation easier than we expected.
  The Tanno tensor plays an important role in the analysis of contact Riemannian manifolds (see also \cite{WW1}).

  In Section 3, we construct the invariant
  $$\mathfrak{Q}=\sum\limits_{\alpha,\beta,\gamma}|Q_{\alpha\beta}^{\bar{\gamma}}(q)|^2,$$
  where $Q_{\alpha\beta}^{\bar{\gamma}}$ is the components of the Tanno tensor with respect to a special frame. The Tanno tensor is nonzero at point $q$ if and only if $\mathfrak{Q}$ is strictly positive at this point. In Section 4, as in the CR case in Section 3 in \cite{JL1}, for a fixed contact form $\theta$, we can make certain components of the curvature tensor $R_{\alpha\ \gamma\bar{\beta}}^{\ \gamma}(q)$ and the Webster torsion tensor vanish at point $q$ after a suitable conformal transformation. This will make our calculation easier.

  In Section 5, we calculate the asymptotic expansion for $\mathscr{Y}_{\theta,h}(f^{\varepsilon})$ explicitly. By the preparation in Section 3 and Section 4, we finally find
  \begin{equation}\label{eq1.6}
  \mathscr{Y}_{\theta,h}(f^{\varepsilon})=\lambda(\mathscr{H}^n)\bigg(1-\frac{3n-1}{12(n-1)n(n+1)}\mathfrak{Q}\varepsilon^2\bigg)+O(\varepsilon^3).
  \end{equation}
  So if the complex structure is not integrable, we prove the main theorem.

  In Appendix A, we discuss the transformation formulae of the connection coefficients, the Webster torsion tensor and    curvature tensors under the conformal transformations,  and the covariance of the Webster torsion   and    curvature tensors, which is used in Section 4. In Appendix B, we give the details of the calculation of the second-order terms of the Yamabe functional $\mathscr{Y}_{\theta,h}(f^{\varepsilon})$.

  Besides the Yamabe problem on Riemmanian manifolds, CR manifolds, and contact Riemannian manifolds, there is also the Yamabe problem on quaternionic contact manifolds (cf. \cite{IV1}, \cite{Wa2} and references therein). It is interesting to find the asymptotic expansion of this Yamabe functional.
Another interesting problem is to find the asymptotic expansion of the Yamabe-type functional on  differential forms \cite{Wa3}.

  \section{Construction of the normal coordinates}
  \subsection{The TWT connection}
  \begin{proposition}(cf. (7)-(9) in \cite{BD1})\label{prop2.0a}
  On a contact Riemannian manifold $(M,\theta,h,J)$, there exists a unique linear connection $\nabla$ such that
  \begin{equation}\label{TWT}
  \begin{aligned}
  &\nabla{\theta}=0,\quad\nabla{T}=0,\\
  &\nabla{h}=0,\\
  &\tau(X,Y)=2d\theta(X,Y)T,{\quad}X,Y\in{\Gamma(HM)},\\
  &\tau(T,JZ)=-J{\tau(T,Z)},{\quad}Z\in{{\Gamma(TM)}},
  \end{aligned}
  \end{equation}
  where $\tau$ is the torsion of $\nabla$, i.e. $\tau(X,Y)=\nabla_XY-\nabla_YX-[X,Y]$ for $X,Y\in{\Gamma(TM)}$.
  \end{proposition}
This connection is called the {\it TWT connection}. The $(1,2)$-tensor field $Q$ defined by
  \begin{equation}\label{eq2.8}
  Q(X,Y):=(\nabla_YJ)X,{\quad}X,Y\in{\Gamma(TM)},
  \end{equation}
  is called the {\it Tanno tensor} (cf. (10) in \cite{BD1}). Tanno proved that a contact Riemannian manifold  is a CR manifold if and only if $Q\equiv0$ (cf.  Proposition 2.1 in \cite{T1}).
  The curvature tensor of TWT connection is $R(X,Y)=\nabla_X\nabla_Y-\nabla_Y\nabla_X-\nabla_{[X,Y]}.$
The Ricci tensor of the TWT connection is defined by $Ric(Y,Z)=tr\{X\longrightarrow{R(X,Z)Y}\},$
for any $X,Y,Z\in{TM}$. The scalar curvature is $R=tr(Ric)$.

  We extend $h$, $J$ and $\nabla$ to the complexified tangent bundle by $\mathbb{C}$-linear extension:
  \begin{equation}\label{eq2.8a}
  \begin{aligned}
  h(X_1+iY_1,X_2+iY_2)&:=h(X_1,X_2)-h(Y_1,Y_2)+i\big(h(X_1,Y_2)+h(X_2,Y_1)\big),\\
  J(X_1+iY_1)&:=JX_1+iJY_1,\\
  \nabla_{(X_1+iY_1)}(X_2+iY_2)&:=\nabla_{X_1}X_2-\nabla_{Y_1}Y_2+i\big(\nabla_{X_1}Y_2+\nabla_{Y_1}X_2\big),
  \end{aligned}
  \end{equation}
  for any $Z_j=X_j+iY_j\in{\mathbb{C}TM}$, $j=1,2$.
  \begin{corollary}\label{prop2.0f}
  The Riemannian metric $h$, the complex structure $J$, the TWT connection, the torsion and   curvature tensors are preserved under the complex conjugation, i.e.,
  $$\overline{h(Z_1,Z_2)}=h(\overline{Z_1},\overline{Z_2}),{\quad}J\overline{Z_1}=\overline{JZ_1},{\quad}\overline{\nabla_{Z_1}Z_2}=\nabla_{\overline{Z_1}}\overline{Z_2},$$
  $$\overline{\tau(Z_1,Z_2)}=\tau(\overline{Z_1},\overline{Z_2}),{\quad}\overline{R(Z_1,Z_2)Z_3}=R(\overline{Z_1},\overline{Z_2})\overline{Z_3},$$
  for any $Z_1,Z_2,Z_3\in{\mathbb{C}TM}$.
  \end{corollary}
  \begin{proof}
  For $Z_1,Z_2,Z_3\in{\mathbb{C}TM}$, $h$, $J$ and $\nabla$ are preserved under the complex conjugation follows from the definition of the extension \eqref{eq2.8a}.

  It's apparent that $\overline{[Z_1,Z_2]}=[\overline{Z_1},\overline{Z_2}]$. Since we already have $\nabla$ are preserved under the complex conjugation, so by the definitions of $\tau$ and $R$, $\tau(Z_1,Z_2)=\nabla_{Z_1}Z_2-\nabla_{Z_2}Z_1-[Z_1,Z_2]$ and $R=\nabla_{Z_1}\nabla_{Z_2}Z_3-\nabla_{Z_2}\nabla_{Z_1}Z_3-\nabla_{[Z_1,Z_2]}Z_3$ are also preserved under the complex conjugation.
  \end{proof}
\subsection{The structure equations}
  \begin{notation}
  In this paper, from now on, we adopt the following index conventions:
  \begin{equation*}
  \begin{split}
  &a,b,c,d,e,\cdots\in\{1,2,\cdots,2n\},\\
  &j,k,l,r,s,\cdots\in\{0,1,\cdots,2n\},\\
  &\alpha,\beta,\gamma,\rho,\lambda,\mu,\cdots\in\{1,\cdots,n\},\\
  &\bar{\alpha}=\alpha+n.
  \end{split}
  \end{equation*}
  The order of index $j$ is defined to be $o(j)=2$ if $j=0$, and $o(j)=1$ otherwise. For a multi-index $J=(j_1,\cdots,j_s)$, we denote $\sharp{J}=s$, $o(J)=o(j_1)+\cdots+o(j_s)$, $x^J=x^{j_1}\cdots{x^{j_s}}$, $Z_J=Z_{j_s}\cdots{Z_{j_1}}$, and $\partial_J=\partial^s/\partial{x}^{j_s}\cdots{\partial{x}^{j_1}}$.
  \end{notation}

  In this subsection, we consider the structure equations with respect to a general frame $\{W_j\}$, where $\{W_a\}$ are horizontal and $W_0=T$ is the Reeb vector field. Let $U_q$ be a neighborhood of a point $q$ where this frame is defined. It's easy to see that $h(T,T)=\theta(T)=1$ and $h(W_a,T)=\theta(W_a)=0$ by \eqref{eq2.1}. In horizontal space, we set $h(W_a,W_b)=h_{ab}$ and using $h_{ab}$ and its inverse matrix to lower and raise indices. And the Einstein summation convention will be used.

Let $\{\theta^j\}$ be the coframe dual to $\{W_j\}$.
  Write $\nabla{W_j}=\omega_{j}^{k}\otimes{W_k}$, with the TWT-connection 1-forms $\omega_{j}^{k}=\Gamma_{ij}^k\theta^i$. For the almost complex structure $J$, we write $J=J_{\ k}^l\theta^k\otimes{W_l}$ or equivalently $JW_k=J_{\ k}^{l}W_l$.
  \begin{proposition}\label{prop2.1}
  \begin{equation}\label{eq2.0}
  \begin{aligned}
  \omega_0^k=0,{\quad}\omega_j^0=0,{\quad}\Gamma_{i0}^k=0,{\quad}\Gamma_{ij}^0=0,\\
  J_{\ 0}^{k}=0,{\quad}J_{\ j}^{0}=0,{\quad}J_{ab}=-J_{ba}.
  \end{aligned}
  \end{equation}
  \end{proposition}
  \begin{proof}
  $\omega_0^k=0$ follows from $\nabla{T}=0$. And by $\theta(\nabla{X})=0$ for any $X\in{HM}$, we have $\omega_a^0=0$. $\omega_j^0=0$ follows. $\Gamma_{i0}^k=0$ and $\Gamma_{ij}^0=0$ follows from $\omega_0^k=0$ and $\omega_j^0=0$, respectively.

Note that \eqref{eq2.1} implies some useful relations (cf. p. 351 in \cite{T1}),
 \begin{equation}\label{eq2.3a}
 \begin{aligned}
 &JT=0,{\quad}\theta(JX)=0,\\
 &h(X,Y)=h(JX,JY)+\theta(X)\theta(Y),{\quad}d\theta(X,JY)=-d\theta(JX,Y),
 \end{aligned}
 \end{equation}
  for any $X,Y\in{TM}$. $JT=0$ implies $J_{\ 0}^k=0$, and $\theta(JW_j)=0$ in \eqref{eq2.3a} implies $J_{\ j}^0=0$.
 Since
 \begin{equation}\label{eq2.16b}
 h(W_a,JW_b)=h(W_a,J_{\ b}^cW_c)=h_{ac}J_{\ b}^c=J_{ab},
 \end{equation}
 holds by $h(X,JY)=d\theta(X,Y)=-d\theta(Y,X)=-h(Y,JX),$
 for any $X,Y\in{TM}$, we get $J_{ab}=h(W_a,JW_b)=-h(W_b,JW_a)=-J_{ba}$.
 \end{proof}
The {\it Webster torsion} is defined by
$$\tau_{\ast}(X)=\tau(T,X),{\quad}X\in{TM},$$
(cf. p. 279 in \cite{BD1}). We have the following lemma for the Webster torsion.
\begin{lemma}\label{lem2.1}
Let $(M,\theta,h,J)$ be a contact Riemannian manifold and $T$ be the Reeb vector. Then:

{\text (1)} (cf. Lemma 1 in \cite{BD1}) $\text{(a)}\ \tau_{\ast}(T)=0$, $\text{(b)}\ \tau_{\ast}{\circ}J+J{\circ}\tau_{\ast}=0$, $\ \text{(c)}\ \tau_{\ast}TM\subset{HM}$, $\ \ \ $
$\text{(d)}\ \tau_{\ast}T^{(1,0)}M{\subset}T^{(0,1)}M$, $\tau_{\ast}T^{(0,1)}M{\subset}T^{(1,0)}M$.

{\text (2)} (cf. Lemma 3 in \cite{BD1}) The Webster torsion $\tau_{\ast}$ is self-adjoint, i.e. $h(\tau_{\ast}X,Y)=h(X,\tau_{\ast}Y)$ for any $X,Y\in{TM}$.
\end{lemma}\label{lem2.1a}
By (c) in Lemma \ref{lem2.1} (1), we can write $\tau_{\ast}(W_{a})=A_{a}^{b}W_{b}$. And we define $\tau^a:=A_b^a\theta^b$.
We also write
$R(W_k,W_l)W_j=\nabla_{W_k}\nabla_{W_l}W_j-\nabla_{W_l}\nabla_{W_k}W_j-\nabla_{[W_k,W_l]}W_j=R_{j\ kl}^{\ s}W_s,$
for the components of the curvature tensor.

Recall that we have the following identities for exterior derivatives
\begin{equation}\label{eq2.15a}
\begin{aligned}
&\phi\wedge\psi(X,Y)=\frac{1}{2}\bigg(\phi(X)\psi(Y)-\psi(X)\phi(Y)\bigg),\\
&X{\lrcorner}(\phi\wedge\psi)=2(\phi\wedge\psi)(X,\cdot)=\phi(X)\psi-\psi(X)\phi,\\
&2(d\phi)(X,Y)=X(\phi(Y))-Y(\phi(X))-\phi([X,Y])=(\nabla_X\phi)Y-(\nabla_Y\phi)X+\phi(\tau(X,Y)),
\end{aligned}
\end{equation}
where $\phi$ and $\psi$ is any $1$-form. The Lie derivation of a differential form $\phi$ is given by
\begin{equation}\label{eq2.15b}
\mathscr{L}_X\phi=X\lrcorner{d\phi}+d(X\lrcorner\phi).
\end{equation}
Note that here we use the definition of the exterior derivative with a factor $\frac{1}{2}$. The reason we use this definition is that the $n$-form defined is this way has the property that $dx^1{\wedge}\cdots{\wedge}dx^n$ equals to the Lebesgue measure on $\mathbb{R}^n$. We may refer to Section 4 in \cite{BG1} for these identities.
\begin{proposition}\label{prop2.0e}
With $J_{ab}$, $A_{ab}$, $R_{a\ cd}^{\ b}$ defined as above, we have the following structure equations.
\begin{equation}\label{SE}
\begin{aligned}
&d\theta=J_{\alpha\beta}\theta^{\alpha}\wedge\theta^{\beta}+2J_{\alpha\bar{\beta}}\theta^{\alpha}\wedge\theta^{\bar{\beta}}
+J_{\bar{\alpha}\bar{\beta}}\theta^{\bar{\alpha}}\wedge\theta^{\bar{\beta}},\\
&d\theta^a=\theta^b\wedge{\omega_b^a}+\theta\wedge{\tau^a}=\theta^b\wedge{\omega_b^a}+A_b^a\theta\wedge\theta^b,\\
&d\omega_a^b-\omega_a^c\wedge\omega_c^b=R_{a\ \lambda\bar{\mu}}^{\ b}\theta^{\lambda}\wedge\theta^{\bar{\mu}}+\frac{1}{2}R_{a\ \lambda\mu}^{\ b}\theta^{\lambda}\wedge\theta^{\mu}+\frac{1}{2}R_{a\ \bar{\lambda}\bar{\mu}}^{\ b}\theta^{\bar{\lambda}}\wedge\theta^{{\bar\mu}}+R_{a\ 0\bar{\mu}}^{\ b}\theta\wedge\theta^{\bar{\mu}}-R_{a\ \lambda0}^{\ b}\theta\wedge\theta^{\lambda},\\
&R(X,Y)W_a=2\big(d\omega_a^b-\omega_a^c\wedge\omega_c^b\big)(X,Y)W_b.
\end{aligned}
\end{equation}
\end{proposition}
\begin{proof}
By \eqref{eq2.1} and \eqref{eq2.16b}, we have $d\theta(W_{\alpha},W_{\beta})=h(W_{\alpha},JW_{\beta})=J_{\alpha\beta}$,
$d\theta(W_{\alpha},W_{\bar{\beta}})=J_{\alpha\bar{\beta}}$,
$d\theta(W_{\bar{\alpha}},W_{\bar{\beta}})=J_{\bar{\alpha}\bar{\beta}}$. We also have $d\theta(T,\cdot)\equiv0$. So the first identity in \eqref{SE} follows.

Substituting $\phi=\theta^{a}$, $X=W_c$ and $Y=W_d$ in \eqref{eq2.15a}, by
$(\nabla_X\phi)Y=X(\phi(Y))-\phi(\nabla_XY)$
for any $1$-forms $\phi$, we get
\begin{align}
\notag 2d\theta^a(W_c,W_d)&=(\nabla_{W_c}\theta^a)W_d-(\nabla_{W_d}\theta^a)W_c+\theta^a(\tau(W_c,W_d))=-\theta^a(\nabla_{W_c}W_d)+\theta^a(\nabla_{W_d}W_c)\\
\notag &=-\Gamma_{cd}^a+\Gamma_{dc}^a=2(\theta^b\wedge{\omega_b^a}+A_b^a\theta\wedge\theta^b)(W_c,W_d),
\end{align}
by \eqref{TWT} and \eqref{eq2.15a}. And similarly we get
\begin{align}
\notag 2d\theta^a(T,W_d)&=(\nabla_T\theta^a)W_d-(\nabla_{W_d}\theta^a)T+\theta^a(\tau(T,W_d))=-\theta^a({\nabla_TW_d})+A_d^a\\
\notag &=-\Gamma_{0d}^a+A_d^a=2\bigg(\theta^b\wedge{\omega_b^a}+A_b^a\theta\wedge\theta^b\bigg)(T,W_d).
\end{align}
So the second identity in \eqref{SE} holds.

For the fourth identity of \eqref{SE}, we have
\begin{align}
\notag R(X,Y)&W_a=\nabla_X\nabla_YW_a-\nabla_Y\nabla_XW_a-\nabla_{[X,Y]}W_a\\
\notag &=\nabla_X(\omega_a^b(Y)W_b)-\nabla_Y(\omega_a^b(X)W_b)-\omega_a^b([X,Y])W_b\\
\notag &=X(\omega_a^b(Y))W_b-Y(\omega_a^b(X))W_b-\omega_a^b([X,Y])W_b+\omega_a^b(Y)\omega_b^c(X)W_c-\omega_a^b(X)\omega_b^c(Y)W_c\\
\notag &=2\big(d\omega_a^b-\omega_a^c\wedge\omega_c^b\big)(X,Y)W_b,
\end{align}
by the definition of curvatures and \eqref{eq2.15a}. The third identity of \eqref{SE} follows by applying both sides to $X=W_j$, $Y=W_k$ in the fouth identity of \eqref{SE}.
\end{proof}
\begin{remark}
Note that the structure equations (13), (14) and (39) in \cite{BD1} are the special case of \eqref{SE} with respect to a $T^{(1,0)}M$-frame.
\end{remark}
Consequently, we have
\begin{align}\label{eq2.14}
\notag R_{a\ cd}^{\ b}&=2(d\omega_a^b)(W_c,W_d)-2\omega_a^e\wedge\omega_e^b(W_c,W_d)\\
\notag &=(\nabla_{W_c}\omega_a^b)(W_d)-(\nabla_{W_d}\omega_a^b)(W_c)+\omega_a^b(\tau(W_c,W_d))-\Gamma_{ca}^e\Gamma_{de}^b+\Gamma_{da}^e\Gamma_{ce}^b\\
\notag &=W_c\Gamma_{da}^b-W_d\Gamma_{ca}^b-\omega_a^b(\nabla_{W_c}W_d)+\omega_a^b(\nabla_{W_d}W_c)+\omega_a^b(2h(W_c,JW_d)T)-\Gamma_{ca}^e\Gamma_{de}^b+\Gamma_{da}^e\Gamma_{ce}^b\\
&=W_c\Gamma_{da}^b-W_d\Gamma_{ca}^b-\Gamma_{cd}^e\Gamma_{ea}^b+\Gamma_{dc}^e\Gamma_{ea}^b-\Gamma_{ca}^e\Gamma_{de}^b+\Gamma_{da}^e\Gamma_{ce}^b+2\Gamma_{0a}^bJ_{cd},
\end{align}
by \eqref{TWT}, \eqref{eq2.15a} and the fourth identity in \eqref{SE}.
\subsection{The special frame and the normal coordinates}
$\mathscr{H}^n=\mathbb{C}^n\times{\mathbb{R}}$ with coordinates $x=(z,t)$ has the structure of the Heisenberg group. The Heisenberg norm is $|x|=(|z|^4+t^2)^{1/4}$. Here we choose the contact form $\Theta=dt-iz^{\alpha}dz^{\bar{\alpha}}+iz^{\bar{\alpha}}dz^{\alpha}$ on $\mathscr{H}^n$ and set $\Theta^{\alpha}=dz^{\alpha}$. Their dual are
$${\quad}Z_0=\frac{\partial}{\partial{t}},\qquad
Z_{\alpha}=\frac{\partial}{\partial{z}^{\alpha}}-iz^{\bar{\alpha}}\frac{\partial}{\partial{t}}
.$$

On $\mathscr{H}^n$, the orbit of the the parabolic dilation is a parabola through $0\in\mathscr{H}^n$. Recall that in the Riemannian geometry, the classical exponential map sends radial lines in the tangent space to geodesics. Similarly, in the CR geometry, Jerison and Lee \cite{JL1} defined the parabolic exponential map, which sends a parabola in the tangent space to a parabolic geodesic. In a contact Riemannian manifold, the parabolic exponential map can be defined in the same way as in the CR case. A smooth curve $\gamma(s)$ in a contact Riemannian manifold $M$ is a {\it parabolic geodesic} if it satisfies ODE:
\begin{equation}\label{eq2.9}
\nabla_{\dot{\gamma}}\dot{\gamma}=2cT,
\end{equation}
for some $c\in\mathbb{R}$, where $\nabla$ is the TWT connection and $T$ is the Reeb vector field. We have the following proposition.
  \begin{proposition}\label{prop2.0b}
  Let $(M,\theta,h,J)$ be a contact Riemannian manifold and $q\in{M}$. For any $W\in{H_qM}$ and $c\in\mathbb{R}$, let $\gamma=\gamma_{W,c}$ denote the solution to the ODE \eqref{eq2.9} with initial conditions $\gamma(0)=q$ and $\dot{\gamma}(0)=W$. We call $\gamma$ the parabolic geodesic determined by $W$ and $c$. Define the parabolic exponential map $\Psi:T_qM\to{M}$ by \[\Psi(W+cT)=\gamma_{W,c}(1).\]
  Then $\Psi$ maps a neighborhood of $0$ in $T_qM$ diffeomorphically to a neighborhood of $q$ in $M$, and sends $sW+s^2cT$ to $\gamma_{W,c}(s)$.
  \end{proposition}
  \begin{proof}
  The proof is the same as that in the CR case (Theorem 2.1 in \cite{JL1}) since the integrability of $J$ is not used. Choosing a coordinate $\{x^i\}$ centered at $q$, we let $\Gamma_{ij}^k$ denote the Christoffel symbols of the TWT connection in these coordinates, i.e. $\nabla_{\frac{\partial}{\partial{x}^i}}{\frac{\partial}{\partial{x}^j}}=\Gamma_{ij}^k{\frac{\partial}{\partial{x}^k}}$.
  ODE \eqref{eq2.9} can be written as
  \begin{equation}\label{eq2.9a}
  \ddot{\gamma}^k(s)=-\Gamma_{ij}^k(\gamma(s))\dot\gamma^i(s)\dot\gamma^j(s)+2cT^k(\gamma(s)),
  \end{equation}
  where $\dot{\gamma}^k(s)=\frac{dx^k}{ds}(\gamma(s))$, $\ddot{\gamma}^k(s)=\frac{d^2x^k}{ds^2}(\gamma(s))$ and $T^k=dx^k(T)$ in these coordinates. This proposition follows from the uniqueness of the solution to this ODEs and smooth dependence of the solution on the parameters.
  \end{proof}
  A vector field $X\in{M}$ is called {\it parallel} along a curve $\gamma(s)$ if it satisfies
$$\nabla_{\dot{\gamma}}X=0.$$
  \begin{proposition}\label{prop2.0c}
  Suppose $X$ is a vector field defined in a neighborhood of $q$ in $M$ which is parallel along each curve $\gamma_{W,c}$. Then $X$ is smooth near $q$.
  \end{proposition}
This proposition can be proved in the same way as Lemma 2.2 in \cite{JL1} since the integrability of $J$ is not used. Choosing coordinates $\{x^i\}$ centered at $q$, we can write $X=X^j\frac{\partial}{\partial{x}^j}$ for some functions $X_j$. For each curve $\gamma_{W,c}(s)$, we write $\xi^j(s,W,c)=X^j(\gamma_{W,c}(s))$. Then the differential equation $\nabla_{\dot\gamma(s)}X=0$ becomes:
  $$\frac{\partial}{\partial{s}}\xi^j(s,W,c)=-\Gamma_{kl}^j(\gamma_{W,c}(s))\dot\gamma_{W,c}^k(s)\xi^l(s,W,c),$$
  with initial condition $\xi^j(0,W,c)=X^j(0)$. $X$ is smooth since the solutions to this ODEs depend smoothly on parameters.

  As introduced before, the complexification of the tangent space $\mathbb{C}T_qM$ at point $q$ has a unique subbuddle $T_q^{(1,0)}M$ such that $JX=iX$ for any $X\in{T_q^{(1,0)}M}$. Set $T_q^{(0,1)}M=\overline{T_q^{(1,0)}M}$, and for any $X\in{T_q^{(0,1)}M}$, $JX=-iX$. Furthermore, we choose an orthonormal basis of the horizontal space with respect to the metric $h$ at point $q$ as in the following lemma.
\begin{lemma}\label{lem2.1b}
  We can choose $W_{\alpha;q}\in{T_q^{(1,0)}M}$ and $W_{\bar{\alpha};q}:=\overline{W_{\alpha;q}}\in{T_q^{(0,1)}M}$, such that
  $$h(W_{\alpha;q},W_{\bar{\beta};q})=\delta_{\alpha\bar{\beta}},{\quad}h(W_{\alpha;q},W_{\beta;q})=0.$$
  \end{lemma}
  \begin{proof}
  Choose a real vector $\{X_1\}$ on $H_qM$ such that $h(X_1,X_1)=2$ and set $X_{n+1}:=JX_1$. Then by
  $h(X_1,X_{n+1})=h(X_1,JX_1)=d\theta(X_1,X_1)=0$
  and $h(X_{n+1},X_{n+1})=h(JX_1,JX_1)=h(X_1,X_1)=2$,
  $X_{n+1}$ is orthogonal to $X_1$. We can choose $X_2$ orthogonal to $span\{X_1,JX_1\}$, and define $X_{n+2}:=JX_2$. Repeating the procedure, we can choose an orthogonal basis $X_1,\cdots,X_{2n}$ with $h(X_a,X_b)=2\delta_{ab}$ and $JX_{\alpha}=X_{\alpha+n}$.
  Now define
  \begin{equation}\label{eq2.16a}
  W_{\alpha;q}:=\frac{1}{2}(X_{\alpha}-iX_{\alpha+n})=\frac{1}{2}(X_{\alpha}-iJX_{\alpha}),
  {\quad}W_{\bar{\alpha};q}=\overline{W_{\alpha;q}}.
  \end{equation}
  We see that $W_{\alpha;q}\in{T_q^{(1,0)}M}$ and $W_{\bar{\alpha};q}\in{T_q^{(0,1)}M}$. Then by \eqref{eq2.16a} and $\mathbb{C}$-linear extension \eqref{eq2.8a}, we have $h(W_{\alpha;q},W_{\bar{\beta};q})=\frac{1}{4}h(X_{\alpha}-iX_{\alpha+n},X_{\beta}+iX_{\beta+n})
  =\frac{1}{4}(h(X_{\alpha},X_{\beta})+h(X_{\alpha+n},X_{\beta+n}))=\delta_{\alpha\bar{\beta}}$ and $h(W_{\alpha;q},W_{\beta;q})=0$.
  \end{proof}
  We extend $\{W_{\alpha;q}\}$ by parallel translation along each parabola $\gamma_{W,c}$, i.e. $\nabla_{\dot{\gamma}}W_{\alpha}=0$. Let $W_{\bar{\alpha}}=\overline{W_{\alpha}}$, so $W_{\bar{\alpha}}$ is also parallel along $\gamma_{W,c}$. $T$ is automatically parallel along each curve $\gamma_{W,c}$ by $\nabla{T}=0$ in \eqref{TWT}.
  Since every point in some punctured neighborhood $U$ near $q$ is on a unique $\gamma_{W,c}$,  the frame $\{W_{\alpha},W_{\bar{\alpha}},T\}$ is well-defined and smooth near $q$ by Proposition \ref{prop2.0c}. We call such a frame a {\it special frame}.

  Let $\{\theta^{\beta},\theta^{\bar{\beta}},\theta\}$ denote the coframe dual to $\{W_{\alpha},W_{\bar{\alpha}},T\}$, i.e.,
  $\theta^{\beta}(W_{\alpha})=\delta_{\alpha}^{\beta}$, $\theta^{\beta}(W_{\bar{\alpha}})=\theta^{\beta}(T)=0$, and
  $\theta(W_{\alpha})=\theta(W_{\bar{\alpha}})=0$, $\theta(T)=1$. From now on we denote
  $$W_0:=T,{\quad}\theta^0:=\theta.$$
Since $\nabla{T}=0$, we have $\nabla_{\dot\gamma}W_k=0$. So $0=\nabla_{\dot\gamma}(\theta^j(W_k))=(\nabla_{\dot\gamma}\theta^j)(W_k)+\theta^j(\nabla_{\dot\gamma}W_k)=(\nabla_{\dot\gamma}\theta^j)(W_k)$ holds for each geodesic $\gamma(s)$. Namely, $\nabla_{\dot\gamma}\theta^j=0$ along each $\gamma$, so $\{\theta^{\alpha},\theta^{\bar{\alpha}},\theta\}$ is also parallel along each $\gamma$. We call such a coframe a {\it special coframe}.  Define an isomorphism $\iota:T_qM\to{\mathscr{H}^n}$ by $\iota(V)=(\theta^{\alpha}(V),\theta^{\bar{\alpha}}(V),\theta(V))=(z^{\alpha},z^{\bar{\alpha}},t)$, which determines a coordinate chart $\iota\circ{\Psi^{-1}}$ in a neighborhood of $q$.
  We call this chart {\it the normal coordinates} determined by $\{W_{\alpha},W_{\bar{\alpha}},T\}$.
\begin{remark}
(1) In the  CR case, Jerison and Lee chose a $T^{(1,0)}M$-frame at $q$ with norm $h(W_{\alpha;q},W_{\bar{\beta};q})=2\delta_{\alpha\bar{\beta}}$, $h(W_{\alpha;q},W_{\beta;q})=0$ to construct a special frame. This is because they used the structure equation $d\theta=ih_{\alpha\bar{\beta}}\theta^{\alpha}\wedge\theta^{\bar{\beta}}$ (cf. p. 307 in \cite{JL1}). But here in contact Riemannian case, we follow \cite{BD1} to use the structure equation $d\theta=-2ih_{\alpha\bar{\beta}}\theta^{\alpha}\wedge\theta^{\bar{\beta}}$ (cf. (13) in \cite{BD1}) for a $T^{(1,0)}M$-frame at $q$. That's why we choose a $T^{(1,0)}M$-frame at $q$ with norm as Lemma \ref{lem2.1b} to construct a special frame.

(2) By Corollary \ref{prop2.0f}, since $W_{\bar{\alpha}}=\overline{W_{\alpha}}$ holds for our special frame, the complex conjugation can be reflected in the indices of the components of $\omega_a^b$, $h_{ab}$, $J_{\ a}^b$, $A_{ab}$, $R_{abcd}$ and their covariant derivations, e.g.,
$$\overline{\omega_{\alpha}^{\bar{\beta}}}=\omega_{\bar{\alpha}}^{\beta},{\quad}\overline{J_{\ \alpha}^{\beta}}=J_{\ \bar{\alpha}}^{\bar{\beta}},{\quad}\overline{h_{\alpha\bar{\beta}}}=h_{\bar{\alpha}\beta}.$$
\end{remark}
\begin{proposition}\label{prop2.0d}
  A special frame $\{W_j\}$ is parallel along each parabolic geodesic and satisfies
  \begin{equation}\label{eq2.18}
  \begin{aligned}
  &h_{\alpha\bar{\beta}}=\delta_{\alpha\bar{\beta}},{\quad}h_{\alpha\beta}=0,\\
  &J_{\ \alpha}^{\beta}(q)=i\delta_{\alpha}^{\beta},{\quad}J_{\ \alpha}^{\bar\beta}(q)=0,\\
  &J_{\alpha\bar{\beta}}(q)=-i\delta_{\alpha\bar{\beta}},{\quad}J_{\alpha\beta}(q)=0,{\quad}J_{\ \alpha}^{\bar{\beta}}=-J_{\ \beta}^{\bar{\alpha}}.
  \end{aligned}
  \end{equation}
  \end{proposition}
  \begin{proof}
  By $\nabla{h}=0$, we see that
  $$\frac{d}{ds}\big(h_{ab}(\gamma(s))\big)=h(\nabla_{\dot{\gamma}}W_a,W_b)+h(W_a,\nabla_{\dot{\gamma}}W_b)=0,$$
  along each $\gamma$. So $h_{\alpha\bar{\beta}}\equiv\delta_{\alpha\bar{\beta}}$ and $h_{\alpha\beta}\equiv0$ hold near $q$.

  $J_{\ \alpha}^{\beta}(q)=i\delta_{\alpha}^{\beta}$ and $J_{\ \alpha}^{\bar\beta}(q)=0$ follows from Lemma \ref{lem2.1b} by our choice of the special frame at $q$. Then by \eqref{eq2.16b}, $J_{\alpha\bar{\beta}}(q)=h_{\alpha\bar{\gamma}}J_{\ \bar{\beta}}^{\bar{\gamma}}(q)=-i\delta_{\alpha\bar{\beta}}$ and $J_{\alpha\beta}(q)=h_{\alpha\bar{\gamma}}J_{\ \beta}^{\bar{\gamma}}(q)=0$ hold. For the last identity in \eqref{eq2.18}, we have $J_{\ \beta}^{\bar{\alpha}}=h_{\alpha\bar{\gamma}}J_{\ \beta}^{\bar{\gamma}}=J_{\alpha\beta}=-J_{\beta\alpha}=-h_{\beta\bar{\gamma}}J_{\ \alpha}^{\bar{\gamma}}=-J_{\ \alpha}^{\bar{\beta}}$ by $h_{\alpha\bar{\gamma}}=\delta_{\alpha\bar{\gamma}}$ and the anti-symmetry of $J_{ab}$ in Proposition \ref{prop2.1}.
  \end{proof}
\begin{remark}
Recall that when $(M,\theta,h,J)$ is a CR manifold, $Q=\nabla{J}=0$, and so $J$ is also parallel along each parabolic geodesic. Hence $(J_{\ a}^b)$ is a constant matrix near $q$ (see Proposition 2.3 in \cite{JL1}). But on a contact Riemannian manifold $(M,\theta,h,J)$, $Q=\nabla{J}$ may not vanish. So $J_{\ a}^b$ may not be   constant   near $q$. We only know that $J_{\ \alpha}^{\beta}(q)=i\delta_{\alpha}^{\beta}$ and $J_{\ \alpha}^{\bar{\beta}}(q)=0$ at the point $q$.
\end{remark}
The following corollary follows from Lemma \ref{lem2.1}.
\begin{corollary}\label{cor2.0a}
With respect to a special frame, we have
$$A_{\alpha}^{\beta}(q)=0,{\quad}A_{\alpha\bar{\beta}}(q)=0,{\quad}A_{ab}=A_{ba}.$$
\end{corollary}

The parabolic dilations in this coordinate $(z,t)$ is $\delta_s(z,t)=(sz,s^2t)$ for $s>0$, the generator of the parabolic dilation is the vector field
  \begin{equation}\label{eq2.15c}
  P_{(z,t)}=z^{\alpha}\frac{\partial}{\partial{z}^{\alpha}}+z^{\bar{\alpha}}\frac{\partial}{\partial{z}^{\bar{\alpha}}}
  +2t\frac{\partial}{\partial{t}}.
  \end{equation}
  A tensor field $\varphi$ is called {\it homogeneous of degree $m$} if $\mathscr{L}_P\varphi=m\varphi$. For any tensor $\varphi$, we denote $\varphi_{(m)}$ as the part of its Taylor's series that is homogeneous of degree $m$ in terms of the parabolic dilations. So $\mathscr{L}_P\varphi_{(m)}=m\varphi_{(m)}$. If $\varphi$ is a differential form,
  \begin{equation}\label{eq2.12}
  \varphi_{(m)}=\frac{1}{m}(P\lrcorner{d\varphi}+d(P\lrcorner{\varphi}))_{(m)},
  \end{equation}
  by Cartan's formula \eqref{eq2.15b}. For example, $z^a$, $dz^a$ is homogeneous of degree $1$. $z^0=t$, $dt$ and $\Theta$ is homogeneous of degree $2$. $\frac{\partial}{\partial{z^a}}$ and $\frac{\partial}{\partial{t}}$ is homogeneous of degree $-1$ and $-2$, respectively. With respect to a normal coordinate, we define the vector fields by
 \begin{equation}\label{eq2.10}
  Z_{\alpha}=\frac{\partial}{\partial{z^{\alpha}}}-iz^{\bar{\alpha}}\frac{\partial}{\partial{t}},{\quad}\alpha=1,2,\cdots,n,
  {\quad}Z_0=\frac{\partial}{\partial{t}},
  \end{equation}
  and their dual
  \begin{equation}\label{eq2.11}
  \Theta^{\alpha}=dz^{\alpha},\quad\Theta^{\bar{\alpha}}=dz^{\bar{\alpha}},\quad
  \Theta=dt-iz^{\alpha}dz^{\bar{\alpha}}+iz^{\bar{\alpha}}dz^{\alpha}.
  \end{equation}
Hence $Z_{\alpha}$ and $Z_0$ are homogeneous of degree $-1$ and degree $-2$, respectively. Namely $Z_j$ is homogenous of degree $-o(j)$.
\begin{remark}
 In this paper, if indices $\alpha$ and $\bar{{\alpha}}$ both appear in low (or upper) indices, then the index $\alpha$ will be taken summation, e.g.
 $$\Theta=dt-iz^{\alpha}dz^{\bar{\alpha}}+iz^{\bar{\alpha}}dz^{\alpha}=dt-i\sum\limits_{\alpha}z^{\alpha}dz^{\bar{\alpha}}+i\sum\limits_{\alpha}z^{\bar{\alpha}}dz^{\alpha}.$$
\end{remark}
  \begin{theorem}
  On a contact Riemannian manifold $(M,\theta,h,J)$, suppose $F$ is a smooth function defined near $q$. Then with respect to the normal coordinates, for any $m$, we have
  \begin{equation}\label{eq2.12a}
  F_{(m)}(x)=\sum\limits_{o(K)=m}\frac{1}{(\sharp{K})!}z^KZ_KF(q).
  \end{equation}
  The notations of the multi-index are defined as in Notation 2.1.
  \end{theorem}
  This theorem can be proved exactly in the same way as Lemma 3.10 in \cite{JL1} since the integrability of $J$ is not used here.
\subsection{Homogeneous parts of the special frame (coframe) and the connection coefficients}
As in the CR case in \cite{JL1}, there exists a simple relation between the Euler vector field $P$ and the special coframe.
\begin{lemma}\label{lem2.2}
With respect to the normal coordinates $(z^{\alpha},z^{\bar{\alpha}},t)$, we have
  \begin{equation}\label{eq2.17}
  \theta(P)=2t,\quad\theta^a(P)=z^a,\quad\omega_b^a(P)=0,
  \end{equation}
  where $P$ is the Euler vector field. Equivalently, $P=z^{\alpha}W_{\alpha}+z^{\bar{\alpha}}W_{\bar{\alpha}}+2tT$.
  \end{lemma}
  \begin{proof}
  This lemma can be proved in the same way as Lemma 2.4 in \cite{JL1} since the integrability of $J$ is not used. We mention it briefly. We need to show that \eqref{eq2.17} holds along each parabolic geodesic $\gamma_{W,c}$. Fix a vector $W+cT$ at $q$ with $W\in{H_qM}$ and $c\in\mathbb{R}$, we write $W=w^aW_{a;q}$. In these coordinates, the parabolic geodesic $\gamma=\gamma_{W,c}$ is given explicitly by
  $$(z^a,t)=\gamma(s)=(sw^a,s^2c),$$
  by Proposition \ref{prop2.0b}. Note that by the definition \eqref{eq2.15c}, $P_{\gamma(s)}=P_{(sw^a,s^2c)}= \sum_a sw^a\frac{\partial}{\partial{z^a}}+ 2s^2c\frac{\partial}{\partial{t}}$. Then by explicit computation $\dot{\gamma}(s)=\sum_a  w^a\frac{\partial}{\partial{z^a}}+ 2s c\frac{\partial}{\partial{t}}=s^{-1}P_{\gamma(s)}$ for $s\ne0$. Along $\gamma$, by $\nabla\theta=0$, we have
  $$\frac{d}{ds}\bigg(\theta\big(\dot{\gamma}(s)\big)\bigg)=\theta(\nabla_{\dot{\gamma}}\dot{\gamma}(s))=\theta(2cT)=2c,$$
  and so $\theta(\dot\gamma(0))=0$, $\theta(\dot\gamma(s))=2cs$. Then $\theta(P)=\theta(s\dot\gamma(s))=2s^2c=2t$.
  Similarly by using $\nabla_{\dot\gamma}\theta^a=0$,
  $$\frac{d}{ds}\theta^a(\dot\gamma(s))=\theta^a(\nabla_{\dot\gamma}\dot\gamma(s))=\theta^a(2cT)=0.$$
  Note that $\theta^a(\dot\gamma(0))=\theta^a(W)=w^a$. We get $\theta^a(\dot\gamma(s))=w^a$ all along $\gamma$. So $\theta^a(P)=\theta^a(s\dot\gamma(s))=sw^a=z^a$.

  For the last identity in \eqref{eq2.17}, note that we have $\omega_b^a(P)\theta^b=0$ by $\nabla_P\theta^a=\nabla_{s\dot{\gamma}(s)}\theta^a=0$ by $\theta^a$ being parallel along each parabolic geodesic. Since $\theta^b$ are independent, we get $\omega_b^a(P)=0$ for any $a$, $b$.
  \end{proof}
  Then by Lemma \ref{lem2.2}, we have the following proposition:
  \begin{proposition}\label{thm2.3}
  With respect to a special frame and under the normal coordinates defined as above, we have
  \begin{equation}\label{eq2.19}
  \begin{aligned}
  &\omega_{a(m)}^b=\frac{1}{m}{(R_{a\ cd}^{\ b}z^c\theta^d+2tR_{a\ 0d}^{\ b}\theta^d+R_{a\ c0}^{\ b}z^c\theta)}_{(m)},\\
  &\theta^{b}_{(m)}={\frac{1}{m}(z^a\omega_a^b+2tA_a^b\theta^a-A_a^bz^a\theta+dz^b)}_{(m)},\\
  &\theta_{(m)}={\frac{1}{m}(2J_{ab}z^a\theta^b+2dt)}_{(m)}.
  \end{aligned}
  \end{equation}
  \end{proposition}
  \begin{proof}
  We have
  \begin{align}
  \notag \omega_{a(m)}^b&=\frac{1}{m}(P\lrcorner{d\omega_a^b})_{(m)}=\frac{1}{m}\bigg(P\lrcorner\bigg(R_{a\ \lambda\bar{\mu}}^{\ b}\theta^{\lambda}\wedge\theta^{\bar{\mu}}+\frac{1}{2}R_{a\ \lambda\mu}^{\ b}\theta^{\lambda}\wedge\theta^{\mu}+\frac{1}{2}R_{a\ \bar{\lambda}\bar{\mu}}^{\ b}\theta^{\bar{\lambda}}\wedge\theta^{\bar{\mu}}\\
  \notag &\ \ \ +R_{a\ 0\bar{\mu}}^{\ b}\theta\wedge\theta^{\bar{\mu}}-R_{a\ \lambda0}^{\ b}\theta\wedge\theta^{\lambda}\bigg)\bigg)_{(m)}\\
  \notag &=\frac{1}{m}\bigg(R_{a\ \lambda\bar{\mu}}^{\ b}(z^{\lambda}\theta^{\bar{\mu}}-z^{\bar{\mu}}\theta^{\lambda})+\frac{1}{2}R_{a\ \lambda\mu}^{\ b}(z^{\lambda}\theta^{\mu}-z^{\mu}\theta^{\lambda})\\
  \notag &\ \ \ +\frac{1}{2}R_{a\ \bar{\lambda}\bar{\mu}}^{\ b}(z^{\bar{\lambda}}\theta^{\bar{\mu}}-z^{\bar{\mu}}\theta^{\bar{\lambda}})+R_{a\ 0\bar{\mu}}^{\ b}(2t\theta^{\bar{\mu}}-z^{\bar{\mu}}\theta)-R_{a\ \lambda0}^{\ b}(2t\theta^{\lambda}-z^{\lambda}\theta)\bigg)_{(m)}\\
  \notag &=\frac{1}{m}{(R_{a\ cd}^{\ b}z^c\theta^d+2tR_{a\ 0d}^{\ b}\theta^d+R_{a\ c0}^{\ b}z^c\theta)}_{(m)}.
  \end{align}
 by \eqref{eq2.15a}, \eqref{SE}, \eqref{eq2.12} and $\omega_a^b(P)$=0 in \eqref{eq2.17}. Here we also use the relation $R_{a\ jk}^{\ b}=-R_{a\ kj}^{\ b}$ for the curvature tensor. Similarly we get
\begin{align}
  \notag \theta^b_{(m)}&=\frac{1}{m}\bigg(P\lrcorner{d\theta^b}+d(\theta^b(P))\bigg)_{(m)}
  =\frac{1}{m}\bigg(P\lrcorner\bigg(\theta^a\wedge\omega_a^b+A_a^b\theta\wedge\theta^a\bigg)+dz^b\bigg)_{(m)}\\
  \notag &={\frac{1}{m}\bigg(z^a\omega_a^b+2tA_a^b\theta^a-A_a^bz^a\theta+dz^b\bigg)}_{(m)},
\end{align}
by \eqref{eq2.15a}, \eqref{SE}, \eqref{eq2.12} and \eqref{eq2.17}.

Noting that $J_{ab}$ is anti-symmetric, we get
\begin{align}
\notag \theta_{(m)}&=\frac{1}{m}(P\lrcorner{d\theta}+d(\theta(P)))_{(m)}=\frac{1}{m}\bigg(P\lrcorner\big(J_{\alpha\beta}\theta^{\alpha}\wedge\theta^{\beta}+2J_{\alpha\bar{\beta}}\theta^{\alpha}\wedge\theta^{\bar{\beta}}+J_{\bar{\alpha}\bar{\beta}}\theta^{\bar{\alpha}}\wedge\theta^{\bar{\beta}}\big)+2dt\bigg)_{(m)}\\
\notag &=\frac{1}{m}\bigg(J_{\alpha\beta}\big(z^{\alpha}\theta^{\beta}-z^{\beta}\theta^{\alpha}\big)+2J_{\alpha\bar{\beta}}\big(z^{\alpha}{\theta}^{\bar{\beta}}-z^{\bar{\beta}}\theta^{\alpha}\big)+J_{\bar{\alpha}\bar{\beta}}\big(z^{\bar{\alpha}}\theta^{\bar{\beta}}-z^{\bar{\beta}}\theta^{\bar{\alpha}}\big)+2dt\bigg)_{(m)}\\
\notag &=\frac{1}{m}\bigg(2J_{\alpha\beta}z^{\alpha}\theta^{\beta}+2J_{\alpha\bar{\beta}}z^{\alpha}{\theta}^{\bar{\beta}}+2J_{\bar{\alpha}\beta}z^{\bar{\alpha}}\theta^{\beta}+2J_{\bar{\alpha}\bar{\beta}}z^{\bar{\alpha}}\theta^{\bar{\beta}}+2dt\bigg)_{(m)}={\frac{1}{m}(2J_{ab}z^a\theta^b+2dt)}_{(m)},
\end{align}
by \eqref{eq2.15a}, \eqref{SE}, \eqref{eq2.12} and \eqref{eq2.17}. Proposition \ref{thm2.3} is proved.
\end{proof}

Then we have the following corollary:
\begin{corollary}\label{cor2.1}
With respect to a special frame, we have
  \begin{equation}\label{eq2.22}
  \begin{aligned}
  &\omega_{a(1)}^b=0,\quad\omega_{a(2)}^b=\frac{1}{2}R_{a\ cd}^{\ b}(q)z^cdz^d,\\
  &\theta^b_{(1)}=dz^b,\quad\theta^b_{(2)}=0,\\
  &\theta^b_{(3)}=\frac{1}{6}R_{a\ cd}^{\ b}(q)z^az^cdz^d,{\quad}mod{\quad}\mathscr{A},\\
  &\theta_{(2)}=\Theta,\quad\theta_{(3)}=\frac{2}{3}J_{ab(1)}z^adz^b,\\
  &\theta_{(4)}=\frac{1}{12}J_{ab}(q)R_{e\ cd}^{\ b}(q)z^az^cz^edz^d+\frac{1}{2}J_{ab(2)}z^adz^b,{\quad}mod{\quad}\mathscr{A},
  \end{aligned}
  \end{equation}
  where $\mathscr{A}$ means terms linearly depending on $A_a^b(q)$.
  \end{corollary}
  \begin{proof}
  By the first identity in \eqref{eq2.19}, it's obvious that $\omega_{a(1)}^b=0$. Then it follows from the second identity in \eqref{eq2.19} and $\omega_{a(1)}^b=0$ that $\theta^b_{(1)}=dz^b$ and $\theta^b_{(2)}=0$.

By the third identity in \eqref{eq2.19} for $m=2$ and \eqref{eq2.18}, we get
\begin{align}
\notag \theta_{(2)}&=dt+J_{\alpha\bar{\beta}}(q)z^{\alpha}dz^{\bar{\beta}}+J_{\bar{\beta}\alpha}(q)z^{\bar{\beta}}dz^{\alpha}
=dt-i\delta_{\alpha\bar{\beta}}z^{\alpha}dz^{\bar{\beta}}+i\delta_{\alpha\bar{\beta}}z^{\bar{\beta}}dz^{\alpha}=\Theta.
\end{align}
By the third identity in \eqref{eq2.19} for $m=3$ and $\theta_{(2)}^b=0$, we get $\theta_{(3)}=\frac{2}{3}J_{ab(1)}z^adz^b$. By \eqref{eq2.19} for $m=2$ and $\theta^d_{(1)}=dz^d$, we find that $\omega_{a(2)}^b=\frac{1}{2}R_{a\ cd}^{\ b}(q)z^cdz^d$. Hence
\begin{equation}
\notag \theta_{(3)}^b=\frac{1}{3}z^a\omega_{a(2)}^b=\frac{1}{6}R_{a\ cd}^{\ b}(q)z^az^cdz^d,{\quad}{mod}{\quad}\mathscr{A},
\end{equation}
holds by \eqref{eq2.19}. And so we also have
\begin{align}
\notag \theta_{(4)}&=\frac{1}{4}(2J_{ab}(q)z^a\theta_{(3)}^b+2J_{ab(2)}z^adz^b)\\
\notag &=\frac{1}{12}J_{ab}(q)R_{e\ cd}^{\ b}(q)z^az^cz^edz^d+\frac{1}{2}J_{ab(2)}z^adz^b,\quad{mod}{\quad}\mathscr{A},
\end{align}
by \eqref{eq2.18} and \eqref{eq2.19}.
\end{proof}
\begin{remark}
(1) In the CR case, Jerison and Lee \cite{JL1} used the identity $d\theta(X,Y)=h(JX,Y)$. But in contact Riemannian case, people usually use $d\theta(X,Y)=h(X,JY)=-h(JX,Y)$ (cf. \cite{BD1} or \cite{T1}). So $J_{ab}$ is different from \cite{JL1} by a factor $-1$. That's why we choose $\Theta$ as \eqref{eq2.11}, which coincides with  standard contact form in \cite{JL1} up to  signs.

(2) We choose $dV_{\theta}=(-1)^n\theta\wedge(d\theta)^n$ as the volume form on $(M,\theta,h,J)$. And we will see later in Section 5 that  the volume form $dV=(-1)^n\Theta\wedge(d\Theta)^n$ on the Heisenberg group is positive (cf. \eqref{eq5.11}).

(3) Recall that in the CR case $\theta_{(3)}$   vanishes by the integrability of $J$ (cf. Proposition 2.5 in \cite{JL1}). While in the general case, $\theta_{(3)}$ may not vanish.
\end{remark}
\subsection{The asymptotic expansion of the special frame}
  We can also examine the Taylor series of $W_j$ in terms of $Z_j$'s under the normal coordinate $(z,t)$. We write
  \begin{equation}\label{eq2.27}
  W_j=s_j^kZ_k=s_j^{\alpha}Z_{\alpha}+s_j^{\bar{\alpha}}Z_{\bar{\alpha}}+s_j^0Z_0,
  \end{equation}
  for some functions $s_j^k$ (following the idea of p. 327 in [3]). As mentioned above, $Z_j$ is homogeneous of degree $-o(j)$, so we can examine the Taylor series of $W_j$ by the Taylor series of their coefficient functions $s_j^k$. We denote $\varphi\in\mathscr{O}_m$ if all the terms in the Taylor series of $\varphi$ in normal coordinates are homogeneous of degree $\ge{m}$. It's easy to see that if $\varphi\in\mathscr{O}_{m_1}$, $\psi\in\mathscr{O}_{m_2}$ then $\varphi\psi\in\mathscr{O}_{m+n}$.
  \begin{proposition}\label{prop2.2}
  With respect to the normal coordinates, for the functions $s_j^k$ defined as \eqref{eq2.27}, we have:
  \begin{equation}\label{eq2.29}
  \begin{aligned}
  &s_{\beta(0)}^{\alpha}=s_{\bar{\beta}(0)}^{\bar{\alpha}}=\delta_{\beta}^{\alpha},{\quad}s_{\beta(0)}^{\bar{\alpha}}=s_{\bar{\beta}(0)}^{\alpha}=0,
  {\quad}s_{b(1)}^a=0,\\
  &s_{b(2)}^a=-\frac{1}{6}R_{d\ cb}^{\ a}(q)z^cz^d,{\quad}mod{\quad}\mathscr{A},\\
  &s_{b(0)}^0=s_{b(1)}^0=0,{\quad}s_{b(2)}^0=\frac{2}{3}J_{ba(1)}z^a,\\
  &s_{b(3)}^0=\frac{i}{12}R_{d\ cb}^{\ \bar{\alpha}}(q)z^{\alpha}z^cz^d-\frac{i}{12}R_{d\ cb}^{\ \alpha}(q)z^{\bar{\alpha}}z^cz^d-\frac{1}{2}J_{{\alpha}b(2)}z^{\alpha}-\frac{1}{2}J_{\bar{\alpha}b(2)}z^{\bar{\alpha}},\ mod\ \mathscr{A},\\
  &s_{0(0)}^a=0,{\quad}s_{0(0)}^0=1,{\quad}s_{0(1)}^0=0.
  \end{aligned}
  \end{equation}
  \end{proposition}
  \begin{proof}
  By $\theta^a_{(1)}=dz^a$ and $\theta_{(2)}=\Theta$ in Corollary \ref{cor2.1},
  $\theta^a(W_b)=\delta_b^a$ and $\theta(W_b)=0$ lead to $W_{\beta(-1)}=Z_{\beta}$, $W_{\bar{\beta}(-1)}=Z_{\bar{\beta}}$. By $\theta^a(W_b)=\delta_b^a$ and $\theta_{(2)}^a=0$, we have
  \begin{align}
  \notag 0=\Bigg(\bigg(dz^a+\theta_{(2)}^a+\mathscr{O}_3)(Z_b+s_{b(1)}^cZ_c+s_{b(2)}^0Z_0+\mathscr{O}_1\bigg)\Bigg)_{(1)}
  =dz^a\big(s_{b(1)}^cZ_c\big)+\theta_{(2)}^a(Z_b)=s_{b(1)}^a.
  \end{align}
  Equivalently we can write $W_{b(0)}=s_{b(1)}^aZ_a+s_{b(2)}^0Z_0=0,\ mod\ Z_0.$
  And we have
  \begin{align}
  \notag 0&=\delta_{b(2)}^a=\Bigg(\bigg(dz^a+\theta_{(2)}^a+\theta_{(3)}^a+\mathscr{O}_4\bigg)
  \bigg(Z_b+W_{b(0)}+W_{b(1)}+\mathscr{O}_2\bigg)\Bigg)_{(2)}\\
  \notag &=\theta_{(3)}^a(Z_b)+dz^a(W_{b(1)})=\theta_{(3)}^a(Z_b)+s_{b(2)}^a,
  \end{align}
  and so
  \begin{equation}
  \notag s_{b(2)}^a=-\theta_{(3)}^a(Z_b)=-\frac{1}{6}R_{d\ cb}^{\ a}(q)z^cz^d,{\quad}mod{\quad}\mathscr{A}.
  \end{equation}
  By $\theta(W_b)=0$, we have
  $0=\Theta(W_{b(0)})+\theta_{(3)}(W_{b(-1)})=s_{b(2)}^0+\theta_{(3)}(Z_b)$, i.e. $s_{b(2)}^0=-\theta_{(3)}(Z_b)=-\frac{2}{3}J_{ab(1)}z^{a}=\frac{2}{3}J_{ba(1)}z^{a}$. Also $0=\theta_{(4)}(W_{b(-1)})+\theta_{(3)}(W_{b(0)})+\Theta(W_{b(1)})=\theta_{(4)}(Z_b)+\theta_{(3)}(s_{b(2)}^0Z_0)+\Theta(s_{b(2)}^aZ_a+s_{b(3)}^0Z_0)$, which gives
  \begin{align}
  \notag s_{b(3)}^0&=-\theta_{(4)}(Z_b)=-\frac{1}{12}J_{af}(q)R_{e\ cb}^{\ f}(q)z^az^cz^e-\frac{1}{2}J_{ab(2)}z^a,{\quad}mod{\quad}\mathscr{A}\\
  \notag &=-\frac{1}{12}J_{\alpha\bar{\rho}}(q)R_{d\ cb}^{\ \bar{\rho}}(q)z^{\alpha}z^cz^d-\frac{1}{12}J_{\bar{\alpha}\rho}(q)R_{d\ cb}^{\ \rho}(q)z^{\bar{\alpha}}z^cz^d-\frac{1}{2}J_{{\alpha}b(2)}z^{\alpha}-\frac{1}{2}J_{\bar{\alpha}b(2)}z^{\bar{\alpha}}\\
  \notag &=\frac{i}{12}R_{d\ cb}^{\ \bar{\alpha}}(q)z^{\alpha}z^cz^d-\frac{i}{12}R_{d\ cb}^{\ \alpha}(q)z^{\bar{\alpha}}z^cz^d-\frac{1}{2}J_{{\alpha}b(2)}z^{\alpha}-\frac{1}{2}J_{\bar{\alpha}b(2)}z^{\bar{\alpha}},{\quad}mod{\quad}\mathscr{A},
  \end{align}
  by \eqref{eq2.18} and \eqref{eq2.22}.

By $\theta(T)=1$, we get
$$1=\theta_{(2)}(W_{0(-2)})=\Theta\bigg(s_{0(0)}^0\frac{\partial}{\partial{t}}\bigg)=s_{0(0)}^0,$$
and $W_{0(-2)}=\frac{\partial}{\partial{t}}$. By the fact that $\theta_{(3)}$ has no $dt$ term (see \eqref{eq2.22}), we get
$$0=\theta_{(3)}(W_{0(-2)})+\theta_{(2)}(W_{0(-1)})=\Theta\bigg(s_{0(0)}^aZ_a+s_{0(1)}^0\frac{\partial}{\partial{t}}\bigg)
=s_{0(1)}^0.$$
By $\theta^a(T)=0$ and $\theta_{(2)}^a=0$ in \eqref{eq2.22}, we get
$$0=\theta_{(2)}^a(W_{0(-2)})+\theta_{(1)}^a(W_{0(-1)})=dz^a\bigg(s_{0(0)}^bZ_b\bigg)=s_{0(0)}^a.$$
We finish the proof of \eqref{eq2.29}.
\end{proof}
By Proposition \ref{prop2.2}, we have
\begin{equation}\label{eq2.36}
\begin{aligned}
&W_a=Z_a+\mathscr{O}_0=Z_a+\frac{2}{3}J_{ab(1)}z^b\frac{\partial}{\partial{t}}+\mathscr{O}_{1}
=Z_a+\frac{2}{3}J_{ab(1)}z^b\frac{\partial}{\partial{t}}+s_{a(2)}^bZ_b+s_{a(3)}^0Z_0,\\
&W_0=\frac{\partial}{\partial{t}}+\mathscr{O}_{0},
\end{aligned}
\end{equation}
where $s_{a(2)}^b$ and $s_{a(3)}^0$ are given by \eqref{eq2.29}. In our case, $W_a$ has extra term $\frac{2}{3}J_{ab(1)}z^b\frac{\partial}{\partial{t}}$, which vanishes in the CR case (cf. p. 314 in \cite{JL1}).
\begin{corollary}\label{cor2.3}
With respect to a special frame centered at $q$, the connection coefficients vanish at $q$, i.e.,
\begin{equation}\label{eq3.1b}
\Gamma_{jk}^l(q)=0,{\quad}\Gamma_{jk}^l\in\mathscr{O}_1.
\end{equation}
\end{corollary}
\begin{proof}
 By $\omega_{a(1)}^b=0$ in \eqref{eq2.22} we get $\Gamma_{ca}^b(q)=\omega_{a(1)}^b(W_{c(-1)})=0$. Again by \eqref{eq2.22},  $\omega_{a(1)}^b=0$ and $\omega_{a(2)}^b$ having no $dt$ term, we see that $\Gamma_{0a}^b(q)=\omega_{a(2)}^b(W_{0(-2)})+\omega_{a(1)}^b(W_{0(-1)})=0$. We also have $\Gamma_{j0}^{l}=0$ by $\nabla{T}=0$ and $\Gamma_{ab}^0=0$ by $\theta(\nabla_XY)=0$ for any $X,Y\in{HM}$. So \eqref{eq3.1b} follows.
 \end{proof}
\section{The asymptotic expansion of the almost complex structure, the curvature and Tanno tensors}
  In this section, we will discuss the curvature tensor,  the Tanno tensor and the almost complex structure with respect to the special frame centered at $q$.
\subsection{The Tanno tensor at point $q$}
  For the Tanno tensor, we write $Q(W_j,W_k)=(\nabla_{W_k}J)W_j=Q_{jk}^lW_l$. The components of the Tanno tensor $Q_{jk}^l$ can be written as:
  \begin{equation}\label{eq4.1}
  Q_{jk}^l=W_kJ_{\ j}^l-\Gamma_{kj}^sJ_{\ s}^l+\Gamma_{ks}^lJ_{\ j}^s.
  \end{equation}
So at the point $q$, we have
\begin{equation}\label{eq4.5a}
Q_{jk}^l(q)=W_kJ_{\ j}^l(q),
\end{equation}
 by \eqref{eq3.1b}. And noting that $\overline{Q_{\alpha\beta}^{\bar{\gamma}}}=Q_{\bar{\alpha}\bar{\beta}}^{\gamma}$ by definition, we set
\begin{equation}\label{eq4.2}
\mathfrak{Q}:=Q_{\alpha\beta}^{\bar{\gamma}}(q)Q_{\bar{\alpha}\bar{\beta}}^{\gamma}(q)
=\sum\limits_{\alpha,\beta,\gamma}|Q_{\alpha\beta}^{\bar{\gamma}}(q)|^2.
\end{equation}
\begin{proposition}\label{prop4.1}
With respect to a special frame centered at $q$, $J_{\ \alpha(1)}^{\gamma}=0$, $J_{\alpha\bar{\beta}(1)}=0$.
\end{proposition}
\begin{proof}
Since $J^2W_{\alpha}=-W_{\alpha}$ means $J_{\ \alpha}^bJ_{\ b}^{\gamma}=-\delta_{\alpha}^{\gamma}$, we get
$$0=J_{\ \alpha(0)}^{\beta}J_{\ \beta(1)}^{\gamma}+
J_{\ \alpha(1)}^{\beta}J_{\ \beta(0)}^{\gamma}+J_{\ \alpha(0)}^{\bar{\beta}}J_{\ \bar{\beta}(1)}^{\gamma}
+J_{\ \alpha(1)}^{\bar{\beta}}J_{\ \bar{\beta}(0)}^{\gamma}=2iJ_{\ \alpha(1)}^{\gamma},$$
by \eqref{eq2.18}. Hence $J_{\ \alpha(1)}^{\gamma}=0$.
And so $J_{\alpha\bar{\beta}(1)}=h_{\alpha\bar{\gamma}}J_{\ \bar{\beta}(1)}^{\bar{\gamma}}=0$.
\end{proof}
  \begin{proposition}\label{prop4.2}
  With respect to a special frame centered at $q$, components of the Tanno tensor Q at $q$ satisfy the following relations:
  \begin{equation}\label{Q}
  \begin{aligned}
  &Q_{ab}^0=Q_{0j}^k=Q_{i0}^k=0,\\
  &Q_{\alpha\beta}^{\gamma}(q)=Q_{\alpha\bar{\beta}}^{\gamma}(q)=Q_{\bar{\beta}\alpha}^{\gamma}(q)=0,\\
  &Q_{\beta\alpha}^{\bar{\rho}}(q)Q_{\bar{\alpha}\bar{\beta}}^{\rho}(q)=\frac{1}{2}\mathfrak{Q}.
  \end{aligned}
  \end{equation}
  In particular, $Q(q)\ne0$ if and only if $\mathfrak{Q}>0$.
  \end{proposition}
\begin{proof}
$Q_{ab}^0=0$ follows from that for any $X,Y\in{HM}$, we have
\begin{align}
\notag \theta(Q(X,Y))&=\theta((\nabla_YJ)X)=h(T,(\nabla_YJ)X)=h(T,\nabla_Y{(JX)})-h(T,J\nabla_YX)\\
\notag &=\theta(\nabla_Y{(JX)})-\theta(J\nabla_YX)=0,
\end{align}
by \eqref{eq2.1} and $JX$, $\nabla_Y{(JX)}$, $J\nabla_YX$ all being horizontal.

$Q_{0j}^k=0$ follows from $Q(T,Y)=(\nabla_YJ)T=\nabla_Y(JT)-J\nabla_YT=0$ by $JT=0$ and $\nabla{T}=0$. By \eqref{eq4.5a} and Proposition \ref{prop4.1}, $Q_{\alpha\beta}^{\gamma}(q)=W_{\beta}J_{\ \alpha}^{\gamma}(q)=W_{\beta(-1)}J_{\ \alpha(1)}^{\gamma}+W_{\beta(0)}(J_{\ \alpha(0)}^{\gamma})=0$. Similarly, $Q_{\alpha\bar{\beta}}^{\gamma}(q)=0$.

To prove $Q_{i0}^k=0$ and $Q_{\bar{\beta}\alpha}^{\gamma}(q)=0$, recall that
\begin{equation}\label{eq4.6}
2h(Q(X,Y),Z)=h\bigg(N^{(1)}(X,Z)-\theta(X)N^{(1)}(T,Z)-\theta(Z)N^{(1)}(X,T),JY\bigg),
\end{equation}
for any $X,Y,Z\in{TM}$ (cf. (15) in \cite{BD1}), where
$$N^{(1)}=[J,J]+2(d\theta)\otimes{T},{\quad}[J,J](X,Y)=J^2[X,Y]+[JX,JY]-J[JX,Y]-J[X,JY].$$

$Q_{i0}^k=0$ is equivalent to $Q(X,T)=0$, for any $X\in{TM}$. Apply $Y=T$ to \eqref{eq4.6} to get $h(Q(X,T),Z)=0$ for any $X,Z\in{TM}$ by $JT=0$.

Substituting $X=W_{\bar{\beta}},Y=W_a,Z=W_{\bar{\rho}}$ into \eqref{eq4.6}, and noting that $h(JX,JY)=h(X,Y)$ for any $X,Y\in{HM}$, we get
  \begin{align}\label{eq4.7}
  \notag 2h_{\gamma\bar{\rho}}&Q_{\bar{\beta}a}^{\gamma}=h(N^{(1)}(W_{\bar{\beta}},W_{\bar{\rho}}),JW_a)=h([J,J](W_{\bar{\beta}},W_{\bar{\rho}}),JW_a)\\
  \notag &=h(J^2[W_{\bar{\beta}},W_{\bar{\rho}}]+[JW_{\bar{\beta}},JW_{\bar{\rho}}]-J[JW_{\bar{\beta}},W_{\bar{\rho}}]-J[W_{\bar{\beta}},JW_{\bar{\rho}}],JW_a)\\
  &=-h([W_{\bar{\beta}},W_{\bar{\rho}}],JW_a)+h([JW_{\bar{\beta}},JW_{\bar{\rho}}],JW_a)
  -h([JW_{\bar{\beta}},W_{\bar{\rho}}],W_a)-h([W_{\bar{\beta}},JW_{\bar{\rho}}],W_a).
  \end{align}
  by using \eqref{eq2.1}.
  Note that $[X,Y]=\nabla_XY-\nabla_YX-\tau(X,Y)$ and $\tau(X,Y)=0\ {mod}\ T$, for any $X,Y\in{HM}$. We find that
  \begin{align}
  \notag [W_{\bar{\beta}},W_{\bar{\rho}}]&=\nabla_{W_{\bar{\beta}}}W_{\bar{\rho}}-\nabla_{W_{\bar{\rho}}}W_{\bar{\beta}}=\Gamma_{\bar{\beta}\bar{\rho}}^cW_c-\Gamma_{\bar{\rho}\bar{\beta}}^cW_c,{\quad}mod{\quad} T,\\
  \notag [JW_{\bar{\beta}},JW_{\bar{\rho}}]&=\nabla_{J_{\ \bar{\beta}}^cW_c}(J_{\ \bar{\rho}}^dW_d)-\nabla_{J_{\ \bar{\rho}}^dW_d}(J_{\ \bar{\beta}}^cW_c)\\
  \notag &=J_{\ \bar{\beta}}^c(W_cJ_{\ \bar{\rho}}^d)W_d+J_{\ \bar{\beta}}^cJ_{\ \bar{\rho}}^e\Gamma_{ce}^dW_d-J_{\ \bar{\rho}}^c(W_cJ_{\ \bar{\beta}}^d)W_d-J_{\ \bar{\rho}}^eJ_{\ \bar{\beta}}^c\Gamma_{ec}^dW_d,{\quad}mod{\quad}T,\\
  \notag [JW_{\bar{\beta}},W_{\bar{\rho}}]&=-W_{\bar{\rho}}J_{\ \bar{\beta}}^dW_d+J_{\ \bar{\beta}}^c\Gamma_{c\bar{\rho}}^dW_d-J_{\ \bar{\beta}}^c\Gamma_{\bar{\rho}c}^dW_d{\quad}mod{\quad}T,\\
  \notag [W_{\bar{\beta}},JW_{\bar{\rho}}]&=W_{\bar{\beta}}J_{\ \bar{\rho}}^dW_d+J_{\ \bar{\rho}}^c\Gamma_{\bar{\beta}c}^dW_d-J_{\ \bar{\rho}}^c\Gamma_{c\bar{\beta}}^dW_d{\quad}mod{\quad}T.
  \end{align}
  Then \eqref{eq4.7} becomes
  \begin{align}\label{eq4.8}
  \notag 2Q_{\bar{\beta}a}^{\rho}&=\big(-\Gamma_{\bar{\beta}\bar{\rho}}^d+\Gamma_{\bar{\rho}\bar{\beta}}^d\big)J_{\ a}^fh_{df}+\bigg(J_{\ \bar{\beta}}^c(W_cJ_{\ \bar{\rho}}^d)+J_{\ \bar{\beta}}^cJ_{\ \bar{\rho}}^e\Gamma_{ce}^d-J_{\ \bar{\rho}}^c(W_cJ_{\ \bar{\beta}}^d)-J_{\ \bar{\rho}}^eJ_{\ \bar{\beta}}^c\Gamma_{ec}^d\bigg)J_{\ a}^fh_{df}\\
  &+\big(W_{\bar{\rho}}J_{\ \bar{\beta}}^d-J_{\ \bar{\beta}}^c\Gamma_{c\bar{\rho}}^d+J_{\ \bar{\beta}}^c\Gamma_{\bar{\rho}c}^d\big)h_{ad}+\big(-W_{\bar{\beta}}J_{\ \bar{\rho}}^d-J_{\ \bar{\rho}}^c\Gamma_{\bar{\beta}c}^d+J_{\ \bar{\rho}}^c\Gamma_{c\bar{\beta}}^d\big)h_{ad},
  \end{align}
  by writing $JW_a=J_{\ a}^fW_f$ and $h_{\gamma\bar{\rho}}=\delta_{\gamma\bar{\rho}}$. \eqref{eq4.8} will be used later.

  At the point $q$, \eqref{eq4.8} for $a=\alpha$ becomes
  \begin{align}
  \notag 2Q_{\bar{\beta}\alpha}^{\rho}(q)&=\bigg(J_{\ \bar{\beta}}^c(q)(W_cJ_{\ \bar{\rho}}^d)(q)-J_{\ \bar{\rho}}^c(q)(W_cJ_{\ \bar{\beta}}^d)(q)\bigg)J_{\ \alpha}^f(q)h_{df}+W_{\bar{\rho}}J_{\ \bar{\beta}}^d(q)h_{\alpha{d}}-W_{\bar{\beta}}J_{\ \bar{\rho}}^d(q)h_{\alpha{d}}\\
  \notag &=W_{\bar{\beta}}J_{\ \bar{\rho}}^{\bar{\mu}}(q)h_{\alpha\bar{\mu}}-W_{\bar{\rho}}J_{\ \bar{\beta}}^{\bar{\mu}}(q)h_{\alpha\bar{\mu}}+W_{\bar{\rho}}J_{\ \bar{\beta}}^{\bar{\mu}}(q)h_{\alpha\bar{\mu}}-W_{\bar{\beta}}J_{\ \bar{\rho}}^{\bar{\mu}}(q)h_{\alpha\bar{\mu}}=0.
  \end{align}
  by vanishing of connection coefficients \eqref{eq3.1b}. So $Q_{\bar{\beta}\alpha}^{\gamma}(q)=0$.

  It remains to prove the last identity in \eqref{Q}. Similarly at the point $q$,  for $a=\bar{\alpha}$, \eqref{eq4.8} becomes
  \begin{align}
  \notag 2Q_{\bar{\beta}\bar{\alpha}}^{\rho}(q)&=2h_{\gamma\bar{\rho}}Q_{\bar{\beta}\bar{\alpha}}^{\gamma}(q)=J_{\ \bar{\beta}}^c(q)(W_cJ_{\ \bar{\rho}}^d)(q)J_{\ \bar{\alpha}}^f(q)h_{df}-J_{\ \bar{\rho}}^c(q)(W_cJ_{\ \bar{\beta}}^d)(q)J_{\ \bar{\alpha}}^f(q)h_{df}\\
  \notag &\ \ \ +W_{\bar{\rho}}J_{\ \bar{\beta}}^d(q)h_{\bar{\alpha}{d}}-W_{\bar{\beta}}J_{\ \bar{\rho}}^d(q)h_{\bar{\alpha}{d}}\\
  \notag &=-W_{\bar{\beta}}J_{\ \bar{\rho}}^{\mu}(q)h_{\bar{\alpha}\mu}+W_{\bar{\rho}}J_{\ \bar{\beta}}^{\mu}(q)h_{\bar{\alpha}\mu}+W_{\bar{\rho}}J_{\ \bar{\beta}}^{\mu}(q)h_{\bar{\alpha}\mu}-W_{\bar{\beta}}J_{\ \bar{\rho}}^{\mu}(q)h_{\bar{\alpha}\mu}\\
  \notag &=-2Q_{\bar{\rho}\bar{\beta}}^{\mu}(q)h_{\mu\bar{\alpha}}+2Q_{\bar{\beta}\bar{\rho}}^{\mu}(q)h_{\mu\bar{\alpha}}
  =2Q_{\bar{\beta}\bar{\rho}}^{\alpha}(q)-2Q_{\bar{\rho}\bar{\beta}}^{\alpha}(q)
  \end{align}
  by \eqref{eq4.5a}. Taking conjugate on the both sides of last equation, we get $Q_{\beta\alpha}^{\bar{\rho}}(q)=Q_{\beta\rho}^{\bar{\alpha}}(q)-Q_{\rho\beta}^{\bar{\alpha}}(q).$
  Then we have
\begin{align}
\notag \sum\limits_{\alpha,\beta,\rho}|Q_{\beta\alpha}^{\bar{\rho}}(q)|^2&=\sum\limits_{\alpha,\beta,\rho}Q_{\beta\alpha}^{\bar{\rho}}(q)Q_{\bar{\beta}\bar{\alpha}}^{\rho}(q)
=\sum\limits_{\alpha,\beta,\rho}\bigg(Q_{\beta\rho}^{\bar{\alpha}}(q)-Q_{\rho\beta}^{\bar{\alpha}}(q)\bigg)\bigg(Q_{\bar{\beta}\bar{\rho}}^{\alpha}(q)-Q_{\bar{\rho}\bar{\beta}}^{\alpha}(q)\bigg)\\
\notag &=\sum\limits_{\alpha,\beta,\rho}2|Q_{\beta\rho}^{\bar{\alpha}}(q)|^2-2\sum\limits_{\alpha,\beta,\rho}Q_{\beta\rho}^{\bar{\alpha}}(q)Q_{\bar{\rho}\bar{\beta}}^{\alpha}(q).
\end{align}
By changing indices, we get $2Q_{\alpha\beta}^{\bar{\gamma}}(q)Q_{\bar{\beta}\bar{\alpha}}^{\gamma}(q)=\sum\limits_{\alpha,\beta,\gamma}|Q_{\alpha\beta}^{\bar{\gamma}}(q)|^2=\mathfrak{Q}$.
\end{proof}
\begin{remark}
On a contact Riemannian manifold, if we choose a local $T^{(1,0)}M$-frame, only $Q_{\alpha\beta}^{\bar{\gamma}}$ is non-vanishing for the Tanno tensor $Q$ (cf. (16)-(18) in \cite{BD1}). Here we have the similar property at point $q$.
\end{remark}
\subsection{The asymptotic expansion of the almost complex structure at point $q$}
\begin{proposition}\label{prop4.3}
  With respect to a special frame centered at $q$, we have
  \begin{equation}\label{eq4.9}
  J_{\alpha\beta(1)}=J_{\ \beta(1)}^{\bar{\alpha}}=Q_{\beta\gamma}^{\bar{\alpha}}(q)z^{\gamma},
  \end{equation}
  \begin{equation}\label{eq4.10}
  J_{\alpha\bar{\beta}(2)}=\frac{i}{2}Q_{\alpha\lambda}^{\bar{\gamma}}(q)Q_{\bar{\beta}\bar{\mu}}^{\gamma}(q)z^{\lambda}z^{\bar{\mu}}.
  \end{equation}
  \end{proposition}
  \begin{proof}
  By $Q_{\beta\bar{\gamma}}^{\bar{\alpha}}(q)=0$ in \eqref{Q} and the expansion \eqref{eq2.12a}, we get $$J_{\alpha\beta(1)}=h_{\alpha\bar{\rho}}J_{\ \beta(1)}^{\bar{\rho}}=J_{\ \beta(1)}^{\bar{\alpha}}=z^cZ_cJ_{\ \beta}^{\bar{\alpha}}(q)=z^cQ_{\beta{c}}^{\bar{\alpha}}(q)=Q_{\beta\gamma}^{\bar{\alpha}}(q)z^{\gamma}.$$
  It follows from $J_{\ \alpha}^bJ_{\ b}^{\gamma}=-\delta_{\alpha}^{\gamma}$ that
  $$0=J_{\ \alpha(0)}^{\beta}J_{\ \beta(2)}^{\gamma}+J_{\ \alpha(1)}^{\beta}J_{\ \beta(1)}^{\gamma}+J_{\ \alpha(2)}^{\beta}J_{\ \beta(0)}^{\gamma}+J_{\ \alpha(0)}^{\bar{\beta}}J_{\ \bar{\beta}(2)}^{\gamma}+J_{\ \alpha(1)}^{\bar{\beta}}J_{\ \bar{\beta}(1)}^{\gamma}+J_{\ \alpha(2)}^{\bar{\beta}}J_{\ \bar{\beta}(0)}^{\gamma}.$$
  By \eqref{eq2.18} and Proposition \ref{prop4.1}, we get $2iJ_{\ \alpha(2)}^{\gamma}+J_{\ \alpha(1)}^{\bar{\beta}}J_{\ \bar{\beta}(1)}^{\gamma}=0$, i.e. $J_{\ \alpha(2)}^{\gamma}=\frac{i}{2}J_{\ \alpha(1)}^{\bar{\beta}}J_{\ \bar{\beta}(1)}^{\gamma}$.
  So
  $$J_{\alpha\bar{\beta}(2)}=h_{\alpha\bar{\gamma}}J_{\ \bar{\beta}(2)}^{\bar{\gamma}}=-\frac{i}{2}\delta_{\alpha\bar{\gamma}}J_{\ \bar{\beta}(1)}^{\rho}J_{\ \rho(1)}^{\bar{\gamma}}=-\frac{i}{2}J_{\ \bar{\beta}(1)}^{\rho}J_{\ \rho(1)}^{\bar{\alpha}}=\frac{i}{2}J_{\ \bar{\beta}(1)}^{\rho}J_{\ \alpha(1)}^{\bar{\rho}}=\frac{i}{2}Q_{\alpha\lambda}^{\bar{\rho}}(q)Q_{\bar{\beta}\bar{\mu}}^{\rho}(q)z^{\lambda}z^{\bar{\mu}},$$
  by \eqref{eq4.9} and $J_{\ \rho(1)}^{\bar{\alpha}}=-J_{\ \alpha(1)}^{\bar{\rho}}$ in \eqref{eq2.18}. We complete the proof of this proposition.
  \end{proof}
Proposition \ref{prop4.1} and Proposition \ref{prop4.3} leads to the following corollary:
\begin{corollary}\label{cor4.1}
$s_{b(2)}^0$ in \eqref{eq2.29} in Proposition \ref{prop2.2} can be rewritten as
\begin{equation}
\notag s_{\beta(2)}^0=\frac{2}{3}Q_{\alpha\gamma}^{\bar{\beta}}(p)z^{\alpha}z^{\gamma},{\quad}s_{\bar{\beta}(2)}^0
=\frac{2}{3}Q_{\bar{\alpha}\bar{\gamma}}^{\beta}(p)z^{\bar{\alpha}}z^{\bar{\gamma}}.
\end{equation}
\end{corollary}
The following relation shows that $J_{\ \bar{\beta}(2)}^{\gamma}$ is independent of $t$.
\begin{proposition}\label{prop4.3a}
\begin{equation}
J_{\ \beta(2)}^{\bar{\gamma}}=\frac{1}{2}z_cz_dZ_cZ_dJ_{\ \beta}^{\bar{\gamma}}(q),{\quad}J_{\alpha\beta(2)}=\frac{1}{2}z_cz_dZ_cZ_dJ_{\alpha\beta}(q).
\end{equation}
\end{proposition}
\begin{proof}
By expansion \eqref{eq2.12a},
$$J_{\ \beta(2)}^{\bar{\gamma}}=\frac{1}{2}z_cz_dZ_cZ_dJ_{\ \beta}^{\bar{\gamma}}(q)+t\frac{\partial{J}_{\ \beta}^{\bar{\gamma}}}{\partial{t}}(q).$$
To get this proposition, we need to prove that $\frac{\partial}{\partial{t}}J_{\ \beta}^{\bar{\gamma}}(q)=0$.

By \eqref{eq4.5a} and \eqref{Q}, $0=Q_{\beta0}^{\bar{\gamma}}(q)=W_0J_{\ \beta}^{\bar{\gamma}}(q)$. Noting that $W_{0(-2)}=\frac{\partial}{\partial{t}}$, $W_{0(-1)}=0$ (cf. \eqref{eq2.36}) and $J_{\ \beta(0)}^{\bar{\gamma}}=0$ (cf. \eqref{eq2.18}), we get
$$0=W_0J_{\ \beta}^{\bar{\gamma}}(q)
=W_{0(-2)}\bigg(J_{\ \beta(2)}^{\bar{\gamma}}\bigg)+W_{0(-1)}\bigg(J_{\ \beta(1)}^{\bar{\gamma}}\bigg)+W_{0(0)}\bigg(J_{\ \beta(0)}^{\bar{\gamma}}\bigg)=\frac{\partial}{\partial{t}}J_{\ \beta}^{\bar{\gamma}}(q).$$
By $h_{ab}$ being constant, we get $\frac{\partial}{\partial{t}}J_{\alpha\beta}(q)=h_{\alpha\bar{\gamma}}\frac{\partial{J}_{\ \beta}^{\bar{\gamma}}}{\partial{t}}(q)=0$. So we prove Proposition \ref{prop4.3a}.
\end{proof}
\subsection{The asymptotic expansion of curvature tensors}
\begin{proposition}\label{prop4.4}
With respect to a special frame centered at $q$, we have
\begin{equation}\label{eq4.12}
R_{\bar{\beta}\ \lambda\mu}^{\ \alpha}(q)=\frac{1}{4}\bigg(Q_{\mu\lambda}^{\bar{\sigma}}(q)-Q_{\lambda\mu}^{\bar{\sigma}}(q)\bigg)
\bigg(Q_{\bar{\alpha}\bar{\beta}}^{\sigma}(q)-Q_{\bar{\beta}\bar{\alpha}}^{\sigma}(q)\bigg).
\end{equation}
In particular,
\begin{equation}\label{eq4.13}
R_{\bar{\beta}\ \alpha\beta}^{\ \alpha}(q)=-\frac{1}{4}\mathfrak{Q}.
\end{equation}
\end{proposition}
\begin{proof}
At the point $q$, it follows from \eqref{eq2.14} and \eqref{eq3.1b} that
\begin{equation}\label{eq4.17}
R_{\bar{\beta}\ \lambda\mu}^{\ \alpha}(q)=W_{\lambda}\Gamma_{\mu\bar{\beta}}^{\alpha}(q)-W_{\mu}\Gamma_{\lambda\bar{\beta}}^{\alpha}(q)
=Z_{\lambda}\bigg(\Gamma_{\mu\bar{\beta}(1)}^{\alpha}\bigg)-Z_{\mu}\bigg(\Gamma_{\lambda\bar{\beta}(1)}^{\alpha}\bigg).
\end{equation}
Let us calculate $Z_c(\Gamma_{db(1)}^a)$. By \eqref{eq4.1}, we have: $Q_{\bar{\beta}\alpha}^{\gamma}=W_{\alpha}J_{\ \bar{\beta}}^{\gamma}-\Gamma_{{\alpha}\bar{\beta}}^dJ_{\ d}^{\gamma}+\Gamma_{\alpha{d}}^{\gamma}J_{\ \bar{\beta}}^d$. So
$Q_{\bar{\beta}\alpha(1)}^{\gamma}={(W_{\alpha}J_{\ \bar{\beta}}^{\gamma})}_{(1)}-\Gamma_{\alpha\bar{\beta}(1)}^dJ_{\ d}^{\gamma}(q)+\Gamma_{\alpha{d}(1)}^{\gamma}J_{\ \bar{\beta}}^d(q)={(W_{\alpha}J_{\ \bar{\beta}}^{\gamma})}_{(1)}-2i\Gamma_{\alpha\bar{\beta}(1)}^{\gamma}$,
i.e.,
\begin{equation}\label{eq4.14}
\Gamma_{\alpha\bar{\beta}(1)}^{\gamma}=
\frac{i}{2}Q_{\bar{\beta}\alpha(1)}^{\gamma}-\frac{i}{2}{(W_{\alpha}J_{\ \bar{\beta}}^{\gamma})}_{(1)}.
\end{equation}
First we deal with the term $Q_{\bar{\beta}\alpha(1)}^{\gamma}$ in \eqref{eq4.14}. Take index $a$ to be $\alpha$ in \eqref{eq4.8} and consider the homogeneous part of degree $1$. For the right hand side, noting
$J_{\ \alpha}^{\beta}(q)=i\delta_{\alpha}^{\beta}$,
$J_{\ \alpha}^{\bar{\beta}}(q)=0$ by Proposition \ref{prop2.0d}, $J_{\ \alpha(1)}^{\beta}=0$ by Proposition \ref{prop4.1} and $Q_{\alpha\bar{\beta}}^{\gamma}(q)=0$ by \eqref{Q} in Proposition \ref{prop4.2}, we have
\begin{equation}\label{eq4.14a}
\begin{aligned}
\bigg(\big(-\Gamma_{\bar{\beta}\bar{\rho}}^d+\Gamma_{\bar{\rho}\bar{\beta}}^d\big)J_{\ \alpha}^fh_{df}\bigg)_{(1)}=(-\Gamma_{\bar{\beta}\bar{\rho}(1)}^{\bar{\mu}}+\Gamma_{\bar{\rho}\bar{\beta}(1)}^{\bar{\mu}})J_{\ \alpha}^{\gamma}(q)h_{\gamma\bar{\mu}}
=-i\Gamma_{\bar{\beta}\bar{\rho}(1)}^{\bar{\alpha}}+i\Gamma_{\bar{\rho}\bar{\beta}(1)}^{\bar{\alpha}},
\end{aligned}
\end{equation}
and
\begin{equation}\label{eq4.14b}
\begin{aligned}
&\ \ \ \bigg\{\bigg(J_{\ \bar{\beta}}^c(W_cJ_{\ \bar{\rho}}^d)+J_{\ \bar{\beta}}^cJ_{\ \bar{\rho}}^e\Gamma_{ce}^d-J_{\ \bar{\rho}}^c(W_cJ_{\ \bar{\beta}}^d)-J_{\ \bar{\rho}}^eJ_{\ \bar{\beta}}^c\Gamma_{ec}^d\bigg)J_{\ \alpha}^fh_{df}\bigg\}_{(1)}\\
&=\bigg(J_{\ \bar{\beta}}^c(W_cJ_{\ \bar{\rho}}^{\bar{\mu}})+J_{\ \bar{\beta}}^cJ_{\ \bar{\rho}}^e\Gamma_{ce}^{\bar{\mu}}-J_{\ \bar{\rho}}^c(W_cJ_{\ \bar{\beta}}^{\bar{\mu}})-J_{\ \bar{\rho}}^eJ_{\ \bar{\beta}}^c\Gamma_{ec}^{\bar{\mu}}\bigg)_{(1)}J_{\ \alpha}^{\gamma}(q)h_{\gamma\bar{\mu}}\\
&\ \ \ +\bigg(J_{\ \bar{\beta}}^c(W_cJ_{\ \bar{\rho}}^{\mu})+J_{\ \bar{\beta}}^cJ_{\ \bar{\rho}}^e\Gamma_{ce}^{\mu}-J_{\ \bar{\rho}}^c(W_cJ_{\ \bar{\beta}}^{\mu})-J_{\ \bar{\rho}}^eJ_{\ \bar{\beta}}^c\Gamma_{ec}^{\mu}\bigg)_{(0)}J_{\ \alpha(1)}^{\bar{\gamma}}h_{\mu\bar{\gamma}}\\
&=i\delta_{\alpha\bar{\mu}}\bigg(J_{\ \bar{\beta}(1)}^{\lambda}W_{\lambda}J_{\ \bar{\rho}}^{\bar{\mu}}(q)
+J_{\ \bar{\beta}}^{\bar{\lambda}}(q)(W_{\bar{\lambda}}J_{\ \bar{\rho}}^{\bar{\mu}})_{(1)}
+J_{\ \bar{\beta}}^{\bar{\lambda}}(q)J_{\ \bar{\rho}}^{\bar{\sigma}}(q)\Gamma_{\bar{\lambda}\bar{\sigma}(1)}^{\bar{\mu}}\\
&\ \ \ -J_{\ \bar{\rho}(1)}^{\sigma}W_{\sigma}J_{\ \bar{\beta}}^{\bar{\mu}}(q)
-J_{\ \bar{\rho}}^{\bar{\sigma}}(q)(W_{\bar{\sigma}}J_{\ \bar{\beta}}^{\bar{\mu}})_{(1)}
-J_{\ \bar{\rho}}^{\bar{\sigma}}(q)J_{\ \bar{\beta}}^{\bar{\lambda}}(q)\Gamma_{{\bar{\sigma}}\bar{\lambda}(1)}^{\bar{\mu}}\bigg)\\
&\ \ \ +\bigg(J_{\ \bar{\beta}}^{\bar{\lambda}}(q)W_{\bar{\lambda}}J_{\ \bar{\rho}}^{\mu}(q)
-J_{\ \bar{\rho}}^{\bar{\sigma}}(q)W_{\bar{\sigma}}J_{\ \bar{\beta}}^{\mu}(q)\bigg)J_{\ \alpha(1)}^{\bar{\gamma}}h_{\mu\bar{\gamma}}\\
&=iJ_{\ \bar{\beta}(1)}^{\lambda}Q_{\bar{\rho}\lambda}^{\bar{\alpha}}(q)
+(W_{\bar{\beta}}J_{\ \bar{\rho}}^{\bar{\alpha}})_{(1)}-i\Gamma_{\bar{\beta}\bar{\rho}(1)}^{\bar{\alpha}}
-iJ_{\ \bar{\rho}(1)}^{\sigma}Q_{\bar{\beta}\sigma}^{\bar{\alpha}}(q)
-(W_{\bar{\rho}}J_{\ \bar{\beta}}^{\bar{\alpha}})_{(1)}+i\Gamma_{\bar{\rho}\bar{\beta}(1)}^{\bar{\alpha}}\\
&\ \ \ -iQ_{\bar{\rho}\bar{\beta}}^{\gamma}(q)J_{\ \alpha(1)}^{\bar{\gamma}}
+iQ_{\bar{\beta}\bar{\rho}}^{\gamma}(q)J_{\ \alpha(1)}^{\bar{\gamma}}\\
&=(W_{\bar{\beta}}J_{\ \bar{\rho}}^{\bar{\alpha}})_{(1)}-(W_{\bar{\rho}}J_{\ \bar{\beta}}^{\bar{\alpha}})_{(1)}
-i\Gamma_{\bar{\beta}\bar{\rho}(1)}^{\bar{\alpha}}+i\Gamma_{\bar{\rho}\bar{\beta}(1)}^{\bar{\alpha}}
-iQ_{\bar{\rho}\bar{\beta}}^{\gamma}(q)J_{\ \alpha(1)}^{\bar{\gamma}}+iQ_{\bar{\beta}\bar{\rho}}^{\gamma}(q)J_{\ \alpha(1)}^{\bar{\gamma}},
\end{aligned}
\end{equation}
and
\begin{equation}\label{eq4.14c}
\begin{aligned}
\bigg(\big(W_{\bar{\rho}}J_{\ \bar{\beta}}^{\bar{\mu}}-J_{\ \bar{\beta}}^c\Gamma_{c\bar{\rho}}^{\bar{\mu}}
+J_{\ \bar{\beta}}^c\Gamma_{\bar{\rho}c}^{\bar{\mu}}\big)h_{\alpha\bar{\mu}}\bigg)_{(1)}
&=(W_{\bar{\rho}}J_{\ \bar{\beta}}^{\bar{\mu}})_{(1)}\delta_{\alpha\bar{\mu}}
-J_{\ \bar{\beta}}^{\bar{\gamma}}(q)\Gamma_{\bar{\gamma}\bar{\rho}(1)}^{\bar{\mu}}\delta_{\alpha\bar{\mu}}
+J_{\ \bar{\beta}}^{\bar{\gamma}}(q)\Gamma_{\bar{\rho}\bar{\gamma}(1)}^{\bar{\mu}}\delta_{\alpha\bar{\mu}}\\
&=(W_{\bar{\rho}}J_{\ \bar{\beta}}^{\bar{\alpha}})_{(1)}+i\Gamma_{\bar{\beta}\bar{\rho}(1)}^{\bar{\alpha}}
-i\Gamma_{\bar{\rho}\bar{\beta}(1)}^{\bar{\alpha}},
\end{aligned}
\end{equation}
and
\begin{equation}\label{eq4.14d}
\begin{aligned}
\bigg(\big(-W_{\bar{\beta}}J_{\ \bar{\rho}}^{\bar{\mu}}-J_{\ \bar{\rho}}^c\Gamma_{\bar{\beta}c}^{\bar{\mu}}+J_{\ \bar{\rho}}^c\Gamma_{c\bar{\beta}}^{\bar{\mu}}\big)h_{\alpha\bar{\mu}}\bigg)_{(1)}=-(W_{\bar{\beta}}J_{\ \bar{\rho}}^{\bar{\alpha}})_{(1)}+i\Gamma_{\bar{\beta}\bar{\rho}(1)}^{\bar{\alpha}}
-i\Gamma_{\bar{\rho}\bar{\beta}(1)}^{\bar{\alpha}}.
\end{aligned}
\end{equation}
Substitute the summation of \eqref{eq4.14a}-\eqref{eq4.14d} into \eqref{eq4.8} to get
$2Q_{\bar{\beta}\alpha(1)}^{\rho}=iQ_{\bar{\beta}\bar{\rho}}^{\gamma}(q)J_{\ \alpha(1)}^{\bar{\gamma}}
-iQ_{\bar{\rho}\bar{\beta}}^{\gamma}(q)J_{\ \alpha(1)}^{\bar{\gamma}}$,
namely
\begin{equation}\label{eq4.15}
  Q_{\bar{\beta}\alpha(1)}^{\rho}=\frac{i}{2}Q_{\bar{\beta}\bar{\rho}}^{\gamma}(q)J_{\ \alpha(1)}^{\bar{\gamma}}-\frac{i}{2}Q_{\bar{\rho}\bar{\beta}}^{\gamma}(q)J_{\ \alpha(1)}^{\bar{\gamma}}.
\end{equation}
Now we deal with ${(W_{\alpha}J_{\ \bar{\beta}}^{\gamma})}_{(1)}$ in \eqref{eq4.14}. By expansion \eqref{eq2.12a}, $J_{\ \bar{\beta}(1)}^{\gamma}=z^cZ_cJ_{\ \bar{\beta}}^{\gamma}(q)$ is independent with $t$. And note that $W_{\alpha(0)}=s_{\alpha(2)}^0\frac{\partial}{\partial{t}}$ by \eqref{eq2.36}. So $W_{\alpha(0)}(J_{\ \bar{\beta}(1)}^{\gamma})=0$. By $J_{\ \bar{\beta}}^{\gamma}(q)=0$ (see \eqref{eq2.18}) and $J_{\ \bar{\beta}(2)}^{\gamma}$ being independent with $t$ in Proposition \ref{prop4.3a}, we get
\begin{align}\label{eq4.16}
\notag {(W_{\alpha}J_{\ \bar{\beta}}^{\gamma})}_{(1)}&=\Bigg(\bigg(Z_{\alpha}+W_{\alpha(0)}+W_{\alpha(1)}+\mathscr{O}_2\bigg)
\bigg(J_{\ \bar{\beta}(1)}^{\gamma}+J_{\ \bar{\beta}(2)}^{\gamma}+\mathscr{O}_3\bigg)\Bigg)_{(1)}\\
\notag &=W_{\alpha(0)}\bigg(J_{\ \bar{\beta}(1)}^{\gamma}\bigg)+Z_{\alpha}\bigg(J_{\ \bar{\beta}(2)}^{\gamma}\bigg)=Z_{\alpha}\bigg(J_{\ \bar{\beta}(2)}^{\gamma}\bigg)\\
&=\frac{1}{2}\bigg(z^cZ_{\alpha}Z_cJ_{\ \bar{\beta}}^{\gamma}(q)+z^cZ_cZ_{\alpha}J_{\ \bar{\beta}}^{\gamma}(q)\bigg).
\end{align}
So by \eqref{eq4.14}, \eqref{eq4.15} and \eqref{eq4.16}, we get
  \begin{align}\label{eq4.18}
  \notag Z_{\lambda}\bigg(\Gamma_{\mu\bar{\beta}(1)}^{\alpha}\bigg)&=\frac{i}{2}Z_{\lambda}\bigg(Q_{\bar{\beta}\mu(1)}^{\alpha}-{(W_{\mu}J_{\ \bar{\beta}}^{\alpha})}_{(1)}\bigg)\\
  \notag &=-\frac{1}{4}Z_{\lambda}\bigg(Q_{\bar{\beta}\bar{\alpha}}^{\sigma}(q)J_{\ \mu(1)}^{\bar{\sigma}}-Q_{\bar{\alpha}\bar{\beta}}^{\sigma}(q)J_{\ \mu(1)}^{\bar{\sigma}}\bigg)-\frac{i}{4}Z_{\lambda}\bigg(z^cZ_{\mu}Z_cJ_{\ \bar{\beta}}^{\alpha}(q)+z^cZ_cZ_{\mu}J_{\ \bar{\beta}}^{\alpha}(q)\bigg)\\
  &=\frac{1}{4}\bigg(Q_{\bar{\alpha}\bar{\beta}}^{\sigma}(q)Q_{\mu\lambda}^{\bar{\sigma}}(q)-Q_{\bar{\beta}\bar{\alpha}}^{\sigma}(q)Q_{\mu\lambda}^{\bar{\sigma}}(q)\bigg)
  -\frac{i}{4}Z_{\mu}Z_{\lambda}J_{\ \bar{\beta}}^{\alpha}(q)-\frac{i}{4}Z_{\lambda}Z_{\mu}J_{\ \bar{\beta}}^{\alpha}(q).
  \end{align}
  Note that for the last identity, we use the relation
  $Q_{\mu\lambda}^{\bar{\sigma}}(q)=W_{\lambda}J_{\ \mu}^{\bar{\sigma}}(q)=Z_{\lambda}(J_{\ \mu(1)}^{\bar{\sigma}})=Z_{\lambda}(J_{\ \mu(1)}^{\bar{\sigma}})$
  by \eqref{eq4.5a}, and $J_{\ \mu}^{\bar{\sigma}}(q)=0$. Similarly, we get
  \begin{equation}\label{eq4.19}
  Z_{\mu}\bigg(\Gamma_{\lambda\bar{\beta}(1)}^{\alpha}\bigg)
  =\frac{1}{4}\bigg(Q_{\bar{\alpha}\bar{\beta}}^{\sigma}(q)Q_{\lambda\mu}^{\bar{\sigma}}(q)
  -Q_{\bar{\beta}\bar{\alpha}}^{\sigma}(q)Q_{\lambda\mu}^{\bar{\sigma}}(q)\bigg)
  -\frac{i}{4}Z_{\lambda}Z_{\mu}J_{\ \bar{\beta}}^{\alpha}(q)-\frac{i}{4}Z_{\mu}Z_{\lambda}J_{\ \bar{\beta}}^{\alpha}(q).
  \end{equation}
  Substituting \eqref{eq4.18} and \eqref{eq4.19} into \eqref{eq4.17} we get \eqref{eq4.12}.
In particular,
  \begin{align}
  \notag R_{\bar{\beta}\ \alpha\beta}^{\ \alpha}(q)&=\frac{1}{4}\bigg(Q_{\beta\alpha}^{\bar{\sigma}}(q)Q_{\bar{\alpha}\bar{\beta}}^{\sigma}(q)-Q_{\beta\alpha}^{\bar{\sigma}}(q)Q_{\bar{\beta}\bar{\alpha}}^{\sigma}(q)-Q_{\alpha\beta}^{\bar{\sigma}}(q)Q_{\bar{\alpha}\bar{\beta}}^{\sigma}(q)+Q_{\alpha\beta}^{\bar{\sigma}}(q)Q_{\bar{\beta}\bar{\alpha}}^{\sigma}(q)\bigg)\\
  \notag &=\frac{1}{2}(Q_{\beta\alpha}^{\bar{\sigma}}(q)Q_{\bar{\alpha}\bar{\beta}}^{\sigma}(q)-|Q_{\alpha\beta}^{\bar{\sigma}}(q)|^2)=-\frac{1}{4}|Q_{\alpha\beta}^{\bar{\sigma}}(q)|^2=-\frac{1}{4}\mathfrak{Q},
  \end{align}
  by \eqref{Q}.
  \end{proof}
\begin{proposition}\label{prop4.5}(An identity of Bianchi type)
  With respect to a special frame associated with $\theta$, the components of the curvature tensor $R$ have the following relation:
  \begin{equation}\label{eq4.20}
  -R_{\beta\ \lambda\bar{\mu}}^{\ \alpha}+R_{\lambda\ \beta\bar{\mu}}^{\ \alpha}+R_{\bar{\mu}\ \lambda\beta}^{\ \alpha}=0,{\quad}mod{\quad}\mathscr{A}\cup\Gamma.
  \end{equation}
  $\Gamma$ means terms depending on $\Gamma_{jk}^l$ linearly or quadratically.
  \end{proposition}
  \begin{proof}
  Differentiate the second identity in the structure equations \eqref{SE} to get
  \begin{equation}\label{eq4.21} 0=d\theta^{\beta}\wedge\omega_{\beta}^{\alpha}-\theta^{\beta}\wedge{d\omega_{\beta}^{\alpha}}+d\theta^{\bar{\beta}}\wedge\omega_{\bar{\beta}}^{\alpha}-\theta^{\bar{\beta}}\wedge{d\omega_{\bar{\beta}}^{\alpha}}+d\theta\wedge\tau^{\alpha},\quad mod\quad\theta.
  \end{equation}
  By the definition $\omega_b^a=\Gamma_{ja}^b\theta^j$ and $\tau^{\alpha}$, we have
  $$d\theta^b\wedge\omega_b^a=0,\quad mod\quad \Gamma,{\quad}{\text and} {\quad}d\theta\wedge\tau^{\alpha}=0,\quad mod\quad \mathscr{A},$$
  and by the third identity in \eqref{SE}
  $$d\omega_b^a=R_{a\ \lambda\bar{\mu}}^{\ b}\theta^{\lambda}\wedge\theta^{\bar{\mu}}+\frac{1}{2}R_{a\ \lambda\mu}^{\ b}\theta^{\lambda}\wedge\theta^{\mu}+\frac{1}{2}R_{a\ \bar{\lambda}\bar{\mu}}^{\ b}\theta^{\bar{\lambda}}\wedge\theta^{{\bar\mu}},\quad mod\quad \theta\cup\Gamma.$$
  Consequently, by substituting the above identities into \eqref{eq4.21}, we find that
  \begin{align}
  \notag 0=&-R_{\beta\ \lambda\bar{\mu}}^{\ \alpha}\theta^{\beta}\wedge\theta^{\lambda}\wedge\theta^{\bar{\mu}}-\frac{1}{2}R_{\bar{\beta}\ \lambda\mu}^{\ \alpha}\theta^{\bar{\beta}}\wedge\theta^{\lambda}\wedge\theta^{\mu},\quad mod\quad \mathscr{A}\cup\Gamma.
  \end{align}
  Equivalently,
  \begin{align}
  \notag 0&=-R_{\beta\ \lambda\bar{\mu}}^{\ \alpha}+R_{\lambda\ \beta\bar{\mu}}^{\ \alpha}-\frac{1}{2}R_{\bar{\mu}\ \beta\lambda}^{\ \alpha}+\frac{1}{2}R_{\bar{\mu}\ \lambda\beta}^{\ \alpha}=-R_{\beta\ \lambda\bar{\mu}}^{\ \alpha}+R_{\lambda\ \beta\bar{\mu}}^{\ \alpha}+R_{\bar{\mu}\ \lambda\beta}^{\ \alpha},{\quad}mod{\quad}\mathscr{A}\cup\Gamma.
  \end{align}
  \end{proof}
\section{The normalized special frame}
  As in the CR case (cf. Section 3 in \cite{JL1}), to simplify the calculation of the asymptotic expansion of the Yamabe functional, we choose a conformal contact form such that certain components of the Webster torsion and curvature tensors vanish at the point $q$.

  Consider the conformal transformation
  \begin{equation}\label{eq3.1}
  \widehat{\theta}=e^{2u}\theta.
  \end{equation}
The main theorem of this section is the following.
  \begin{theorem}\label{thm3.1}
  For a contact Riemannian manifold $(M,\theta,h,J)$, there exists $(M,\widehat{\theta},\widehat{h},\widehat{J})$ with $\widehat{\theta}=e^{2u}\theta$, such that
  \begin{align}
  \notag &\widehat{R}_{\alpha\ \gamma{\bar{\beta}}}^{\ \gamma}(q)=0,{\quad}\widehat{A}_{\alpha\beta}(q)=0,
  \end{align}
  where $\widehat{R}_{\alpha\ \gamma{\bar{\beta}}}^{\ \gamma}(q)$ and $\widehat{A}_{\alpha\beta}(q)$ are the components of the curvature tensor and the Webster torsion tensor with respect to the special frame $\{\widehat{W_a},\widehat{T}\}$ of $(M,\widehat{\theta},\widehat{h},\widehat{J})$ centered at $q$.
  \end{theorem}
  In this section, we will work over the frame $\{W_a,\widehat{T}\}$ under the conformal transformation. And we abuse the notation to denote the components of the Webster torsion tensor and the curvature tensor with respect to $\{W_a,\widehat{T}\}$ also by $\widehat{A}_{\alpha\beta}(q)$ and $\widehat{R}_{\alpha\ \gamma\bar{\beta}}^{\ \gamma}(q)$.
  We will explain why we can do this later in Lemma \ref{lem3.5}.
\subsection{The transformation formulae under the conformal transformation}
First we do not change the special frame $\{W_a\}$ in the horizontal space. As mentioned in \eqref{eqA.1}, under the conformal transformation, we have $(\theta,J,T,h)\to(\widehat{\theta},\widehat{J},\widehat{T},\widehat{h})$ with
\begin{equation}\label{eq3.2}
\begin{aligned}
&\widehat{T}=e^{-2u}(T+J_{\ a}^bu^aW_b),\\
&\widehat{h}_{ab}=e^{2u}h_{ab},\\
&\widehat{J}_{\ a}^b=J_{\ a}^b,
\end{aligned}
\end{equation}
where $u_a=W_au$ and $u^a=h^{ab}u_b$. And we also get $\widehat{J}_{ab}=e^{2u}{J_{ab}}$ by \eqref{eq3.2}.

Let $\{\theta^b,\theta\}$ denote the special coframe. Noting that we do not change $\{W_a\}$, we require $\widehat{\theta^b}(W_a)=\delta_a^b$ and $\widehat\theta(\widehat{T})=0$. So $\theta^b$ change as:
  \begin{equation}\label{eqA.5}
  \widehat{\theta^a}=\theta^a-J_{\ b}^au^b\theta.\\
  \end{equation}
\begin{lemma}\label{lem3.3}
  If $u\in{\mathscr{O}_m}$, for a fixed special frame of the contact Riemannian manifold $(M,\theta,h,J)$, under the conformal transformation \eqref{eq3.1}, we have
  \begin{equation}\label{eq3.5}
  \widehat{A}_{\alpha\beta}=A_{\alpha\beta}-iZ_{\alpha}Z_{\beta}u+\mathscr{O}_{m-1},
  \end{equation}
  \begin{equation}\label{eq3.6}
  {\widehat{R}}_{\alpha\ \gamma{\bar{\beta}}}^{\ \gamma}=R_{\alpha\ \gamma\bar{\beta}}^{\ \gamma}-\frac{n+2}{2}(Z_{\alpha}Z_{\bar{\beta}}u+Z_{\bar{\beta}}Z_{\alpha}u)
  +\frac{1}{2}h_{\alpha\bar{\beta}}\mathscr{L}_0u+\mathscr{O}_{m-1},
  \end{equation}
  where we set $\mathscr{L}_0=-(Z_{\alpha}Z_{\bar{\alpha}}+Z_{\bar{\alpha}}Z_{\alpha})$.
  \end{lemma}
  This lemma will be proved in the Appendix A. On a contact Riemannian manifold, we have transformation formulae \eqref{eq3.5} and \eqref{eq3.6} under the conformal transformation similar to the CR case (cf. Lemma 3.6 in \cite{JL1}), but with error terms $\mathscr{O}_{m-1}$ instead of $\mathscr{O}_{m}$. We will see that it's sufficient for our purpose.
\begin{lemma}\label{lem3.3a}
Under the conformal transformation \eqref{eq3.1}, for $u\in{\mathscr{O}_m}$, $m\ge2$. The connection 1-form of the TWT connection changes as
$$\widehat\omega_a^b=\omega_a^b+\mathscr{O}_m.$$
\end{lemma}
\begin{proof}
First note that by \eqref{eqA.10} we have
\begin{equation}\label{eq3.6a}
\widehat\Gamma_{ca}^b=\Gamma_{ca}^b+\mathscr{O}_{m-1},{\quad}\widehat\Gamma_{\widehat{0}a}^b=\Gamma_{0a}^b+\mathscr{O}_{m-2}.
\end{equation}
By \eqref{eq3.1}, \eqref{eqA.5} and $e^{2u}=1+\mathscr{O}_m$, we get
$$\widehat\theta=(1+\mathscr{O}_m)\theta=\theta+\mathscr{O}_{m+2},{\quad}\widehat{\theta}^a=\theta^a+\mathscr{O}_{m+1}.$$
So
$$\widehat\omega_a^b=\widehat\Gamma_{ca}^b\widehat{\theta}^c+\widehat\Gamma_{\widehat{0}a}^b\widehat{\theta}
=\omega_a^b+\mathscr{O}_m.$$
\end{proof}
\subsection{The conformal contact form with vanishing $R_{\alpha\ \gamma\bar{\beta}}^{\ \gamma}(q)$ and $A_{\alpha\beta}(q)$}
As in the CR case (cf. p. 320 in \cite{JL1}),  we define the tensor $S_{ab}\theta^a\otimes\theta^b$, whose components are:
\[S_{\alpha\beta}=\overline{S_{\overline{\alpha\beta}}}=-(n+2)iA_{\alpha\beta}(q),{\quad}
S_{\alpha\bar{\beta}}=\overline{S_{\bar{\alpha}\beta}}=R_{\alpha\ \gamma\bar{\beta}}^{\ \gamma}(q).\]
\begin{proposition}
$S_{ab}$ is a symmetric tensor.
\end{proposition}
\begin{proof}
$S_{\alpha\beta}$ and $S_{\bar{\alpha}\bar{\beta}}$ is symmetric following directly from the self-adjointness of $A_{ab}$ (see Lemma \ref{lem2.1}). So we need to prove $S_{\alpha\bar{\beta}}=S_{\bar{\beta}\alpha}$.

Recall that
\begin{equation}\label{eq3.8a}
h(R(X,Y)Z,W)=h(R(W,Z)Y,X)+h\big((LW{\wedge}LZ)Y,X\big)-h\big((LX{\wedge}LY)Z,W\big),
\end{equation}
(cf. (38) in \cite{BD1}), for any vector field $X$, $Y$, $Z$, $W$, where $(X{\wedge}Y)Z=h(X,Z)Y-h(Y,Z)X$ and $L=J-\tau_{\ast}$. Now apply \eqref{eq3.8a} to $X=W_{\gamma}$, $Y=W_{\bar{\beta}}$, $Z=W_{\alpha}$, $W=W_{\bar{\mu}}$ to get
\begin{align}\label{eq3.8b}
\notag h(R(W_{\gamma},W_{\bar{\beta}})W_{\alpha},W_{\bar{\mu}})=&h(R(W_{\bar{\mu}},W_{\alpha})W_{\bar{\beta}},W_{\gamma})+h\big((LW_{\bar{\mu}}{\wedge}LW_{\alpha})W_{\bar{\beta}},W_{\gamma}\big)\\
&-h\big((LW_{\gamma}{\wedge}LW_{\bar{\beta}})W_{\alpha},W_{\bar{\mu}}\big).
\end{align}
On the other hand, by the definition of $L$, $h(LW_a,W_b)=h\big((J-\tau_{\ast})W_a,W_b\big)=h(J_{\ a}^cW_c-A_a^cW_c,W_b)=J_{ba}-A_{ab}$. Then we get
\begin{align}\label{eq3.8c}
\notag &h((LW_{\bar{\mu}}{\wedge}LW_{\alpha})W_{\bar{\beta}},W_{\gamma})|_q=h(h(LW_{\bar{\mu}},W_{\bar{\beta}})LW_{\alpha},W_{\gamma})|_q-h(h(LW_{\alpha},W_{\bar{\beta}})LW_{\bar{\mu}},W_{\gamma})|_q\\ \notag &=(J_{\bar{\beta}\bar{\mu}}-A_{\bar{\mu}\bar{\beta}})(J_{\gamma\alpha}-A_{\alpha\gamma})|_q-(J_{\bar{\beta}\alpha}-A_{\alpha\bar{\beta}})(J_{\gamma\bar{\mu}}-A_{\bar{\mu}\gamma})|_q\\
&=A_{\bar{\mu}\bar{\beta}}(q)A_{\alpha\gamma}(q)-J_{\bar{\beta}\alpha}(q)J_{\gamma\bar{\mu}}(q),
\end{align}
and
\begin{align}\label{eq3.8d}
\notag &h\big((LW_{\gamma}{\wedge}LW_{\bar{\beta}})W_{\alpha},W_{\bar{\mu}}\big)|_q=h(h(LW_{\gamma},W_{\alpha})LW_{\bar{\beta}},W_{\bar{\mu}})|_q-h(h(LW_{\bar{\beta}},W_{\alpha})LW_{\gamma},W_{\bar{\mu}})|_q\\ \notag &=(J_{\alpha\gamma}-A_{\gamma\alpha})|_q(J_{\bar{\mu}\bar{\beta}}-A_{\bar{\beta}\bar{\mu}})|_q-(J_{\alpha\bar{\beta}}-A_{\bar{\beta}\alpha})|_q(J_{\bar{\mu}\gamma}-A_{\gamma\bar{\mu}})|_q\\
&=A_{\gamma\alpha}(q)A_{\bar{\beta}\bar{\mu}}(q)-J_{\alpha\bar{\beta}}(q)J_{\bar{\mu}\gamma}(q)=A_{\bar{\mu}\bar{\beta}}(q)A_{\alpha\gamma}(q)-J_{\bar{\beta}\alpha}(q)J_{\gamma\bar{\mu}}(q),
\end{align}
at the point $q$ by using Proposition \ref{prop2.0d} and Corollary \ref{cor2.0a}. Substitute \eqref{eq3.8c} and \eqref{eq3.8d} to \eqref{eq3.8b} to get
$R_{\alpha\bar{\mu}\gamma\bar{\beta}}(q)=R_{\bar{\beta}\gamma\bar{\mu}\alpha}(q)$ at $q$. Hence $S_{\alpha\bar{\beta}}
=R_{\alpha\ \gamma\bar{\beta}}^{\ \gamma}(q)=h^{\gamma\bar{\mu}}R_{\alpha\bar{\mu}\gamma\bar{\beta}}(q)=R_{\alpha\bar{\gamma}\gamma\bar{\beta}}(q)=R_{\bar{\beta}\gamma\bar{\gamma}\alpha}
=h_{\gamma\bar{\mu}}R_{\bar{\beta}\ \bar{\gamma}\alpha}^{\ \bar{\mu}}(q)=R_{\bar{\beta}\ \bar{\gamma}\alpha}^{\ \bar{\gamma}}(q)=S_{\bar{\beta}\alpha}$. So tensor $\{S_{ab}\}$ is symmetric.
\end{proof}

{\it Proof of Theorem \ref{thm3.1}.} If $u=u(z)$ is a polynomial homogeneous of degree $m$ but independent of $t$, we denote $u\in\mathscr{R}_m$. We assume $u{\in}\mathscr{R}_2$ in the conformal transformation \eqref{eq3.1}.

  For the symmetric tensor $S_{ab}$ as above, we define the polynomial $$S=S_{ab}z^az^b.$$
  By Lemma \ref{lem3.3}, for $u\in\mathscr{R}_2$, we have
  \begin{align}
  \notag &\widehat{S}_{\alpha\beta}=S_{\alpha\beta}-(n+2)Z_{\alpha}Z_{\beta}u,\\
  \notag &\widehat{S}_{\alpha\bar{\beta}}=S_{\alpha\bar{\beta}}-\frac{n+2}{2}(Z_{\bar{\beta}}Z_{\alpha}u+Z_{\alpha}Z_{\bar{\beta}}u)
  +\frac{1}{2}h_{\alpha\bar{\beta}}\mathscr{L}_0u.
  \end{align}
  Now let $\widehat{S}=\widehat{S}_{ab}z^az^b$, we get
  \begin{align}\label{eq3.11}
  \notag \widehat{S}&=\widehat{S}_{ab}z^az^b=\bigg(S_{\alpha\beta}-(n+2)Z_{\alpha}Z_{\beta}u\bigg)z^{\alpha}z^{\beta}
  +\bigg(S_{\alpha\bar{\beta}}-\frac{n+2}{2}(Z_{\bar{\beta}}Z_{\alpha}u+Z_{\alpha}Z_{\bar{\beta}}u)
  +\frac{1}{2}{\delta}_{\alpha\bar{\beta}}\mathscr{L}_0u\bigg)z^{\alpha}z^{\bar{\beta}}\\
  \notag &+\bigg(S_{\bar{\alpha}\beta}-\frac{n+2}{2}(Z_{\beta}Z_{\bar{\alpha}}u+Z_{\bar{\alpha}}Z_{\beta}u)
  +\frac{1}{2}{\delta}_{\beta\bar{\alpha}}\mathscr{L}_0u\bigg)z^{\bar{\alpha}}z^{\beta}
  +\bigg(S_{\bar{\alpha}\bar{\beta}}-(n+2)Z_{\bar{\alpha}}Z_{\bar{\beta}}u\bigg)z^{\bar{\alpha}}z^{\bar{\beta}}\\
  &=S-(n+2)z^az^bZ_aZ_bu+|z|^2\mathscr{L}_0u.
  \end{align}
  Note that
  \begin{equation}
  \notag m^2u=P^2u=(z^aZ_a+2z^0Z_0)^2u=z^az^bZ_aZ_bu+4z^0z^aZ_0Z_au+4z^0z^0Z_0Z_0u+2z^0Z_0u+Pu,
  \end{equation}
  (cf. p. 320 in \cite{JL1}). Thus for $u\in{\mathscr{R}_2}$, we have
  \begin{equation}
  \notag z^az^bZ_aZ_bu=2u.
  \end{equation}
  Therefore by \eqref{eq3.11}, $\widehat{S}=S-2(n+2)u+|z|^2\mathscr{L}_0u$. The operator $-2(n+2)+|z|^2\mathscr{L}_0$ is invertible on $\mathscr{R}_2$ by ${|z|}^2\mathscr{L}_0$ having no positive eigenvalues (cf. Lemma 3.9 in \cite{JL1}). So we can find $u=u_0\in{\mathscr{R}_2}$ such that $\widehat{S}_{(2)}=0$. Namely $u_0$ satisfies $-2(n+2)u_0+|z|^2\mathscr{L}_0u_0=-S_{(2)}=-z^az^bS_{ab}(q)$. Under the conformal transformation $\widehat\theta=e^{2u_0}\theta$ for such $u_0$, we have
  $$\widehat{R}_{\alpha\ \gamma\bar{\beta}}^{\ \gamma}(q)=0,{\quad}\widehat{A}_{\alpha\beta}(q)=0,$$
  with respect to a special frame $\{W_a,\widehat{T}\}$ of $(M,\theta,h,J)$ centered at $q$. Finally, we have the following lemma.
  \begin{lemma}\label{lem3.5}
  Under the conformal transformation \eqref{eq3.1}  with $u\in{\mathscr{R}_2}$, changing the special frame $\{W_a\}$  of $(M,\theta,h,J)$ centered at $q$ to a special frame $\{\widehat{W_a}\}$ of $(M,\widehat{\theta},\widehat{h},\widehat{J})$ centered at $q$ makes the value of the curvature tensor and the Webster torsion tensor at $q$ invariant. So we can abuse the notation to write $\widehat{R}_{\alpha\ \gamma\bar{\beta}}^{\ \gamma}(q)$ and $\widehat{A}_{\alpha\beta}(q)$ no matter they are with respect to $\{W_a,\widehat{T}\}$ or $\{\widehat{W_a},\widehat{T}\}$.
  \end{lemma}
  \begin{proof}
  Since $\{W_a\}$ and $\{\widehat{W_a}\}$ are both horizontal, we write $\widehat{W_a}=v_a^bW_b$ for some invertible matrix $\{v_a^b\}$. The value of $\{W_a\}$ and $\{\widehat{W_a}\}$ at the point $q$ are decided by relation
  $$\widehat{h}(\widehat{W_{\alpha}},\widehat{W_{\bar{\beta}}})=\delta_{\alpha\bar{\beta}}=h(W_{\alpha},W_{\bar{\beta}}),
  {\quad}\widehat{h}(\widehat{W_{\alpha}},\widehat{W_{\beta}})=0=h(W_{\alpha},W_{\beta}),$$
  at $q$. By \eqref{eq3.2}, if $u\in\mathscr{R}_2$, we have $h_{ab}=h(W_a,W_b)=\widehat{h}(\widehat{W_a},\widehat{W_b})=(1+\mathscr{O}_2)v_a^cv_b^dh(W_c,W_d)=v_a^cv_b^d(1+\mathscr{O}_2)h_{cd}$.
  So the special frame satisfies $\widehat{W_a}=v_a^bW_b$ with $v_a^b=\delta_a^b+\mathscr{O}_1$.

  By the classical theory of the differential geometry, $R_{abcd}$ and $A_{ab}$ are covariant tensors. So with changing $W_a{\to}\widehat{W_a}=v_a^bW_b$, these components change as $R_{abcd}{\to}v_a^{a_1}v_b^{b_1}v_c^{c_1}v_d^{d_1}R_{a_1b_1c_1d_1}$ and $A_{ab}{\to}A_{a_1b_1}v_a^{a_1}v_b^{b_1}$ (it's shown in Appendix A.2). So their value at $q$ does not change since $v$ is the identity transformation at $q$.
  \end{proof}
   So we also have
   $$\widehat{R}_{\alpha\ \gamma\bar{\beta}}^{\ \gamma}(q)=0,{\quad}\widehat{A}_{\alpha\beta}(q)=0,$$
   with respect to a special frame of $(M,\widehat{\theta},\widehat{h},\widehat{J})$ centered at $q$.
  \section{The proof of the main theorem}
  \subsection{The asymptotic expansion of the Yamabe functional}
  \begin{lemma}\label{lem5.1}(cf. Theorem 11.3 in \cite{T1}.)
  For a contact Riemannian manifold $(M,\theta,h,J)$, the Yamabe functional $\mathscr{Y}_{\theta,h}(u)$ in \eqref{eq1.2} is invariant under the conformal transformation.
  \end{lemma}
  For a contact Riemannian manifold $(M,\theta,h,J)$, suppose that the almost structure $J$ is not integrable. There exists a point $q$ such that the Tanno tensor $Q(q)\ne0$. By Proposition \ref{prop4.2}, we must have $\mathfrak{Q}>0$. By Lemma \ref{lem5.1}, we can choose $(M,\widehat{\theta},\widehat{h},\widehat{J})$ conformal to $(M,\theta,h,J)$ such that the components of curvature and Webster torsion tensors satisfy Theorem \ref{thm3.1}. And $\mathfrak{Q}>0$ also holds with respect to $(M,\widehat{\theta},\widehat{h},\widehat{J})$. We denote this $(M,\widehat{\theta},\widehat{h},\widehat{J})$ as $(M,\theta,h,J)$ in this section.
  \begin{proposition}\label{prop5.1}
  We can choose $(M,\theta,h,J)$ in its conformal class such that with respect to a special frame $\{W_a,T\}$ of $(M,\theta,h,J)$ centered at $q$,
  \begin{align}
  \notag &A_{ab}(q)=0,{\quad}R_{\alpha\ \gamma\bar{\beta}}^{\ \gamma}(q)=0,
  {\quad}R_{\bar{\beta}\ \alpha\beta}^{\ \alpha}(q)=-\frac{1}{4}\mathfrak{Q},
  {\quad}R_{\alpha\ \beta\bar{\beta}}^{\ \alpha}(q)=\frac{1}{4}\mathfrak{Q}.
  \end{align}
  \end{proposition}
  \begin{proof}
  We only need to show the last identity since the others are given by Corollary \ref{cor2.0a}, Theorem \ref{thm3.1} and \eqref{eq4.13} in Proposition \ref{prop4.4}.
  Since we already have $A_{ab}(q)=0$, $R_{\beta\ \alpha\bar{\beta}}^{\ \alpha}(q)=0$ and $\Gamma_{jk}^l(q)=0$ (by \eqref{eq3.1b}), the Bianchi-type identity \eqref{eq4.20} at $q$ gives us
  \begin{align}
  \notag 0=-R_{\beta\ \alpha\bar{\mu}}^{\ \alpha}(q)+R_{\alpha\ \beta\bar{\mu}}^{\ \alpha}(q)+R_{\bar{\mu}\ \alpha\beta}^{\ \alpha}(q)=R_{\alpha\ \beta\bar{\mu}}^{\ \alpha}(q)+R_{\bar{\mu}\ \alpha\beta}^{\ \alpha}(q),
  \end{align}
  which implies $R_{\alpha\ \beta\bar{\beta}}^{\ \alpha}(q)=
  -R_{\bar{\beta}\ \alpha\beta}^{\ \alpha}(q)=\frac{1}{4}\mathfrak{Q}$.
  \end{proof}
Then we have the main theorem of this section.
  \begin{theorem}\label{thm5.1}
  For a contact Riemannian manifold $(M,\theta,h,J)$ such that $Q(q)\ne0$ for some point $q$. If we choose the normalized contact form and the special frame as Proposition 5.1, then \eqref{eq1.6} holds. In particular, there exists $\varepsilon>0$ such that $\mathscr{Y}_{\theta}(f^{\varepsilon})<\lambda(\mathscr{H}^n)$.
  \end{theorem}

  We write the volume form of the contact manifold $dV_{\theta}=(-1)^n\theta\wedge{d\theta}^n$ as
  \begin{equation}\label{eq5.6}
  dV_{\theta}=(v_0+v_1+v_2+\mathscr{O}_3)dV,
  \end{equation}
  where $v_j$ is a homogeneous polynomial of degree $j=0,1,2$ and $dV=(-1)^n\Theta{\wedge}d\Theta$. By $\theta_{(2)}=\Theta$ in \eqref{eq2.22} we find $v_0=1$.
  \begin{proposition}\label{prop5.2}
  On the contact Riemannian manifold $(M,\theta,h,J)$, we have the following expansion.
  \begin{equation}\label{eq5.7}
  \begin{aligned}
  \int_M|f^{\varepsilon}|^{p}dV_{\theta}&=a_0(n)+a_1(n)\varepsilon+a_2(n)\varepsilon^2+O(\varepsilon^3),\\
  \int_M{|df^{\varepsilon}|}_{H}^2dV_{\theta}&=b_0(n)+b_1(n)\varepsilon+b_2(n)\varepsilon^2+O(\varepsilon^3),\\
  \int_MR|f^{\varepsilon}|^2dV_{\theta}&=c_2(n)\varepsilon^2+O(\varepsilon^3),
  \end{aligned}
  \end{equation}
  where
  \begin{equation}\label{eq5.10}  \begin{aligned} &a_m(n)=\int_{\mathscr{H}^n}{|\Phi|}^pv_mdV,\\ &b_m(n)=2\int_{\mathscr{H}^n}v_m^{jk}{Z_j\Phi}{Z_k\Phi}dV,
  \\& c_2(n)=\int_{\mathscr{H}^n}R(q)|\Phi|^2dV.
    \end{aligned}\end{equation}
  $m=0,1,2$. $v_m$ is defined as \eqref{eq5.6} and
  \begin{equation}\label{eq5.10a}
  v_m^{jk}=\sum\limits_{\substack{m_0+m_1+m_2=m \\ m_i\ge0,\ \beta}}v_{m_0}s_{\beta(m_1+o(j)-1)}^js_{\bar{\beta}(m_2+o(k)-1)}^k.
  \end{equation}
  \end{proposition}
  \begin{proof}
  The estimates \eqref{eq5.7} is similar to the CR case (cf. Proposition 4.2 in \cite{JL1}), but the third identity of \eqref{eq5.7} is $O(\varepsilon^3)$ in the CR case with $R(q)=0$. We sketch the proof here. First note that if a function $|\varphi|\le{CF(\rho)}$, then
  $$\int_{a<\rho<b}{\varphi}dV=O\bigg(\int_a^bF(\rho)\rho^{2n+1}d\rho\bigg).$$
  If we replace $(z,t)$ by $\delta_{\varepsilon}(z,t)=(\varepsilon{z},\varepsilon^2t)$, we have $\delta_{\varepsilon}^{\ast}{\Phi}^{\varepsilon}={\varepsilon}^{-n}\Phi$, $\delta_{\varepsilon}^{\ast}(dV)={\varepsilon}^{2n+2}dV$. We also note that $\Phi{\le}C(1+\rho)^{-2n}$ (cf. p. 330 in \cite{JL1}). So
  \begin{align}
  \notag\int_M|f^{\varepsilon}|^pdV_{\theta}&=\int_{\mathscr{H}^n}|\psi|^p|\Phi^{\varepsilon}|^p(1+v_1+v_2+O(\rho^3))dV\\
  \notag &=\int_{\rho<\kappa/\varepsilon}|\Phi|^p\big(1+\varepsilon{v_1}+\varepsilon^2{v_2}+O(\varepsilon^3\rho^3)\big)dV
  +O\bigg(\int_{\kappa/\varepsilon<\rho<2\kappa/\varepsilon}|\Phi|^pdV\bigg)\\
  \notag &=\int_{\mathscr{H}^n}|\Phi|^p\big(1+\varepsilon{v_1}+\varepsilon^2{v_2}\big)dV
  +O\bigg(\int_{\kappa/\varepsilon}^{\infty}\sum\limits_{i=0}^2\varepsilon^i\rho^i(1+\rho)^{-4n-4}\rho^{2n+1}d\rho\bigg)\\
  \notag &\ \ +O\bigg(\int_{0}^{\kappa/\varepsilon}\varepsilon^3\rho^3(1+\rho)^{-4n-4}\rho^{2n+1}d\rho\bigg)
  +O\bigg(\int_{\kappa/\varepsilon}^{2\kappa/\varepsilon}(1+\rho)^{-4n-4}\rho^{2n+1}d\rho\bigg)\\
  \notag &=\int_{\mathscr{H}^n}|\Phi|^p\big(1+\varepsilon{v_1}+\varepsilon^2{v_2}\big)dV+O(\varepsilon^3),
  \end{align}
  for $n\ge2$. So we get the first identity in \eqref{eq5.7}. Noting that $|df^{\varepsilon}|_{H}^2={\langle}W_{\alpha}f^{\varepsilon}\theta^{\alpha}+W_{\bar{\alpha}}f^{\varepsilon}\theta^{\bar{\alpha}},
  W_{\beta}f^{\varepsilon}\theta^{\beta}+W_{\bar{\beta}}f^{\varepsilon}\theta^{\bar{\beta}}{\rangle}
  =h^{\alpha\bar{\beta}}W_{\alpha}f^{\varepsilon}W_{\bar{\beta}}f^{\varepsilon}
  +h^{\bar{\alpha}\beta}W_{\bar{\alpha}}f^{\varepsilon}W_{\beta}f^{\varepsilon}
  =2W_{\beta}f^{\varepsilon}W_{\bar{\beta}}f^{\varepsilon}$, we can write
  \begin{align}
  \notag \int_M|df^{\varepsilon}|_{H}^2dV_{\theta}&=2\int_MW_{\beta}f^{\varepsilon}W_{\bar{\beta}}f^{\varepsilon}dV_{\theta}\\
  \notag &=2\int_{\mathscr{H}^n}s_{\beta}^jZ_j(\psi\Phi^{\varepsilon})s_{\bar{\beta}}^kZ_k(\psi\Phi^{\varepsilon})(1+v_1+v_2+\cdots)dV\\
  \notag &=2\int_{\rho<\kappa}\big(v_0^{jk}+v_1^{jk}+v_2^{jk}+O(\rho^{1+o(jk)})\big)Z_j{\Phi}^{\varepsilon}Z_k{\Phi}^{\varepsilon}dV\\
  \notag &\ \ \ +O\bigg(\int_{\kappa<\rho<2\kappa}(|Z_j\Phi^{\varepsilon}||Z_k\Phi^{\varepsilon}|+|Z_j\Phi^{\varepsilon}||\Phi^{\varepsilon}|+|\Phi^{\varepsilon}|^2)dV\bigg),
  \end{align}
  by setting $v_m^{jk}$ as \eqref{eq5.10a}, which is a homogeneous polynomial of degree $m+o(jk)-2$.  And by noting that $\delta_{\varepsilon}^{\ast}(Z_j\Phi^{\varepsilon})=\varepsilon^{-n-o(j)}Z_j\Phi$ and $|Z_j\Phi|{\le}C(1+\rho)^{-2n-o(j)}$, then we have
  \begin{align}
  \notag &\int_M|df^{\varepsilon}|_{H}^2dV_{\theta}=2\int_{\rho<\kappa/\varepsilon}\sum\limits_{m=0}^2\varepsilon^mv_m^{jk}Z_j{\Phi}Z_k{\Phi}dV
  +O\bigg(\int_0^{\kappa/\varepsilon}\sum\limits_{i=2}^4\varepsilon^3\rho^{1+i}(1+\rho)^{-4n-i}\rho^{2n+1}d\rho\bigg)\\
  \notag &\ \ \ \ \ \ \ \ \ \ \ \ \ \ \ \ \ \ \ \ \ +O\bigg(\int_{\kappa/\varepsilon}^{2\kappa/\varepsilon}\sum\limits_{i=0}^4\varepsilon^{2-i}(1+\rho)^{-4n-i}\rho^{2n+1}d{\rho}\bigg)\\
  \notag &=2\int_{\mathscr{H}^n}\sum\limits_{m=0}^2\varepsilon^mv_m^{jk}Z_j{\Phi}Z_k{\Phi}dV
  +O\bigg(\int_{\kappa/\varepsilon}^{\infty}\sum\limits_{m=0}^2\sum\limits_{i=2}^4\varepsilon^m\rho^{m+i-2}(1+\rho)^{-4n-i}\rho^{2n+1}d\rho\bigg)
  +O(\varepsilon^3)\\
  \notag &=2\int_{\mathscr{H}^n}\sum\limits_{m=0}^2\varepsilon^mv_m^{jk}Z_j{\Phi}Z_k{\Phi}dV+O(\varepsilon^3).
  \end{align}
  So we get the second identity in \eqref{eq5.7}. The third identity in \eqref{eq5.7} follows from
  \begin{align}
  \notag \int_MR|f^{\varepsilon}|^2dV_{\theta}=\bigg(\int_{\mathscr{H}^n}R(q)|\Phi|^2dV\bigg)\varepsilon^2+O(\varepsilon^3),
  \end{align}
  (cf. p. 332 in \cite{JL1}).
  \end{proof}
  Note that the volume form of the Heisenberg group $dV=(-1)^n\Theta\wedge{d\Theta^n}$ can be written as
  \begin{align}\label{eq5.11}
  \notag dV&=(-1)^n\Theta\wedge{d\Theta^n}=(-1)^n\Theta\wedge{d\Theta^n}=(-1)^ndt\wedge(-2idz^{\alpha}{\wedge}dz^{\bar{\alpha}})^n\\
  \notag &=2^nn!dt{\wedge}(idz^1{\wedge}dz^{\bar{1}}){\wedge}\cdots{\wedge}(idz^n{\wedge}dz^{\bar{n}})=4^nn!dt\wedge{dx^1}\wedge{dy^1}\wedge\cdots\wedge{dx^n}\wedge{dy^n}\\
  &=4^nn!dtd\mu(z)=4^nn!r^{2n-1}d\nu(\zeta)drdt,
  \end{align}
  where $d\mu(z)$ is the Lebesgue measure on $\mathbb{C}^n$,
  and $d\nu$ is the uniform measure on $S^{2n-1}=\{z\in\mathbb{C}^n:|z|=1\}$, normalized so that if $z=r\zeta$, $\zeta\in{S^{2n-1}}$, $d\mu(z)=r^{2n-1}drd\nu(\zeta)$.

  To calculate $a_m(n),b_m(n),m=0,1,2$ explicitly, we need the following lemmas.
  \begin{lemma}\label{lem5.2}
  If a real two-form $\omega=m_{\alpha\beta}dz^{\alpha}\wedge{dz^{\beta}}+2im_{\alpha\bar{\beta}}dz^{\alpha}\wedge{dz^{\bar{\beta}}}
  +m_{\bar{\alpha}\bar{\beta}}dz^{\bar{\alpha}}\wedge{dz^{\bar{\beta}}}$, then
  \begin{equation}
  \begin{aligned}
  \notag &n\Theta\wedge\omega\wedge{d\Theta^{n-1}}=-\delta^{\alpha\bar{\beta}}m_{\alpha\bar{\beta}}\Theta\wedge{d\Theta^n},\\
  \notag &n(n-1)\Theta\wedge\omega^2{\wedge}d\Theta^{n-2}=\bigg((\delta^{\alpha\bar{\beta}}\delta^{\rho\bar{\sigma}}
  -\delta^{\alpha\bar{\sigma}}\delta^{\rho\bar{\beta}})m_{\alpha\bar{\beta}}m_{\rho\bar{\sigma}}
  +\frac{1}{2}(\delta^{\alpha\bar{\rho}}\delta^{\beta\bar{\sigma}}
  -\delta^{\alpha\bar{\sigma}}\delta^{\beta\bar{\rho}})m_{\alpha\beta}m_{\bar{\rho}\bar{\sigma}}\bigg)\Theta{\wedge}d\Theta^n.
  \end{aligned}
  \end{equation}
  \end{lemma}
  \begin{proof}
  This is essentially Lemma 5.1 in \cite{JL1}. To avoid confusion, we will not use the summation convention in the proof of this Lemma. We can calculate that
  \begin{align}
  \notag n\Theta{\wedge}\omega{\wedge}&d\Theta^{n-1}=ndt{\wedge}\bigg(2im_{\alpha\bar{\beta}}dz^{\alpha}\wedge{dz^{\bar{\beta}}}\bigg)
  \wedge\bigg(-2i\sum\limits_{\gamma}dz^{\gamma}{\wedge}dz^{\bar{\gamma}}\bigg)^{n-1}\\
  \notag &=(-2i)^nn!dt{\wedge}\bigg(-\sum\limits_{\alpha,\beta}m_{\alpha\bar{\beta}}dz^{\alpha}\wedge{dz^{\bar{\beta}}}\bigg)
  \wedge\bigg(\sum\limits_{\gamma}dz^1{\wedge}dz^{\bar{1}}\wedge\cdots\wedge\widehat{dz^{\gamma}}
  \wedge\widehat{dz^{\bar{\gamma}}}\wedge\cdots{\wedge}dz^n{\wedge}dz^{\bar{n}}\bigg)\\
  \notag &=(-2i)^nn!\sum\limits_{\alpha}(-m_{\alpha\bar{\alpha}})dt{\wedge}dz^1{\wedge}dz^{\bar{1}}\wedge\cdots{\wedge}dz^n{\wedge}dz^{\bar{n}}
  =-\delta^{\alpha\bar{\beta}}m_{\alpha\bar{\beta}}\Theta\wedge{d\Theta^n}.
  \end{align}
  Here $\widehat{dz^{\gamma}}$ means the exterior derivative has no $dz^{\gamma}$ terms. And for the second identity, we can prove in the same way as the second identity in Lemma 5.1 in \cite{JL1}.
  \end{proof}

  \begin{corollary}\label{cor5.1}
  As defined by \eqref{eq5.6},
  \begin{equation}\label{eqB.7a}
  v_1=0;{\quad}v_2=-\frac{1}{6}R_{\bar{\beta}\ \alpha\mu}^{\ \alpha}(q)z^{\bar{\beta}}z^{\mu}-\frac{1}{6}R_{\beta\ \bar{\alpha}\bar{\mu}}^{\ \bar{\alpha}}(q)z^{\beta}z^{\bar{\mu}},{\quad}mod{\quad}z^{\beta}z^{\mu},\ z^{\bar{\beta}}z^{\bar{\mu}}.
  \end{equation}
  \end{corollary}
  \begin{proof}
  By the definition of $v_1$,
  $$(dV_{\theta})_{(2n+3)}=(-1)^n(\theta{\wedge}d\theta^n)_{(2n+3)}
  =(-1)^n\bigg(\theta_{(3)}{\wedge}d\Theta^n+n\Theta{\wedge}(d\theta)_{(3)}{\wedge}d\Theta^{n-1}\bigg)
  =(-1)^nv_1\Theta{\wedge}d\Theta^n.$$
  By \eqref{eq2.22}, $\theta_{(3)}$ has no $dt$ term and so $\theta_{(3)}{\wedge}d\Theta^n$ vanishes. Note that \eqref{SE}, \eqref{eq2.22} and Proposition \ref{prop4.1} lead to
\begin{equation}\label{eq5.11b}
(d\theta)_{(3)}=J_{\alpha\beta(1)}dz^{\alpha}{\wedge}dz^{\beta}+J_{\bar{\alpha}\bar{\beta}(1)}dz^{\bar{\alpha}}{\wedge}dz^{\bar{\beta}},
\end{equation}
which has no $dz^{\alpha}{\wedge}dz^{\bar{\beta}}$ term. So by Lemma \ref{lem5.2}, $\Theta{\wedge}(d\theta)_{(3)}{\wedge}d\Theta^{n-1}=0$. Hence $v_1=0$.

By the definition of $v_2$ in \eqref{eq5.6}, we see that
\begin{align}
\notag (dV_{\theta})_{(2n+4)}&=(-1)^n(\theta\wedge{d\theta}^n)_{(2n+4)}\\
\notag &=(-1)^n\bigg(\theta_{(4)}\wedge{d\Theta^n}+n\Theta\wedge{(d\theta)}_{(4)}\wedge{d\Theta^{n-1}}\\
\notag &\ \ \ +n\theta_{(3)}{\wedge}(d\theta)_{(3)}{\wedge}d\Theta^{n-1}+\frac{n(n-1)}{2}\Theta\wedge\big((d\theta)_{(3)}\big)^2{\wedge}d\Theta^{n-2}\bigg)\\
\notag &=v_2dV=(-1)^nv_2\Theta\wedge{d\Theta}^n.
\end{align}
By \eqref{eq2.22} and $A_{ab}(q)=0$, $\theta_{(4)}$ has no $dt$ term, and so the first term of the right hand side vanishes. By \eqref{eq2.22} and \eqref{eq5.11b}, $\theta_{(3)}$ and $(d\theta)_{(3)}$ have no $dt$ term. Hence the third term of the right hand side vanishes. Now apply Lemma \ref{lem5.2} to $\omega=(d\theta)_{(3)}$ in \eqref{eq5.11b} to get
\begin{align}
\notag \frac{n(n-1)}{2}\Theta\wedge\big((d\theta)_{(3)}\big)^2{\wedge}d\Theta^{n-2}&
=\frac{1}{4}(J_{\alpha\beta(1)}J_{\bar{\alpha}\bar{\beta}(1)}-J_{\alpha\beta(1)}J_{\bar{\beta}\bar{\alpha}(1)})\Theta{\wedge}d\Theta^n\\
\notag &=\frac{1}{2}J_{\alpha\beta(1)}J_{\bar{\alpha}\bar{\beta}(1)}\Theta{\wedge}d\Theta^n
=\frac{1}{2}Q_{\beta\gamma}^{\bar{\alpha}}(q)Q_{\bar{\beta}\bar{\mu}}^{\alpha}(q)z^{\gamma}z^{\bar{\mu}}\Theta{\wedge}d\Theta^n,
\end{align}
by $J_{ab}=-J_{ba}$ in \eqref{eq2.0} and \eqref{eq4.9}. Noting that by \eqref{SE}, $J_{\alpha\beta}(q)=0$ by
\eqref{eq2.18} and $\theta_{(2)}^a=0$ by \eqref{eq2.22}, we have
  \begin{align}
  \notag (d\theta)_{(4)}&=(J_{\alpha\beta}\theta^{\alpha}\wedge\theta^{\beta}
  +2J_{\alpha\bar{\beta}}\theta^{\alpha}\wedge\theta^{\bar{\beta}}
  +J_{\bar{\alpha}\bar{\beta}}\theta^{\bar{\alpha}}\wedge\theta^{\bar{\beta}})_{(4)}\\
  \notag &=2J_{\alpha\bar{\beta}(2)}dz^{\alpha}\wedge{dz^{\bar{\beta}}}
  +2J_{\alpha\bar{\beta}}(q)\theta^{\alpha}_{(3)}\wedge{dz^{\bar{\beta}}}
  +2J_{\alpha\bar{\beta}}(q)dz^{\alpha}\wedge\theta^{\bar{\beta}}_{(3)}\\
  \notag
  &=2J_{\alpha\bar{\beta}(2)}dz^{\alpha}\wedge{dz^{\bar{\beta}}}-\frac{i}{3}R_{b\ c\lambda}^{\ \alpha}(q)z^bz^cdz^{\lambda}\wedge{dz^{\bar{\alpha}}}-\frac{i}{3}R_{b\ c\bar{\lambda}}^{\ \bar{\alpha}}(q)z^bz^cdz^{\alpha}\wedge{dz^{\bar{\lambda}}},\\
  \notag &=2i\bigg(-iJ_{\alpha\bar{\beta}(2)}dz^{\alpha}\wedge{dz^{\bar{\beta}}}+\frac{1}{6}R_{\beta\ \lambda\bar{\mu}}^{\ \alpha}(q)z^{\beta}z^{\bar{\mu}}dz^{\lambda}\wedge{dz^{\bar{\alpha}}}+\frac{1}{6}R_{\bar{\beta}\ \lambda\mu}^{\ \alpha}(q)z^{\bar{\beta}}z^{\mu}dz^{\lambda}\wedge{dz^{\bar{\alpha}}}\\
  \notag &\ \ \ +\frac{1}{6}R_{\beta\ \bar{\lambda}\bar{\mu}}^{\ \bar{\alpha}}(q)z^{\beta}z^{\bar{\mu}}dz^{\alpha}\wedge{dz^{\bar{\lambda}}}+\frac{1}{6}R_{\bar{\beta}\ \bar{\lambda}\mu}^{\ \bar{\alpha}}(q)z^{\bar{\beta}}z^{\mu}dz^{\alpha}\wedge{dz^{\bar{\lambda}}}\bigg),\\
  \notag &\ \ \ mod{\quad} dz^{\alpha}\wedge{dz}^{\beta},{\quad}dz^{\bar{\alpha}}\wedge{dz}^{\bar{\beta}},
  {\quad}z^{\alpha}z^{\beta},{\quad}z^{\bar{\alpha}}z^{\bar{\beta}},
  \end{align}
  by using Corollary \ref{cor2.1}. Now Apply Lemma \ref{lem5.2} to $\omega=(d\theta)_{(4)}$ to get
  \begin{align}
  \notag &\ \ \ n\Theta\wedge{(d\theta)}_{(4)}\wedge{d\Theta^{n-1}}\\
  \notag &=\bigg(iJ_{\alpha\bar{\alpha}(2)}-\frac{1}{6}R_{\bar{\beta}\ \alpha\mu}^{\ \alpha}(q)z^{\bar{\beta}}z^{\mu}-\frac{1}{6}R_{\beta\ \bar{\alpha}\bar{\mu}}^{\ \bar{\alpha}}(q)z^{\beta}z^{\bar{\mu}}\bigg)\Theta\wedge{d\Theta^n},\quad{mod}{\quad}z^{\alpha}z^{\beta},\ z^{\bar{\alpha}}z^{\bar{\beta}}\\
  \notag &=\bigg(-\frac{1}{2}Q_{\alpha\beta}^{\bar{\sigma}}(q)Q_{\bar{\alpha}\bar{\mu}}^{\sigma}(q)z^{\beta}z^{\bar{\mu}}-\frac{1}{6}R_{\bar{\beta}\ \alpha\mu}^{\ \alpha}(q)z^{\bar{\beta}}z^{\mu}-\frac{1}{6}R_{\beta\ \bar{\alpha}\bar{\mu}}^{\ \bar{\alpha}}(q)z^{\beta}z^{\bar{\mu}}\bigg)\Theta\wedge{d\Theta^n},\\
  \notag &\ \ \ {mod}{\quad}z^{\alpha}z^{\beta},\ z^{\bar{\alpha}}z^{\bar{\beta}}.
  \end{align}
  Here we have used \eqref{eq4.10} for $J_{\alpha\bar{\alpha}(2)}$ and Proposition \ref{prop5.1} for
  $R_{\beta\ \alpha\bar{\mu}}^{\ \alpha}(q)=R_{\bar{\beta}\ \bar{\alpha}\mu}^{\ \bar{\alpha}}(q)=0$.
  So we conclude that
  \begin{align}
  \notag v_2\Theta{\wedge}d\Theta^n&=n\Theta\wedge{(d\theta)}_{(4)}\wedge{d\Theta^{n-1}}+\frac{n(n-1)}{2}\Theta\wedge\big((d\theta)_{(3)}\big)^2{\wedge}d\Theta^{n-2}\\
  \notag &=\bigg(-\frac{1}{6}R_{\bar{\beta}\ \alpha\mu}^{\ \alpha}(q)z^{\bar{\beta}}z^{\mu}-\frac{1}{6}R_{\beta\ \bar{\alpha}\bar{\mu}}^{\ \bar{\alpha}}(q)z^{\beta}z^{\bar{\mu}}\bigg)\Theta{\wedge}d\Theta^n,{\quad}mod{\quad}z^{\beta}z^{\mu},\ z^{\bar{\beta}}z^{\bar{\mu}}.
  \end{align}
  We finish the proof of this corollary.
  \end{proof}
\subsection{Calculation of some integrals}
  \begin{lemma}\label{lem5.3}(cf. Proposition 5.3 in \cite{JL1})
  Let $A=(\alpha_1,\cdots,\alpha_m)$, $B=(\beta_1,\cdots,\beta_m)$ be the multi-indices with $1\le\alpha_i,\beta_i\le{n}$, let $\delta(A,B)=1$ if $A=B$ and 0 otherwise. Then
  \begin{align}
  \notag \int_{S^{2n-1}}z^{\alpha_1}\cdots{z^{\alpha_m}}z^{\bar{\beta}_1}\cdots{z^{\bar{\beta}_m}}d\nu
  =\frac{2\pi^n}{(n+m-1)!}\sum\limits_{\sigma\in{S_m}}\delta(A,\sigma{B}).
  \end{align}
  \end{lemma}
  Lemma \ref{lem5.3} leads to following corollary.
  \begin{corollary}\label{corB.1}
  If $F$ is a function of $r$ and $t$, denote $\mathscr{F}_m=\int_{-\infty}^{\infty}\int_0^{\infty}F(r,t)r^mdrdt$. We have
  \begin{equation}\label{eqB.1}
  \begin{aligned}
  \int_{\mathscr{H}^n}Q_{\alpha\beta}^{\bar{\gamma}}(q)Q_{\bar{\alpha}\bar{\mu}}^{\gamma}(q)z^{\beta}z^{\bar{\mu}}FdV
  &=2(4\pi)^n\mathfrak{Q}\mathscr{F}_{2n+1},\\
  \int_{\mathscr{H}^n}R_{\bar{\beta}\ \alpha\mu}^{\ \alpha}(q)z^{\bar{\beta}}z^{\mu}FdV&=\int_{\mathscr{H}^n}R_{\beta\ \bar{\alpha}\bar{\mu}}^{\ \bar{\alpha}}(q)z^{\beta}z^{\bar{\mu}}FdV=-\frac{(4\pi)^n}{2}\mathfrak{Q}\mathscr{F}_{2n+1},\\
  \int_{\mathscr{H}^n}R_{\rho\ \gamma\bar{\lambda}}^{\ \alpha}(q)z^{\rho}z^{\bar{\alpha}}z^{\gamma}z^{\bar{\lambda}}FdV
  &=\int_{\mathscr{H}^n}R_{\bar{\rho}\ \bar{\gamma}\lambda}^{\ \bar{\alpha}}(q)z^{\bar{\rho}}z^{\alpha}z^{\bar{\gamma}}z^{\lambda}FdV=\frac{(4\pi)^n}{2(n+1)}\mathfrak{Q}\mathscr{F}_{2n+3},\\
  \int_{\mathscr{H}^n}R_{\bar{\rho}\ \lambda\gamma}^{\ \alpha}(q)z^{\bar{\rho}}z^{\bar{\alpha}}z^{\lambda}z^{\gamma}FdV&=\int_{\mathscr{H}^n}R_{\rho\ \bar{\lambda}\bar{\alpha}}^{\ \bar{\gamma}}(q)z^{\rho}z^{\gamma}z^{\bar{\lambda}}z^{\bar{\alpha}}FdV=0,\\
  \int_{\mathscr{H}^n}Q_{\alpha\lambda}^{\bar{\gamma}}(q)Q_{\bar{\beta}\bar{\mu}}^{\gamma}(q)z^{\alpha}z^{\lambda}z^{\bar{\beta}}z^{\bar{\mu}}FdV
  &=\frac{3(4\pi)^n}{n+1}\mathfrak{Q}\mathscr{F}_{2n+3},\\
  \int_{\mathscr{H}^n}J_{\alpha\bar{\beta}(2)}z^{\alpha}z^{\bar{\beta}}FdV&=-\int_{\mathscr{H}^n}J_{\bar{\beta}\alpha(2)}z^{\alpha}z^{\bar{\beta}}FdV
  =\frac{3i(4\pi)^n}{2(n+1)}\mathfrak{Q}\mathscr{F}_{2n+3},\\
  \int_{\mathscr{H}^n}J_{\alpha\beta(2)}z^{\alpha}z^{\beta}FdV&=\int_{\mathscr{H}^n}J_{\bar{\alpha}\bar{\beta}(2)}z^{\bar{\alpha}}z^{\bar{\beta}}FdV=0.
  \end{aligned}
  \end{equation}
  \end{corollary}
  \begin{proof}
  First note that we have \eqref{eq5.11} for $dV$. By using Lemma \ref{lem5.3} for $m=1$, we get
  \begin{align}
  \notag \int_{\mathscr{H}^n}Q_{\alpha\beta}^{\bar{\gamma}}(q)Q_{\bar{\alpha}\bar{\mu}}^{\gamma}(q)z^{\beta}z^{\bar{\mu}}FdV
  &=\int_{-\infty}^{\infty}\int_{0}^{\infty}Q_{\alpha\beta}^{\bar{\gamma}}(q)Q_{\bar{\alpha}\bar{\mu}}^{\gamma}(q)\bigg(\int_{S^{2n-1}}z^{\beta}z^{\bar{\mu}}d{\nu}\bigg)4^nn!r^{2n+1}F(r,t)drdt\\
  \notag &=2(4\pi)^n\mathfrak{Q}\mathscr{F}_{2n+1}.
  \end{align}
  The identities of the second line of \eqref{eqB.1} follows similarly by noting that $R_{\bar{\beta}\ \alpha\beta}^{\ \alpha}(q)=-\frac{1}{4}\mathfrak{Q}$ in Proposition \ref{prop5.1}.

  Let $m=2$ in Lemma \ref{lem5.3}. First we have
  \begin{align}
  \notag \int_{\mathscr{H}^n}R_{\rho\ \gamma\bar{\lambda}}^{\ \alpha}(q)z^{\rho}z^{\bar{\alpha}}z^{\gamma}z^{\bar{\lambda}}FdV&=\int_{-\infty}^{\infty}\int_{0}^{\infty}R_{\rho\ \gamma\bar{\lambda}}^{\ \alpha}(q)\bigg(\int_{S^{2n-1}}z^{\rho}z^{\bar{\alpha}}z^{\gamma}z^{\bar{\lambda}}d{\nu}\bigg)4^nn!r^{2n+3}F(r,t)drdt\\
  \notag &=2\frac{(4\pi)^n}{n+1}\big(R_{\rho\ \alpha\bar{\rho}}^{\ \alpha}(q)+R_{\rho\ \gamma\bar{\gamma}}^{\ \rho}(q)\big)\int_{-\infty}^{\infty}\int_0^{\infty}F(r,t)r^{2n+3}drdt\\
  \notag &=\frac{(4\pi)^n}{2(n+1)}\mathfrak{Q}\mathscr{F}_{2n+3},
  \end{align}
  by $R_{\rho\ \alpha\bar{\rho}}^{\ \alpha}(q)=0$ and $R_{\rho\ \gamma\bar{\gamma}}^{\ \rho}(q)=\frac{\mathfrak{Q}}{4}$ in Proposition \ref{prop5.1}. And $\int_{\mathscr{H}^n}R_{\bar{\rho}\ \bar{\gamma}\lambda}^{\ \bar{\alpha}}(q)z^{\bar{\rho}}z^{\alpha}z^{\bar{\gamma}}z^{\lambda}FdV$ follows similarly or by taking conjugation. Similarly, we have
  \begin{align}
  \notag &\int_{\mathscr{H}^n}R_{\bar{\rho}\ \lambda\gamma}^{\ \alpha}(q)z^{\bar{\rho}}z^{\bar{\alpha}}z^{\lambda}z^{\gamma}FdV=\frac{2(4\pi)^n}{(n+1)}\bigg(R_{\bar{\rho}\ \alpha\rho}^{\ \alpha}(q)+R_{\bar{\rho}\ \rho\alpha}^{\ \alpha}(q)\bigg)\int_{-\infty}^{\infty}\int_0^{\infty}F(r,t)r^{2n+3}drdt=0,
  \end{align}
  by $R_{a\ cd}^{\ b}=-R_{a\ dc}^{\ b}$. The second identity of the fourth line in \eqref{eqB.1} follows from taking conjugation. And also we have
\begin{align}
\notag &\int_{\mathscr{H}^n}Q_{\alpha\lambda}^{\bar{\gamma}}(q)Q_{\bar{\beta}\bar{\mu}}^{\gamma}(q)z^{\alpha}z^{\lambda}z^{\bar{\beta}}z^{\bar{\mu}}FdV\\
\notag &=\frac{2(4\pi)^n}{n+1}\bigg(Q_{\alpha\lambda}^{\bar{\gamma}}(q)Q_{\bar{\alpha}\bar{\lambda}}^{\gamma}(q)
+Q_{\alpha\lambda}^{\bar{\gamma}}(q)Q_{\bar{\lambda}\bar{\alpha}}^{\gamma}(q)\bigg)
\int_{-\infty}^{\infty}\int_0^{\infty}F(r,t)r^{2n+3}drdt
=\frac{3(4\pi)^n}{n+1}\mathfrak{Q}\mathscr{F}_{2n+3}.
\end{align}
by \eqref{eq4.2} and the last identity in \eqref{Q}. By \eqref{eq4.10}, we get
\begin{align}
\notag &\int_{\mathscr{H}^n}J_{\alpha\bar{\beta}(2)}z^{\alpha}z^{\bar{\beta}}FdV=\frac{i}{2}\int_{\mathscr{H}^n}Q_{\alpha\lambda}^{\bar{\gamma}}(q)Q_{\bar{\beta}\bar{\mu}}^{\gamma}(q)z^{\alpha}z^{\lambda}z^{\bar{\beta}}z^{\bar{\mu}}FdV
=\frac{3(4\pi)^ni}{2(n+1)}\mathfrak{Q}\mathscr{F}_{2n+3}.
\end{align}
And by Proposition \ref{prop4.3a}, we get
  \begin{align}
  \notag &\int_{\mathscr{H}^n}J_{\alpha\beta(2)}z^{\alpha}z^{\beta}FdV=\frac{1}{2}\int_{\mathscr{H}^n}z^{\alpha}z^{\beta}z^cz^dZ_cZ_dJ_{\alpha\beta}(q)FdV=\frac{1}{2}\int_{\mathscr{H}^n}z^{\alpha}z^{\beta}z^{\bar{\rho}}z^{\bar{\mu}}Z_{\bar{\rho}}Z_{\bar{\mu}}J_{\alpha\beta}(q)FdV\\
  \notag &=\frac{(4\pi)^n}{n+1}\big(Z_{\bar{\beta}}Z_{\bar{\alpha}}J_{\alpha\beta}(q)+Z_{\bar{\alpha}}Z_{\bar{\beta}}J_{\alpha\beta}(q)\big)\int_{-\infty}^{\infty}\int_0^{\infty}F(r,t)r^{2n+3}drdt=0.
  \end{align}
  The last identity follows from the the anti-symmetry of $J_{\alpha\beta}$. Taking conjugation, we get $\int_{\mathscr{H}^n}J_{\bar{\alpha}\bar{\beta}(2)}z^{\bar{\alpha}}z^{\bar{\beta}}FdV=0$.
\end{proof}
\begin{lemma}\label{lem5.4}(cf. Lemma 5.5 in \cite{JL1})
  Suppose that $\alpha,\gamma+1,\beta+1$ and $\alpha-\gamma-1$ are positive real numbers. If $2\alpha-2\gamma-\beta>3$, then
  \[\int_{-\infty}^{\infty}\int_0^{\infty}|t+i(1+r^2)|^{-\alpha}r^{\beta}|t|^{\gamma}drdt=N_1(\alpha,\beta,\gamma),\]
  where
  \begin{equation}
  \notag N_1(\alpha,\beta,\gamma)=\frac{\Gamma\bigg(\frac{1}{2}(\beta+1)\bigg)\Gamma\bigg(\alpha-\gamma-\frac{1}{2}\beta-\frac{3}{2}\bigg)
  \Gamma\bigg(\frac{1}{2}(\gamma+1)\bigg)\Gamma\bigg(\frac{1}{2}(\alpha-\gamma-1)\bigg)}{2\Gamma(\alpha-\gamma-1)\Gamma(\frac{\alpha}{2})}.
  \end{equation}
  \end{lemma}
  By the expression of $N_1(\alpha,\beta,\gamma)$ above, we get
  $$N_1(2n,2n-1,0)=\frac{\Gamma(n)\Gamma(n-1)\Gamma(\frac{1}{2})\Gamma(\frac{2n-1}{2})}{2\Gamma(2n-1)\Gamma(n)}=\frac{\Gamma(n-1)\Gamma(\frac{1}{2})\Gamma(\frac{2n-1}{2})}{2\Gamma(2n-1)},$$
  $$N_1(2n+2,2n-1,0)=\frac{\Gamma(n)\Gamma(n+1)\Gamma(\frac{1}{2})\Gamma(\frac{2n+1}{2})}{2\Gamma(2n+1)\Gamma(n+1)}=\frac{\Gamma(n)\Gamma(\frac{1}{2})\Gamma(\frac{2n+1}{2})}{2\Gamma(2n+1)},$$
  $$N_1(2n+2,2n+1,0)=\frac{\Gamma(n+1)\Gamma(n)\Gamma(\frac{1}{2})\Gamma(\frac{2n+1}{2})}{2\Gamma(2n+1)\Gamma(n+1)}=\frac{\Gamma(n)\Gamma(\frac{1}{2})\Gamma(\frac{2n+1}{2})}{2\Gamma(2n+1)},$$
  $$N_1(2n+2,2n+3,0)=\frac{\Gamma(n+2)\Gamma(n-1)\Gamma(\frac{1}{2})\Gamma(\frac{2n+1}{2})}{2\Gamma(2n+1)\Gamma(n+1)},$$
  $$N_1(2n+4,2n+3,2)=\frac{\Gamma(n+2)\Gamma(n-1)\Gamma(\frac{3}{2})\Gamma(\frac{2n+1}{2})}{2\Gamma(2n+1)\Gamma(n+2)}=\frac{\Gamma(n-1)\Gamma(\frac{3}{2})\Gamma(\frac{2n+1}{2})}{2\Gamma(2n+1)}.$$
  So we can find that
  \begin{equation}\label{N1}
  \begin{aligned}
  &\frac{N_1(2n,2n-1,0)}{N_1(2n+2,2n+1,0)}=\frac{4n}{n-1},\\
  &\frac{N_1(2n+2,2n+1,0)}{N_1(2n+2,2n-1,0)}=1,\\
  &\frac{N_1(2n+2,2n+3,0)}{N_1(2n+2,2n+1,0)}=\frac{n+1}{n-1},\\
  &\frac{N_1(2n+4,2n+3,2)}{N_1(2n+2,2n+1,0)}=\frac{1}{2(n-1)}.
  \end{aligned}
  \end{equation}
  \subsection{Calculation of constants $a_m(n)$ and $b_m(n)$}
  \begin{lemma}
  For $Z_j$ given by \eqref{eq2.10}, we have
  \begin{equation}\label{eqB.17}
  \begin{aligned}
  &Z_{\alpha}\Phi=inz^{\bar{\alpha}}\frac{t+i(|z|^2+1)}{|w+i|^{n+2}},\\
  &Z_{\bar{\alpha}}\Phi=-inz^{\alpha}\frac{t-i(|z|^2+1)}{|w+i|^{n+2}},\\
  &Z_0\Phi=-nt\frac{1}{|w+i|^{n+2}}.
  \end{aligned}
  \end{equation}
  \end{lemma}
  \begin{proof}
  By $|w+i|^{-n}=(t^2+(|z|^2+1)^2)^{-\frac{n}{2}}$, we have
  \begin{align}
  \notag \frac{\partial}{\partial{t}}(|w+i|^{-n})=-nt|w+i|^{-n-2},
  \end{align}
  and
  \begin{align}
  \notag \frac{\partial}{\partial{z}^{\alpha}}(|w+i|^{-n})&
  =-\frac{n}{2}|w+i|^{-n-2}\cdot2(|z|^2+1)z^{\bar{\alpha}}=-nz^{\bar{\alpha}}|w+i|^{-n-2}(|z|^2+1).
  \end{align}
  The result follows.
  \end{proof}
  Note that we have
  $$\int_{S^{2n-1}}d\nu=\frac{2\pi^n}{(n-1)!},$$
  by the case $m=0$ in Lemma \ref{lem5.3} and by \eqref{N1} we have
  \begin{equation}\label{eq5.12a}
  N_1(2n+2,2n+1,0)=N_1(2n+2,2n-1,0)=\frac{4^{-n}\pi}{2n},
  \end{equation}
  (cf. p. 341 in \cite{JL1} for the second identity). Hence by \eqref{eq5.10}, \eqref{eq5.11} and $v_0=1$ we get
  \begin{equation}\label{eq5.13a}
  \begin{aligned}
  a_0(n)&=\int_{\mathscr{H}^n}{|\Phi|}^pdV
  =4^nn!\int_{-\infty}^{\infty}\int_{0}^{\infty}\int_{S^{2n-1}}d{\nu}\frac{r^{2n-1}}{|t+i(1+r^2)|^{2n+2}}drdt\\
  &=(4\pi)^n(2n)N_1(2n+2,2n-1,0)=\pi^{n+1}.
  \end{aligned}
  \end{equation}
  And we see that
  \begin{align}
  \notag v_0^{jk}Z_j{\Phi}Z_k{\Phi}&=\sum\limits_{m_0=m_1=m_2=0}s_{\beta(m_1+o(j)-1)}^js_{\bar{\beta}(m_2+o(k)-1)}^kv_{m_0}{Z_j\Phi}{Z_k\Phi}\\
  &=s_{\beta(0)}^{\alpha}s_{\bar{\beta}(0)}^{\bar{\gamma}}{Z_{\alpha}\Phi}{Z_{\bar{\gamma}}\Phi}=Z_{\beta}{\Phi}Z_{\bar{\beta}}{\Phi}, \end{align}
  by \eqref{eq2.29} for $s_{\beta(0)}^{\alpha}=\delta_{\beta}^{\alpha}$ and $s_{b(1)}^0=0$.
  So we get
  \begin{equation}\label{eq5.16}
  \begin{aligned}
  b_0(n)&=2\int_{\mathscr{H}^n}v_0^{jk}Z_j{\Phi}Z_k{\Phi}dV=2\int_{\mathscr{H}^n}Z_{\beta}{\Phi}Z_{\bar{\beta}}{\Phi}dV
  =2\int_{\mathscr{H}^n}n^2|z|^2|\omega+i|^{-2n-2}dV\\
  &=2n^24^nn!\int_{-\infty}^{\infty}\int_{0}^{\infty}\int_{S^{2n-1}}d{\nu}\frac{r^{2n+1}}{|t+i(1+r^2)|^{2n+2}}drdt\\
  &=4n^3(4\pi)^nN_1(2n+2,2n+1,0)=2n^2\pi^{n+1},
  \end{aligned}
  \end{equation}
  by \eqref{eqB.17}. By \eqref{eq5.10} and Corollary \ref{cor5.1}, we have
  \begin{equation}\label{eq5.14a}
  a_1(n)=\int_{\mathscr{H}^n}{|\Phi|}^pv_1dV=0.
  \end{equation}
  By Proposition \ref{prop2.2} for $s_b^j$ and $v_1=0$ in \eqref{eqB.7a}, we can get
  \begin{align}
  \notag v_1^{jk}{Z_j\Phi}{Z_k\Phi}&=\sum\limits_{m_0+m_1+m_2=1}s_{\beta(m_1+o(j)-1)}^js_{\bar{\beta}(m_2+o(k)-1)}^kv_{m_0}{Z_j\Phi}{Z_k\Phi}\\
  \notag &=s_{\beta(0)}^{\alpha}s_{\bar{\beta}(2)}^0Z_{\alpha}{\Phi}Z_0{\Phi}+s_{\beta(2)}^{0}s_{\bar{\beta}(0)}^{\bar{\gamma}}Z_0{\Phi}Z_{\bar{\gamma}}{\Phi}
  =z^az^bz^cF(r,t),
  \end{align}
  for some functions $F(r,t)$ only depended with $r$ and $t$. So
  \begin{equation}\label{eq5.17}
  b_1(n)=2\int_{-\infty}^{\infty}\int_0^{\infty}\int_{S_{2n-1}}z^az^bz^cd{\nu}F(r,t)drdt=0,
  \end{equation}
  by Lemma \ref{lem5.3}. $a_2(n)$ is given by following lemma.
  \begin{lemma}\label{lem5.5}
  \begin{equation}\label{eq5.15}
  a_2(n)=\frac{\pi^{n+1}}{12n}\mathfrak{Q}.
  \end{equation}
  \end{lemma}
  \begin{proof}
  By \eqref{eq5.10}, \eqref{eqB.7a}, the second line of \eqref{eqB.1} and Lemma \ref{lem5.4}, we have
  \begin{align}
  \notag &a_2(n)=\int_{\mathscr{H}^n}{|\Phi|}^pv_2dV=\int_{\mathscr{H}^n}\bigg(-\frac{1}{6}R_{\bar{\beta}\ \alpha\mu}^{\ \alpha}(q)z^{\bar{\beta}}z^{\mu}
  -\frac{1}{6}R_{\beta\ \bar{\alpha}\bar{\mu}}^{\ \bar{\alpha}}(q)z^{\beta}z^{\bar{\mu}}\bigg)|w+i|^{-2n-2}dV\\
  \notag &=\frac{1}{6}(4\pi)^n\mathfrak{Q}\int_0^{\infty}\int_{-\infty}^{\infty}|t+i(1+r^2)|^{-2n-2}r^{2n+1}drdt\\
  \notag &=\frac{1}{6}(4\pi)^n\mathfrak{Q}N_1(2n+2,2n+1,0)=\frac{\pi^{n+1}}{12n}\mathfrak{Q}.
  \end{align}
  We finish the proof of Lemma \ref{lem5.5}.
  \end{proof}
  To calculate $b_2(n)$, we need the following results.
  \begin{lemma}\label{lem5.6}
  \begin{equation}\label{eq5.18}
  \begin{aligned}
  \int_{\mathscr{H}^n}v_2^{ab}Z_a{\Phi}Z_b{\Phi}dV=\frac{n^4+2n^3+2n^2}{6(n-1)(n+1)}(4\pi)^nN_1(2n+2,2n+1,0)\mathfrak{Q},\\
  \int_{\mathscr{H}^n}\bigg(v_2^{a0}Z_a{\Phi}Z_0{\Phi}+v_2^{0a}Z_0{\Phi}Z_a{\Phi}\bigg)dV=\frac{-5n^2}{6(n-1)(n+1)}(4\pi)^nN_1(2n+2,2n+1,0)\mathfrak{Q},\\
  \int_{\mathscr{H}^n}v_2^{00}Z_0{\Phi}Z_0{\Phi}dV=\frac{2n^2}{3(n-1)(n+1)}(4\pi)^nN_1(2n+2,2n+1,0)\mathfrak{Q}.
  \end{aligned}
  \end{equation}
  \end{lemma}
  This lemma will be proved in Appendix B, from which we get
  \begin{align}\label{eq5.21}
  b_2(n)=2\int_{\mathscr{H}^n}v_2^{jk}Z_j\Phi{Z_k\Phi}dV=\frac{n^2(n+1)}{3(n-1)}(4\pi)^nN_1(2n+2,2n+1,0)\mathfrak{Q}
  =\frac{n(n+1)\pi^{n+1}}{6(n-1)}\mathfrak{Q}.
  \end{align}
  We also have
  \begin{lemma}\label{lem5.7}
  \begin{equation}
  c_2(n)=-\frac{2n\pi^{n+1}}{n-1}\mathfrak{Q}.
  \end{equation}
  \end{lemma}
  \begin{proof}
  Applying $X=W_0$, $Y=W_{\bar{\mu}}$ and taking index $a=\beta$ in the last identity of \eqref{SE}, we get
  $R_{\beta\ 0\bar{\mu}}^{\ 0}=\theta(R(W_0,W_{\bar{\mu}})W_{\beta})=0$. Hence by definition we have
  $R_{\beta\bar{\mu}}=R_{\beta\ \alpha\bar{\mu}}^{\ \alpha}+R_{\beta\ \bar{\alpha}\bar{\mu}}^{\ \bar{\alpha}}$, and so we get
  \begin{align}
  \notag R(q)&=h^{jk}R_{jk}(q)=h^{\beta\bar{\mu}}R_{\beta\bar{\mu}}(q)+h^{\bar{\beta}\mu}R_{\bar{\beta}\mu}(q)\\
  \notag &=\delta^{\beta\bar{\mu}}\big(R_{\beta\ \alpha\bar{\mu}}^{\ \alpha}(q)
  +R_{\beta\ \bar{\alpha}\bar{\mu}}^{\ \bar{\alpha}}(q)\big)
  +\delta^{\bar{\beta}\mu}\big(R_{\bar{\beta}\ \alpha\mu}^{\ \alpha}(q)
  +R_{\bar{\beta}\ \bar{\alpha}\mu}^{\ \bar{\alpha}}(q)\big)\\
  \notag &=R_{\beta\ \bar{\alpha}\bar{\beta}}^{\ \bar{\alpha}}(q)+R_{\bar{\beta}\ \alpha\beta}^{\ \alpha}(q)=-\frac{1}{2}\mathfrak{Q},
  \end{align}
 by Proposition \ref{prop5.1}. Noting that
 \begin{align}
\notag \int_{\mathscr{H}^n}|\Phi|^{2}dV&=\int_{\mathscr{H}^n}|w+i|^{-2n}dV
=\int_{S^{2n-1}}d\nu\int_{0}^{\infty}\int_{-\infty}^{\infty}\frac{4^nn!r^{2n-1}}{|t+i(1+r^2)|^{2n}}drdt\\
\notag &=2n(4\pi)^nN_1(2n,2n-1,0)=\frac{8n^2}{n-1}(4\pi)^nN_1(2n+2,2n+1,0),
\end{align}
  then by \eqref{eq5.10}, we get
  \begin{equation}
  \notag c_2(n)=-\frac{\mathfrak{Q}}{2}\int_{\mathscr{H}^n}|\Phi|^2dV
  =-\frac{4n^2}{n-1}(4\pi)^nN_1(2n+2,2n+1,0)\mathfrak{Q}=-\frac{2n\pi^{n+1}}{n-1}\mathfrak{Q}.
  \end{equation}
  \end{proof}
\begin{proposition}\label{YFH}
  The extremal of the Yamabe functional on the Heisenberg group is
  $$\lambda(\mathscr{H}^n)=2pn^2\pi.$$
  \end{proposition}
  \begin{proof}
  Recall that on $\mathscr{H}^n$, $\Theta=dt-iz^{\alpha}dz^{\bar{\alpha}}+iz^{\bar{\alpha}}dz^{\alpha}$. Then $\Theta=-2idz^{\alpha}{\wedge}dz^{\bar{\alpha}}$. So
  $-i\delta_{\alpha\bar{\beta}}=d\Theta(Z_{\alpha},Z_{\bar{\beta}})=h(Z_{\alpha},JZ_{\bar{\beta}})
  =-ih(Z_{\alpha},Z_{\bar{\beta}})$, namely $h(Z_{\alpha},Z_{\bar{\beta}})=\delta_{\alpha\bar{\beta}}$.
  Hence it induces the dual norm $|\cdot,\cdot|_{H}$ by ${\langle}\Theta^{\alpha},\Theta^{\bar{\beta}}{\rangle}_{H}=\delta^{\alpha\bar{\beta}}$.
  Then $|d\Phi|_{H}^2={\langle}Z_{\alpha}\Phi\Theta^{\alpha}+Z_{\bar{\alpha}}\Phi\Theta^{\bar{\alpha}},
  Z_{\beta}\Phi\Theta^{\beta}+Z_{\bar{\beta}}\Phi\Theta^{\bar{\beta}}{\rangle}_{H}=2Z_{\beta}{\Phi}Z_{\bar{\beta}}{\Phi}$.
  Since the curvature tensor $R\equiv0$ on the Heisenberg group, by definition we have
  \begin{align}
  \lambda(\mathscr{H}^n)=\frac{\int_{\mathscr{H}^n}p|d\Phi|_{H}^2dV}{(\int_{\mathscr{H}^n}{|\Phi|}^pdV)^{\frac{2}{p}}}
  =\frac{2p\int_{\mathscr{H}^n}Z_{\beta}{\Phi}Z_{\bar{\beta}}{\Phi}dV}{(\int_{\mathscr{H}^n}{|\Phi|}^pdV)^{\frac{2}{p}}}
  =2pn^2\pi,
  \end{align}
  by \eqref{eq5.13a} and \eqref{eq5.16}.
  \end{proof}
  \begin{remark}
  Recall that in \cite{JL1}, Jerison and Lee used the structure equation $d\theta=ih_{\alpha\bar{\beta}}\theta^{\alpha}\wedge\theta^{\bar{\beta}}$ (cf. p. 307 in \cite{JL1}). So in Heisenberg case in \cite{JL1}, $d\Theta=ih_{\alpha\bar{\beta}}\Theta^{\alpha}\wedge\Theta^{\bar{\beta}}=ih_{\alpha\bar{\beta}}dz^{\alpha}{\wedge}dz^{\bar{\beta}}$. On the other hand, $\Theta=dt+iz^{\alpha}dz^{\bar{\alpha}}-iz^{\bar{\alpha}}dz^{\alpha}$ leads to $d\Theta=2idz^{\alpha}{\wedge}dz^{\bar{\alpha}}$, which shows $h_{\alpha\bar{\beta}}=2\delta_{\alpha\bar{\beta}}$ on the Heisenberg group in \cite{JL1}. So   the Yamabe functional here in Proposition \ref{YFH} differs from \cite{JL1} by a factor $2$.
  \end{remark}
{\it Proof of Theorem \ref{thm5.1}.}
  Substituting \eqref{eq5.13a}, \eqref{eq5.14a} and \eqref{eq5.15} to the first identity in \eqref{eq5.7}, we get
  \begin{equation}
  \notag \int_M|f^{\varepsilon}|^pdV_{\theta}=\pi^{n+1}\bigg(1+\frac{1}{12n}\mathfrak{Q}\varepsilon^2\bigg)+O(\varepsilon^3),
  \end{equation}
  and so
  \begin{equation}\label{eq5.22}
  \bigg(\int_M|f^{\varepsilon}|^pdV_{\theta}\bigg)^{-\frac{2}{p}}
  =\pi^{-n}\bigg(1-\frac{1}{12(n+1)}\mathfrak{Q}\varepsilon^2\bigg)+O(\varepsilon^3).
  \end{equation}
Substituting \eqref{eq5.16}, \eqref{eq5.17} and \eqref{eq5.21} to the second identity in \eqref{eq5.7} leads to
\begin{align}\label{eq5.23} \int_Mp{|df^{\varepsilon}|}_{H}^2dV_{\theta}
=2pn^2\pi^{n+1}\bigg(1+\frac{n+1}{12(n-1)n}\mathfrak{Q}\varepsilon^2\bigg)+O(\varepsilon^3).
\end{align}
By the third identity in \eqref{eq5.7} and Lemma \ref{lem5.7}, we get
  \begin{align}\label{eq5.24}
  \int_MR|f^{\varepsilon}|^2dV_{\theta}&=2pn^2\pi^{n+1}\frac{-1}{2(n-1)(n+1)}\mathfrak{Q}\varepsilon^2+O(\varepsilon^3).
  \end{align}
  Finally, substituting \eqref{eq5.22}, \eqref{eq5.23} and \eqref{eq5.24} to the definition of $\mathscr{Y}_{\theta}(f^{\varepsilon})$ in \eqref{eq1.2}, we get
  \begin{align}
  \notag \mathscr{Y}_{\theta}(f^{\varepsilon})&=\bigg(\int_M|f^{\varepsilon}|^pdV_{\theta}\bigg)^{-\frac{2}{p}}
  \bigg(\int_Mp{|df^{\varepsilon}|}_{H}^2dV_{\theta}+\int_MR|f^{\varepsilon}|^2dV_{\theta}\bigg)\\
  \notag &=2pn^2\pi\bigg(1+\frac{-n(n-1)+(n+1)^2-6n}{12(n-1)n(n+1)}\mathfrak{Q}\varepsilon^2\bigg)+O(\varepsilon^3)\\
  \notag &=2pn^2\pi\bigg(1-\frac{3n-1}{12(n-1)n(n+1)}\mathfrak{Q}\varepsilon^2\bigg)+O(\varepsilon^3).
  \end{align}
  As $\mathfrak{Q}>0$ and $n>2$, Theorem 5.1 is proved.

  It remains to show the existence of the extremal. Note that on a contact Riemannian manifold, the Folland-Stein normal coordinates in \cite{FS1} also exist by the same argument (the integrability of $J$ is not used here), we can get the same estimates in Theorem 4.3 in \cite{JL2}. Then the existence of the extremal can be proved in same way as Section 5 and Section 6 in \cite{JL2} by using the Folland-Stein normal coordinates.
  \begin{appendix}
  \section{The transformation formulae}
  In this appendix, we will discuss the conformal transformations and prove Lemma \ref{lem3.3}. Recall that here we don't change $\{W_a\}\in{HM}$. Our tensors after conformal transformation is with respect to $\{W_a,\widehat{T}\}$ and $\{\widehat{\theta}^a,\widehat\theta\}$ satisfying \eqref{eq3.1}-\eqref{eqA.5}, e.g.
  $$\tau(\widehat{T},W_a)=\widehat{A}_a^bW_b,{\quad}\Gamma_{\widehat{0}b}^c=\omega_b^c(\widehat{T}).$$
  First we will discuss how the connection coefficients, the curvature tensor and the Webster torsion tensor change under a conformal transformation. The idea in this appendix follows from the proof of Lemma 10 in \cite{BD1} with a local $T^{(1,0)}M$-frame.
  \subsection{The transformation formulae of the connection coefficients under the conformal transformations}
  \begin{lemma}\label{lemA.0}
  We have
  \begin{align}\label{eqA.8}
  \notag 2h(\nabla_XY,Z)=&X(h(Y,Z))+Y(h(X,Z))-Z(h(X,Y))+\\
  \notag &-2h(X,JZ)\theta(Y)-2h(Y,JZ)\theta(X)+2h(X,JY)\theta(Z)+\\
  &-h([X,Z],Y)-h([Y,Z],X)+h([X,Y],Z),
  \end{align}
  for any $X,Y,Z\in{TM}$. And also we have
  \begin{equation}\label{eqA.8a}
  2h(\nabla_TY,Z)=T(h(Y,Z))-h([T,Z],Y)+h([T,Y],Z).
  \end{equation}
  for any $Y,Z{\in}HM$.
  \end{lemma}
  \begin{proof}
  We refer to p. 334 in \cite{BD1} for \eqref{eqA.8}.
  For \eqref{eqA.8a}, we have
  \begin{align}
  \notag T(h(Y,Z))&=h(\nabla_TY,Z)+h(Y,\nabla_TZ)=h(\nabla_TY,Z)+h(Y,[T,Z])+h(Y,\tau_{\ast}Z)\\
  \notag &=h(\nabla_TY,Z)+h(Y,[T,Z])+h(\tau_{\ast}Y,Z)=2h(\nabla_TY,Z)+h([T,Z],Y)-h([T,Y],Z),
  \end{align}
  by $\nabla{T}=0$, the definition of the Webster torsion $\tau_{\ast}$ and its self-adjointness (see Lemma \ref{lem2.1}).
  \end{proof}
\begin{corollary}\label{lemA.1}
  With respect to any frame $\{W_a,T\}$ with $\{W_a\}$ horizontal, we have
  \begin{align}\label{eqA.6}
  \notag \Gamma_{ab}^c=&\frac{1}{2}h^{cd}\Big(W_a(h_{bd})+W_b(h_{ad})-W_d(h_{ab})\\
  &\ \ \ \ \ \ -h([W_a,W_d],W_b)-h([W_b,W_d],W_a)+h([W_a,W_b],W_d)\Big),
  \end{align}
  and
  \begin{align}\label{eqA.7}
  \Gamma_{0b}^c=\frac{1}{2}h^{cd}\Big(T(h_{bd})-h([T,W_d],W_b)+h([T,W_b],W_d)\Big).
  \end{align}
  \end{corollary}
  \begin{proof}
  \eqref{eqA.6} follows by substituting $X=W_a$, $Y=W_b$, $Z=W_d$ into \eqref{eqA.8}. \eqref{eqA.7} follows by substituting $Y=W_b$, $Z=W_d$ into \eqref{eqA.8a}.
\end{proof}
\begin{lemma}\label{lemA.2}
Under the conformal transformation \eqref{eq3.1}, if $u\in{\mathscr{O}_m}$, we have
\begin{equation}\label{eqA.9}
[\widehat{T},W_{\beta}]=[T,W_{\beta}]-iZ_{\beta}Z_{\bar{\alpha}}uW_{\alpha}+iZ_{\beta}Z_{\alpha}uW_{\bar{\alpha}}+\mathscr{E}_{m-1}(W),
\end{equation}
where $\mathscr{O}_{m-1}\mathscr{E}(W)$ denote the linear combination of $W_j$'s with coefficients $\mathscr{O}_{m-1}$.
\end{lemma}
\begin{proof}
We have
\begin{align}
\notag [\widehat{T},W_{\beta}]&=[e^{-2u}(T+J_{\ a}^cu^aW_c),W_{\beta}]\\
\notag &=e^{-2u}[T,W_{\beta}]+e^{-2u}[J_{\ a}^cu^aW_c,W_{\beta}]+2e^{-2u}u_{\beta}(T+J_{\ a}^cu^aW_c)\\
\notag &=[T,W_{\beta}]-(W_{\beta}u^a)J_{\ a}^cW_c+\mathscr{O}_{m-1}\mathscr{E}(W)\\
\notag &=[T,W_{\beta}]-(W_{\beta}u_{\bar{\mu}})h^{\alpha\bar{\mu}}J_{\ \alpha}^{\rho}(q)W_{\rho}-(W_{\beta}u_{\mu})h^{\mu\bar{\alpha}}J_{\ \bar{\alpha}}^{\bar{\rho}}(q)W_{\bar{\rho}}+\mathscr{E}_{m-1}(W)\\
\notag &=[T,W_{\beta}]-iZ_{\beta}Z_{\bar{\alpha}}uW_{\alpha}+iZ_{\beta}Z_{\alpha}uW_{\bar{\alpha}}+\mathscr{E}_{m-1}(W),
\end{align}
by $h_{\alpha\bar{\beta}}=\delta_{\alpha\bar{\beta}}$, $J_{\ \alpha}^{\rho}(q)=i\delta_{\ \alpha}^{\rho}$ in \eqref{eq2.18} and $u_a{\in}\mathscr{O}_{m-1}$ for $u\in\mathscr{O}_m$. \eqref{eqA.9} follows.
\end{proof}
\begin{proposition}\label{propA.1}
  Under the conformal transformation \eqref{eq3.1}, the connection coefficients of the TWT connection change as
  \begin{equation}\label{eqA.10}
  \begin{aligned}
  \widehat{\Gamma}_{ab}^c&=\Gamma_{ab}^c+u_a\delta_b^c+u_b\delta_a^c-u^ch_{ab},\\
  \widehat{\Gamma}_{\widehat{0}\beta}^{\rho}&=\Gamma_{0\beta}^{\rho}+u_0\delta_{\beta}^{\rho}
  -\frac{i}{2}(Z_{\bar{\rho}}Z_{\beta}u+Z_{\beta}Z_{\bar{\rho}}u)+\mathscr{O}_{m-1},
  \end{aligned}
  \end{equation}
  where $u^c=h^{cd}u_d$.
  \end{proposition}
\begin{proof}
  By Lemma \ref{lemA.1}, we get
  \begin{align}
  \notag \widehat{\Gamma}_{ab}^c&=\frac{1}{2}e^{-2u}h^{cd}\bigg(W_a\bigg(e^{2u}h_{bd}\bigg)+W_b\bigg(e^{2u}h_{ad}\bigg)-W_d\bigg(e^{2u}h_{ab}\bigg)\\
  \notag &\ \ \ -e^{2u}h([W_a,W_d],W_b)-e^{2u}h([W_b,W_d],W_a)+e^{2u}h([W_a,W_b],W_d)\bigg)\\
  \notag &=\Gamma_{ab}^c+u_a\delta_b^c+u_b\delta_a^c-u^ch_{ab}.
  \end{align}

Note that
  \begin{equation}\label{eqA.12}
  \widehat{\Gamma}_{\widehat{0}\beta}^{\rho}=\frac{1}{2}\widehat{h}^{\rho\bar{\mu}}\Big(\widehat{T}(\widehat{h}_{\beta\bar{\mu}})
  -\widehat{h}([\widehat{T},W_{\bar{\mu}}],W_{\beta})+\widehat{h}([\widehat{T},W_{\beta}],W_{\bar{\mu}})\Big),
  \end{equation}
For the first term in the right hand side of \eqref{eqA.12}, according to \eqref{eq3.2}, we have
  \begin{align}
  \notag \frac{1}{2}\widehat{h}^{\rho\bar{\mu}}\widehat{T}(\widehat{h}_{\beta\bar{\mu}})&=\frac{1}{2}e^{-4u}h^{\rho\bar{\mu}}(T+J_{\ a}^eu^aW_e)(e^{2u}h_{\beta\bar{\mu}})=\frac{1}{2}h^{\rho\bar{\mu}}T(h_{\beta\bar{\mu}})+u_0\delta_{\beta}^{\rho}+\mathscr{O}_{m-1}.
  \end{align}
Take conjugation on both sides of \eqref{eqA.9} to get
$[\widehat{T},W_{\bar{\mu}}]=[T,W_{\bar{\mu}}]+iZ_{\bar{\mu}}Z_{\alpha}uW_{\bar{\alpha}}-iZ_{\bar{\mu}}Z_{\bar{\alpha}}uW_{\alpha}+\mathscr{E}_{m-1}(W).$
So for the second and third terms of \eqref{eqA.12}, we have
\begin{align}
\notag -\frac{1}{2}\widehat{h}^{\rho\bar{\mu}}\widehat{h}([\widehat{T},W_{\bar{\mu}}],W_{\beta})
&=-\frac{1}{2}h^{\rho\bar{\mu}}h\bigg([T,W_{\bar{\mu}}]+iZ_{\bar{\mu}}Z_{\alpha}uW_{\bar{\alpha}},W_{\beta}\bigg)+\mathscr{O}_{m-1}\\
\notag &=-\frac{1}{2}h^{\rho\bar{\mu}}h([T,W_{\bar{\mu}}],W_{\beta})-\frac{i}{2}Z_{\bar{\rho}}Z_{\beta}u+\mathscr{O}_{m-1},
\end{align}
and
\begin{align}
\notag \frac{1}{2}\widehat{h}^{\rho\bar{\mu}}\widehat{h}([\widehat{T},W_{\beta}],W_{\bar{\mu}})
&=\frac{1}{2}h^{\rho\bar{\mu}}h\bigg([T,W_{\beta}]-iZ_{\beta}Z_{\bar{\alpha}}uW_{\alpha},W_{\bar{\mu}}\bigg)+\mathscr{O}_{m-1}\\
\notag &=\frac{1}{2}h^{\rho\bar{\mu}}h([T,W_{\beta}],W_{\bar{\mu}})-\frac{i}{2}Z_{\beta}Z_{\bar{\rho}}u+\mathscr{O}_{m-1}.
\end{align}
So \eqref{eqA.12} becomes $\widehat{\Gamma}_{\widehat{0}\beta}^{\rho}=\Gamma_{0\beta}^{\rho}+u_0\delta_{\beta}^{\rho}
-\frac{i}{2}(Z_{\bar{\rho}}Z_{\beta}u+Z_{\beta}Z_{\bar{\rho}}u)+\mathscr{O}_{m-1}$.
\end{proof}
\subsection{The transformation formulae of the curvature and Webster torsion tensors under the conformal transformations}
{\it Proof of Lemma \ref{lem3.3}.} By $\nabla{T}=0$ and $\tau_{\ast}W_a=\tau(T,W_a)=\nabla_TW_a-[T,W_a]$, we get
\begin{align}
\notag A_{ab}&=h(A_a^cW_c,W_b)=h(\tau_{\ast}W_a,W_b)=h(\nabla_TW_a-[T,W_a],W_b)\\
\notag &=T(h_{ab})-h(W_a,\nabla_TW_b)-h([T,W_a],W_b)\\
\notag &=T(h_{ab})-h(W_a,\tau_{\ast}W_b+[T,W_b])-h([T,W_a],W_b)\\
\notag &=T(h_{ab})-A_{ba}-h([W_a,[T,W_b])-h([T,W_a],W_b).
\end{align}
Since the tensor $A$ is self-adjoint by Lemma \ref{lem2.1}, we get
$$A_{ab}=\frac{1}{2}\bigg(T(h_{ab})-h([W_a,[T,W_b])-h([T,W_a],W_b)\bigg).$$
In particular,
$\notag A_{\alpha\beta}=-\frac{1}{2}\big(h(W_{\alpha},[T,W_{\beta}]))+h([T,W_{\alpha}],W_{\beta})\big)$.
Applying Lemma \ref{lemA.2} with the frame $\{W_a,\widehat{T}\}$, we get
\begin{align}
\notag \widehat{A}_{\alpha\beta}&=-\frac{1}{2}\bigg(\widehat{h}(W_{\alpha},[\widehat{T},W_{\beta}]))+\widehat{h}([\widehat{T},W_{\alpha}],W_{\beta})\bigg)
=A_{\alpha\beta}-\frac{i}{2}Z_{\alpha}Z_{\beta}u-\frac{i}{2}Z_{\beta}Z_{\alpha}u+\mathscr{O}_{m-1}\\
\notag &=A_{\alpha\beta}-iZ_{\alpha}Z_{\beta}u+\mathscr{O}_{m-1},
\end{align}
by \eqref{eqA.9} and $[Z_{\alpha},Z_{\beta}]=0$. And \eqref{eq2.14} with frame $\{W_a,\widehat{T}\}$, we get
\begin{align}\label{eqA.14}
\widehat{R}_{\alpha\ \gamma\bar{\beta}}^{\ \gamma}=W_{\gamma}\widehat{\Gamma}_{\bar{\beta}\alpha}^{\ \gamma}-W_{\bar{\beta}}\widehat{\Gamma}_{\gamma\alpha}^{\gamma}
-\widehat{\Gamma}_{\gamma\bar{\beta}}^e\widehat{\Gamma}_{e\alpha}^{\gamma}
+\widehat{\Gamma}_{\bar{\beta}\gamma}^e\widehat{\Gamma}_{e\alpha}^{\gamma}
-\widehat{\Gamma}_{\gamma\alpha}^e\widehat{\Gamma}_{\bar{\beta}e}^{\gamma}
+\widehat{\Gamma}_{\bar{\beta}\alpha}^e\widehat{\Gamma}_{{\gamma}e}^{\gamma}
+2\widehat{\Gamma}_{\widehat{0}\alpha}^{\gamma}\widehat{J}_{\gamma\bar{\beta}}.
\end{align}
By the first identity in \eqref{eqA.10} and \eqref{eq2.18}, for $u\in\mathscr{O}_m$, we have
\begin{align}
\notag W_{\gamma}\widehat{\Gamma}_{\bar{\beta}\alpha}^{\ \gamma}&=W_\gamma\bigg(\Gamma_{\bar{\beta}\alpha}^{\ \gamma}+u_{\bar{\beta}}\delta_{\alpha}^{\gamma}-u^{\gamma}h_{\alpha\bar{\beta}}\bigg)=W_\gamma\Gamma_{\bar{\beta}\alpha}^{\ \gamma}+W_{\gamma}(u_{\bar{\beta}})\delta_{\alpha}^{\gamma}-h^{\gamma\bar{\mu}}W_\gamma(u_{\bar{\mu}})h_{\alpha\bar{\beta}}\\
&=W_\gamma\Gamma_{\bar{\beta}\alpha}^{\ \gamma}
+Z_{\alpha}Z_{\bar{\beta}}u-\delta_{\alpha\bar{\beta}}Z_{\gamma}Z_{\bar{\gamma}}u+\mathscr{O}_{m-1},
\end{align}
and
\begin{align}
\notag W_{\bar{\beta}}\widehat{\Gamma}_{\gamma\alpha}^{\gamma}&=
W_{\bar{\beta}}\bigg({\Gamma}_{\gamma\alpha}^{\gamma}
+u_{\gamma}\delta_{\alpha}^{\gamma}+u_{\alpha}\delta_{\gamma}^{\gamma}\bigg)
=W_{\bar{\beta}}{\Gamma}_{\gamma\alpha}^{\gamma}+(n+1)Z_{\bar{\beta}}Z_{\alpha}u+\mathscr{O}_{m-1}.
\end{align}
Again by the first identity of \eqref{eqA.10}, we have $\widehat{\Gamma}_{ab}^c={\Gamma}_{ab}^c+\mathscr{O}_{m-1}$. And by \eqref{eq3.1b}, we have ${\Gamma}_{ab}^c=\mathscr{O}_{1}$. So we get
\begin{equation}
\widehat{\Gamma}_{ab}^c\widehat{\Gamma}_{de}^f={\Gamma}_{ab}^c{\Gamma}_{de}^f+\mathscr{O}_{m},
\end{equation}
for any indices $a,b,c,d,e,f$. By the second identity of \eqref{eqA.10}, $J_{\gamma\bar{\beta}}=-i\delta_{\gamma\bar{\beta}}$ in \eqref{eq2.18} and $u_0=Tu=\frac{\partial{u}}{\partial{t}}+\mathscr{O}_m$ by \eqref{eq2.36}, we have
\begin{align}
\notag
2\widehat{\Gamma}_{\widehat{0}\alpha}^{\gamma}\widehat{J}_{\gamma\bar{\beta}}&=2\bigg(\Gamma_{0\alpha}^{\gamma}+u_0\delta_{\alpha}^{\gamma}
-\frac{i}{2}Z_{\bar{\gamma}}Z_{\alpha}u-\frac{i}{2}Z_{\alpha}Z_{\bar{\gamma}}u\bigg)J_{\gamma\bar{\beta}}+\mathscr{O}_{m-1}\\
&=2\Gamma_{0\alpha}^{\gamma}J_{\gamma\bar{\beta}}-2i\frac{\partial{u}}{\partial{t}}\delta_{\alpha\bar{\beta}}
-Z_{\bar{\beta}}Z_{\alpha}u-Z_{\alpha}Z_{\bar{\beta}}u+\mathscr{O}_{m-1}.
\end{align}
Noting that $[Z_{\alpha},Z_{\bar{\beta}}]=[\frac{\partial}{\partial{z}^{\alpha}}
-iz^{\bar{\alpha}}\frac{\partial}{\partial{t}},\frac{\partial}{\partial{z}^{\bar{\beta}}}
+iz^{\beta}\frac{\partial}{\partial{t}}]
=2i\delta_{\alpha\bar{\beta}}\frac{\partial}{\partial{t}}$, \eqref{eqA.14} leads to
\begin{align}
  \notag \widehat{R}_{\alpha\ \gamma\bar{\beta}}^{\ \gamma}&=R_{\alpha\ \gamma\bar{\beta}}^{\ \gamma}-2i\delta_{\alpha\bar{\beta}}\frac{\partial{u}}{\partial{t}}-(n+2)Z_{\bar{\beta}}Z_{\alpha}u
  -\delta_{\alpha\bar{\beta}}Z_{\gamma}Z_{\bar{\gamma}}u+\mathscr{O}_{m-1}\\
  \notag &=R_{\alpha\ \gamma\bar{\beta}}^{\ \gamma}-2i\delta_{\alpha\bar{\beta}}\frac{\partial{u}}{\partial{t}}
  -\frac{n+2}{2}\bigg(Z_{\bar{\beta}}Z_{\alpha}u+Z_{\alpha}Z_{\bar{\beta}}u
  -2i\delta_{\alpha\bar{\beta}}\frac{\partial{u}}{\partial{t}}\bigg)\\
  \notag &\ \ \ -\frac{1}{2}\delta_{\alpha\bar{\beta}}\bigg(Z_{\gamma}Z_{\bar{\gamma}}u+Z_{\bar{\gamma}}Z_{\gamma}u
  +2ni\frac{\partial{u}}{\partial{t}}\bigg)+\mathscr{O}_{m-1}\\
  \notag &=R_{\alpha\ \gamma\bar{\beta}}^{\ \gamma}-\frac{n+2}{2}\bigg(Z_{\bar{\beta}}Z_{\alpha}u+Z_{\alpha}Z_{\bar{\beta}}u\bigg)
  +\frac{1}{2}{\delta}_{\alpha\bar{\beta}}\mathscr{L}_0u+\mathscr{O}_{m-1},
\end{align}
with $\mathscr{L}_0=-(Z_{\alpha}Z_{\bar{\alpha}}+Z_{\bar{\alpha}}Z_{\alpha})$.
\subsection{The Covariance of $R_{abcd}$ and $A_{ab}$}
  Fix a contact Riemannian structure $(M,\theta,h,J)$ and a connection $\nabla$. For a frame $\{W_a\}$ of $HM$, $\{W_j\}=\{W_a,T\}$ is a frame of $TM$. Take an invertible transformation for $\{W_a\}$ by writing
  \begin{equation}\label{eqA.16}
  \widetilde{W_a}=v_a^cW_c,
  \end{equation}
  for an invertible matrix $(v_a^c)$. $\widetilde{T}=T$ by $\theta$ is fixed. Let $\{\widetilde\theta^b,\widetilde\theta\}$ be the coframe dual to $\{\widetilde{W_a},\widetilde{T}\}$. Then we can write
  \begin{equation}\label{eqA.17}
  \widetilde\theta^b=u_c^b\theta^c,{\quad}\theta^c=v_b^c\widetilde\theta^b,
  \end{equation}
  where $(u_b^c)$ is the inverse matrix of $(v_a^c)$, i.e. $u_c^bv_a^c=\delta_a^b$. We write $\nabla{W_a}=\omega_a^c\otimes{W_c}$ and $\nabla\widetilde{W_a}=\widetilde\omega_a^c\otimes{W_c}$.
  \begin{proposition}
  With the transformation mentioned above, Let $\widetilde{A}_{ab}$ and $\widetilde{R}_{abcd}$ be the Webster torsion tensor and the curvature tensor with respect to $\{\widetilde{W_a},T\}$. Then we have
  $$\widetilde{A}_{ab}=v_a^{a_1}v_b^{b_1}A_{a_1b_1},{\quad}\widetilde{R}_{abcd}=v_a^{a_1}v_b^{b_1}v_c^{c_1}v_d^{d_1}R_{a_1b_1c_1d_1}.$$
  \end{proposition}
  \begin{proof}
  By \eqref{eqA.16}, we have
  \begin{align}
  \notag \nabla\widetilde{W_a}&=\nabla(v_a^cW_c)=dv_a^c\otimes{W_c}+v_a^c\nabla{W_c}=\big(u_c^bdv_a^c+u_c^bv_a^d\omega_d^c\big)\otimes\widetilde{W_b}.
  \end{align}
  So $\widetilde\omega_a^b=u_c^bdv_a^c+u_c^bv_a^d\omega_d^c$, which is equivalent to
  \begin{equation}\label{eqA.18}
  dv_a^c=\widetilde\omega_a^bv_b^c-v_a^b\omega_b^c.
  \end{equation}

  Differentiating the second identity of \eqref{eqA.17}, we get:
  \begin{align}
  \notag &d\theta^c=dv_a^c\wedge\widetilde\theta^a+v_a^cd\widetilde\theta^a=\big(\widetilde\omega_a^bv_b^c-v_a^b\omega_b^c\big)\wedge\widetilde\theta^a+v_a^cd\widetilde\theta^a,
  \end{align}
  by \eqref{eqA.18}. Therefore,
  \begin{align}\label{eqA.19}
  d\widetilde\theta^a-\widetilde\theta^b\wedge\widetilde\omega_b^a=u_b^a\big(d\theta^b-\theta^c\wedge\omega_c^b\big).
  \end{align}
    So $\widetilde{A}_a^b=(d\widetilde\theta^b-\widetilde\theta^c\wedge\widetilde\omega_c^b)(\widetilde{T},\widetilde{W_a})
  =u_{b_1}^b\big(d\theta^{b_1}-\theta^c\wedge\omega_c^{b_1}\big)(T,v_a^{a_1}W_{a_1})=v_a^{a_1}u_{b_1}^b{A}_{a_1}^{b_1}.$
  For the covariance of $R_{abcd}$, we can refer to Section 3 in \cite{We1}, or it can be achieved in the same way as tensor $A$ by differentiating \eqref{eqA.18}. So we finish proving that $R_{abcd}$ and $A_{ab}$ are covariant.
  \end{proof}
\section{The calculation of $a_2(n)$ and $b_2(n)$}
  In Appendix B we will show the detailed calculation of Lemma \ref{lem5.6}. Recall that our contact manifold $(M,\theta,h,J)$ satisfies Proposition \ref{prop5.1}.
  \subsection{Calculation of $v_2^{jk}$}
  \begin{lemma}\label{lemB.1}
  For $v_2^{jk}$ defined in \eqref{eq5.10a}, we have
  \begin{equation}\label{eqB.8}
  \begin{aligned}
  &v_2^{\alpha\gamma}=-\frac{1}{6}R_{d\ c\bar{\alpha}}^{\ \gamma}(q)z^dz^c,\\
  &v_2^{\alpha\bar{\gamma}}=-\frac{1}{6}\bigg(R_{d\ c\gamma}^{\ \alpha}(q)+R_{d\ c\bar{\alpha}}^{\ \bar{\gamma}}(q)\bigg)z^cz^d+\delta_{\beta}^{\alpha}\delta_{\beta}^{\gamma}v_2,\\
  &v_2^{\bar{\alpha}\gamma}=0,\\
  &v_2^{\bar{\alpha}\bar{\gamma}}=-\frac{1}{6}R_{d\ c\gamma}^{\ \bar{\alpha}}(q)z^dz^c,\\
  &v_2^{\alpha0}=-\frac{1}{2}J_{\beta\bar{\alpha}(2)}z^{\beta}-\frac{1}{2}J_{\bar{\beta}\bar{\alpha}(2)}z^{\bar{\beta}}
  +\frac{i}{12}R_{d\ c\bar{\alpha}}^{\ \bar{\rho}}(q)z^dz^cz^{\rho}-\frac{i}{12}R_{d\ c\bar{\alpha}}^{\ \rho}(q)z^dz^cz^{\bar{\rho}},\\
  &v_2^{\bar{\alpha}0}=0,\\
  &v_2^{0\alpha}=0,\\
  &v_2^{0\bar{\alpha}}=-\frac{1}{2}J_{\beta\alpha(2)}z^{\beta}-\frac{1}{2}J_{\bar{\beta}\alpha(2)}z^{\bar{\beta}}
  +\frac{i}{12}R_{d\ c\alpha}^{\ \bar{\rho}}(q)z^dz^cz^{\rho}-\frac{i}{12}R_{d\ c\alpha}^{\ \rho}(q)z^dz^cz^{\bar{\rho}},\\
  &v_2^{00}=\frac{4}{9}Q_{\gamma\lambda}^{\bar{\beta}}(q)Q_{\bar{\sigma}\bar{\mu}}^{\beta}(q)z^{\gamma}z^{\lambda}z^{\bar{\sigma}}z^{\bar{\mu}}.
  \end{aligned}
  \end{equation}
  \end{lemma}
  \begin{proof}
  In the following we will use Proposition \ref{prop2.2} repeatedly, especially
  $$s_{\beta(0)}^{\alpha}=s_{\bar{\beta}(0)}^{\bar{\alpha}}=\delta_{\beta}^{\alpha},{\quad}s_{\beta(0)}^{\bar{\alpha}}=s_{\bar{\beta}(0)}^{\alpha}=0,{\quad}s_{b(1)}^a=0,{\quad}s_{b(0)}^0=s_{b(1)}^0=0,$$
  and we also have $v_0=1$, $v_1=0$, by Corollary \ref{cor5.1}. We find that
  \begin{align}
  \notag v_2^{\alpha\gamma}&=\sum\limits_{m_0+m_1+m_2=2}s_{\beta(m_1)}^{\alpha}s_{\bar{\beta}(m_2)}^{\gamma}v_{m_0}\\
  \notag &=s_{\beta(2)}^{\alpha}s_{\bar{\beta}(0)}^{\gamma}v_{0}+s_{\beta(0)}^{\alpha}s_{\bar{\beta}(2)}^{\gamma}v_{0}
  +s_{\beta(1)}^{\alpha}s_{\bar{\beta}(1)}^{\gamma}v_{0}+s_{\beta(1)}^{\alpha}s_{\bar{\beta}(0)}^{\gamma}v_{1}
  +s_{\beta(0)}^{\alpha}s_{\bar{\beta}(1)}^{\gamma}v_{1}+s_{\beta(0)}^{\alpha}s_{\bar{\beta}(0)}^{\gamma}v_{2}\\
  \notag &=\delta_{\beta}^{\alpha}s_{\bar{\beta}(2)}^{\gamma}v_0=s_{\bar{\alpha}(2)}^{\gamma}
  =-\frac{1}{6}R_{d\ c\bar{\alpha}}^{\ \gamma}(q)z^dz^c,
  \end{align}
  by \eqref{eq2.29} for $s_{\bar{\beta}(2)}^{\alpha}$. Similarly we get
  \begin{align}
  \notag v_2^{\alpha\bar{\gamma}}&=\sum\limits_{m_0+m_1+m_2=2}s_{\beta(m_1)}^{\alpha}s_{\bar{\beta}(m_2)}^{\bar{\gamma}}v_{m_0}
  =s_{\beta(2)}^{\alpha}\delta_{\bar{\beta}}^{\bar{\gamma}}+\delta_{\beta}^{\alpha}s_{\bar{\beta}(2)}^{\bar{\gamma}}
  +\delta_{\beta}^{\alpha}\delta_{\bar{\beta}}^{\bar{\gamma}}v_2
  =s_{\gamma(2)}^{\alpha}+s_{\bar{\alpha}(2)}^{\bar{\gamma}}+\delta_{\gamma}^{\alpha}v_2\\
  \notag &=-\frac{1}{6}\bigg(R_{d\ c\gamma}^{\ \alpha}(q)+R_{d\ c\bar{\alpha}}^{\ \bar{\gamma}}(q)\bigg)z^dz^c+\delta_{\gamma}^{\alpha}v_2,\\
  \notag v_2^{\bar{\alpha}\gamma}&=\sum\limits_{m_0+m_1+m_2=2}s_{\beta(m_1)}^{\bar{\alpha}}s_{\bar{\beta}(m_2)}^{\gamma}v_{m_0}=0,\\
  \notag v_2^{\bar{\alpha}\bar{\gamma}}&=\sum\limits_{m_0+m_1+m_2=2}s_{\beta(m_1)}^{\bar{\alpha}}s_{\bar{\beta}(m_2)}^{\bar{\gamma}}v_{m_0}
  =s_{\beta(2)}^{\bar{\alpha}}\delta_{\bar{\beta}}^{\bar{\gamma}}=s_{\gamma(2)}^{\bar{\alpha}}
  =-\frac{1}{6}R_{d\ c\gamma}^{\ \bar{\alpha}}(q)z^dz^c,\\
  \notag v_2^{\alpha0}&=\sum\limits_{m_0+m_1+m_2=2}s_{\beta(m_1)}^{\alpha}s_{\bar{\beta}(m_2+1)}^0v_{m_0}
  =\delta_{\beta}^{\alpha}s_{\bar{\beta}(3)}^0=s_{\bar{\alpha}(3)}^0\\
  \notag &=-\frac{1}{2}J_{\beta\bar{\alpha}(2)}z^{\beta}-\frac{1}{2}J_{\bar{\beta}\bar{\alpha}(2)}z^{\bar{\beta}}+\frac{i}{12}R_{d\ c\bar{\alpha}}^{\ \bar{\rho}}(q)z^dz^cz^{\rho}-\frac{i}{12}R_{d\ c\bar{\alpha}}^{\ \rho}(q)z^dz^cz^{\bar{\rho}},\\
  \notag v_2^{\bar{\alpha}0}&=\sum\limits_{m_0+m_1+m_2=2}s_{\beta(m_1)}^{\bar{\alpha}}s_{\bar{\beta}(m_2+1)}^0v_{m_0}=0,\\
  \notag v_2^{0\alpha}&=\sum\limits_{m_0+m_1+m_2=2}s_{\beta(m_1+1)}^0s_{\bar{\beta}(m_2)}^{\alpha}v_{m_0}=0,\\
  \notag v_2^{0\bar{\alpha}}&=\sum\limits_{m_0+m_1+m_2=2}s_{\beta(m_1+1)}^0s_{\bar{\beta}(m_2)}^{\bar{\alpha}}v_{m_0}
  =s_{\beta(3)}^0\delta_{\bar{\beta}}^{\bar{\alpha}}=s_{\alpha(3)}^0\\
  \notag &=-\frac{1}{2}J_{\beta\alpha(2)}z^{\beta}-\frac{1}{2}J_{\bar{\beta}\alpha(2)}z^{\bar{\beta}}+\frac{i}{12}R_{d\ c\alpha}^{\ \bar{\rho}}(q)z^dz^cz^{\rho}-\frac{i}{12}R_{d\ c\alpha}^{\ \rho}(q)z^dz^cz^{\bar{\rho}}.
  \end{align}
  By noting that we have Corollary \ref{cor4.1} for $s_{b(2)}^0$, we get
  \begin{align}
  \notag v_2^{00}=\sum\limits_{m_0+m_1+m_2=2}s_{\beta(m_1+1)}^0s_{\bar{\beta}(m_2+1)}^0v_{m_0}=s_{\beta(2)}^0s_{\bar{\beta}(2)}^0
  =\frac{4}{9}Q_{\gamma\lambda}^{\bar{\beta}}(q)Q_{\bar{\sigma}\bar{\mu}}^{\beta}(q)z^{\gamma}z^{\lambda}z^{\bar{\sigma}}z^{\bar{\mu}}.
  \end{align}
So we finish the proof of Lemma B.1.
\end{proof}
\subsection{Proof of Lemma \ref{lem5.6}}
By the first identity  in \eqref{eqB.8} for $v_2^{\alpha\gamma}$ and \eqref{eqB.17}, we get
  \begin{align}\label{eqB.20}
  \notag &\ \ \ \int_{\mathscr{H}^n}v_2^{\alpha\gamma}Z_{\alpha}\Phi{Z_{\gamma}\Phi}dV\\
  \notag &=\int_{\mathscr{H}^n}\frac{n^2}{6}R_{d\ c\bar{\alpha}}^{\ \gamma}(q)
  z^dz^{\bar{\gamma}}z^cz^{\bar{\alpha}}\frac{t^2+2i(|z|^2+1)t-(|z|^2+1)^2}{|w+i|^{2n+4}}dV\\
  \notag &=\int_{\mathscr{H}^n}\frac{n^2}{6}R_{\rho\ \lambda\bar{\alpha}}^{\ \gamma}(q)
  z^{\rho}z^{\bar{\gamma}}z^{\lambda}z^{\bar{\alpha}}\frac{t^2+2i(|z|^2+1)t-(|z|^2+1)^2}{|w+i|^{2n+4}}dV\\
  &=\frac{n^2(4\pi)^n}{12(n+1)}\mathfrak{Q}
  \int_{-\infty}^{\infty}\int_0^{\infty}\frac{t^2+2i(r^2+1)t-(r^2+1)^2}{|1+i(1+r^2)|^{2n+4}}r^{2n+3}drdt,
  \end{align}
  where the last identity is by the third identity in \eqref{eqB.1}. Similarly by \eqref{eqB.8}, \eqref{eqB.1} and \eqref{eqB.17}, we get
  \begin{align}\label{eqB.21}
  \int_{\mathscr{H}^n}v_2^{\bar{\alpha}\bar{\gamma}}Z_{\bar{\alpha}}\Phi{Z_{\bar{\gamma}}\Phi}dV
  =\frac{n^2(4\pi)^n}{12(n+1)}\mathfrak{Q}\int_{-\infty}^{\infty}\int_0^{\infty}\frac{t^2-2i(r^2+1)t
  -(r^2+1)^2}{|1+i(1+r^2)|^{2n+4}}r^{2n+3}drdt.
  \end{align}
Recall that we have \eqref{eqB.7a} for $v_2$. So we have
  \begin{align}
  \notag &\ \ \ \int_{\mathscr{H}^n}v_2n^2|z|^2\frac{t^2+(|z|^2+1)^2}{|w+i|^{2n+4}}dV\\
  \notag &=-\frac{1}{6}\int_{\mathscr{H}^n}\big(
  R_{\bar{\beta}\ \alpha\mu}^{\ \alpha}(q)z^{\bar{\beta}}z^{\mu}+R_{\beta\ \bar{\alpha}\bar{\mu}}^{\ \bar{\alpha}}(q)z^{\beta}z^{\bar{\mu}}\big)n^2|z|^2\frac{t^2+(|z|^2+1)^2}{|w+i|^{2n+4}}dV\\
  \notag &=\frac{n^2}{6}(4\pi)^n\mathfrak{Q}\int_{-\infty}^{\infty}\int_0^{\infty}\frac{r^{2n+3}}{|t+i(1+r^2)|^{2n+2}}drdt=\frac{n^2}{6}(4\pi)^n\mathfrak{Q}N_1(2n+2,2n+3,0),
  \end{align}
  by \eqref{eqB.1} and Lemma \ref{lem5.4}. Then by \eqref{eqB.8} and \eqref{eqB.17}, we get
  \begin{equation}\label{eqB.22}
  \begin{aligned}
  &\ \ \ \int_{\mathscr{H}^n}v_2^{\alpha\bar{\gamma}}Z_{\alpha}\Phi{Z_{\bar{\gamma}}\Phi}dV\\
  &=\int_{\mathscr{H}^n}\bigg(-\frac{n^2}{6}(R_{d\ c\gamma}^{\ \alpha}(q)z^dz^c+R_{d\ c\bar{\alpha}}^{\ \bar{\gamma}}(q)z^dz^c)+\delta_{\gamma}^{\alpha}v_2\bigg)z^{\bar{\alpha}}z^{\gamma}\frac{t^2+(|z|^2+1)^2}{|w+i|^{2n+4}}dV\\
  &=\int_{\mathscr{H}^n}\bigg(\frac{n^2}{6}\big(R_{\rho\ \gamma\bar{\lambda}}^{\ \alpha}(q)z^{\rho}z^{\bar{\alpha}}z^{\gamma}z^{\bar{\lambda}}-R_{\bar{\rho}\ \lambda\gamma}^{\ \alpha}(q)z^{\bar{\rho}}z^{\bar{\alpha}}z^{\lambda}z^{\gamma}\\
  &\ \ -R_{\rho\ \bar{\lambda}\bar{\alpha}}^{\ \bar{\gamma}}(q)z^{\rho}z^{\gamma}z^{\bar{\lambda}}z^{\bar{\alpha}}+R_{\bar{\rho}\ \bar{\alpha}\lambda}^{\ \bar{\gamma}}(q)z^{\bar{\rho}}z^{\gamma}z^{\bar{\alpha}}z^{\lambda}\big)
  \frac{t^2+(|z|^2+1)^2}{|w+i|^{2n+4}}+v_2n^2|z|^2\frac{t^2+(|z|^2+1)^2}{|w+i|^{2n+4}}\bigg)dV\\
  &=\frac{n^2}{6(n+1)}(4\pi)^n\mathfrak{Q}\int_{-\infty}^{\infty}\int_0^{\infty}
  \frac{\big(t^2+(r^2+1)^2\big)r^{2n+3}}{|t^2+i(1+r^2)|^{2n+4}}drdt+\frac{n^2}{6}(4\pi)^n\mathfrak{Q}N_1(2n+2,2n+3,0).
  \end{aligned}
  \end{equation}
  The last identity is by the third and fourth identities in \eqref{eqB.1}. And by \eqref{eqB.8}, we have
  \begin{equation}\label{eqB.23}
  \int_{\mathscr{H}^n}v_2^{\bar{\alpha}\gamma}Z_{\bar{\alpha}}{\Phi}Z_{\gamma}{\Phi}dV=0.
  \end{equation}
  Taking summation of \eqref{eqB.20}, \eqref{eqB.21}, \eqref{eqB.22} and \eqref{eqB.23}, we get
  \begin{align}
  \notag &\ \ \ \int_{\mathscr{H}^n}v_2^{ab}Z_a{\Phi}Z_b{\Phi}dV\\
  \notag &=\frac{n^2}{3(n+1)}(4\pi)^n\mathfrak{Q}\int_{-\infty}^{\infty}\int_0^{\infty}\frac{t^2r^{2n+3}}{|t^2+i(1+r^2)|^{2n+4}}drdt+\frac{n^2}{6}(4\pi)^n\mathfrak{Q}N_1(2n+2,2n+3,0)\\
  \notag &=\frac{n^2}{3(n+1)}(4\pi)^nN_1(2n+4,2n+3,2)\mathfrak{Q}+\frac{n^2}{6}(4\pi)^nN_1(2n+2,2n+3,0)\mathfrak{Q}\\
  \notag &=\frac{n^4+2n^3+2n^2}{6(n-1)(n+1)}(4\pi)^nN_1(2n+2,2n+1,0)\mathfrak{Q},
  \end{align}
  by using the last identity in \eqref{N1}. So the first identity in \eqref{eq5.18} follows.

  By \eqref{eqB.8}, \eqref{eqB.17} and substituting identities in \eqref{eqB.1} for certain terms, we get
  \begin{align}\label{eqB.24}
  \notag &\ \ \ \int_{\mathscr{H}^n}v_2^{\alpha0}Z_{\alpha}{\Phi}Z_0{\Phi}dV\\
  \notag &=\int_{\mathscr{H}^n}\bigg(-\frac{1}{2}J_{\beta\bar{\alpha}(2)}z^{\beta}-\frac{1}{2}J_{\bar{\beta}\bar{\alpha}(2)}z^{\bar{\beta}}
  +\frac{i}{12}R_{d\ c\bar{\alpha}}^{\ \bar{\rho}}(q)z^cz^dz^{\rho}-\frac{i}{12}R_{d\ c\bar{\alpha}}^{\ \rho}(q)z^cz^dz^{\bar{\rho}}\bigg)\\
  \notag &\ \ \ (-in^2z^{\bar{\alpha}})\frac{t^2+it(|z|^2+1)}{|w+i|^{2n+4}}dV\\
  \notag &=n^2\int_{\mathscr{H}^n}\bigg(\frac{i}{2}J_{\beta\bar{\alpha}(2)}z^{\beta}z^{\bar{\alpha}}+\frac{i}{2}J_{\bar{\beta}\bar{\alpha}(2)}z^{\bar{\beta}}z^{\bar{\alpha}}
  +\frac{1}{12}\big(R_{\beta\ \bar{\mu}\bar{\alpha}}^{\ \bar{\rho}}(q)z^{\beta}z^{\rho}z^{\bar{\mu}}z^{\bar{\alpha}}-R_{\bar{\mu}\ \bar{\alpha}\beta}^{\ \bar{\rho}}(q)z^{\bar{\mu}}z^{\rho}z^{\bar{\alpha}}z^{\beta}\big)\\
  \notag &\ \ -\frac{1}{12}R_{\beta\ \mu\bar{\alpha}}^{\ \rho}(q)z^{\beta}z^{\bar{\rho}}z^{\mu}z^{\bar{\alpha}}\bigg)\frac{t^2+it(|z|^2+1)}{|w+i|^{2n+4}}dV\\
  \notag &=\frac{n^2}{n+1}(4\pi)^n\bigg(-\frac{3}{4}+0+0-\frac{1}{24}-\frac{1}{24}\bigg)\mathfrak{Q}\int_{-\infty}^{\infty}\int_0^{\infty}\frac{t^2+it(r^2+1)}{|t^2+i(1+r^2)|^{2n+4}}r^{2n+3}drdt\\
  &=-\frac{5n^2}{6(n+1)}(4\pi)^n\mathfrak{Q}\int_{-\infty}^{\infty}\int_0^{\infty}\frac{t^2+it(r^2+1)}{|t^2+i(1+r^2)|^{2n+4}}r^{2n+3}drdt.
  \end{align}
By taking conjugation of \eqref{eqB.24}, we get
  \begin{align}\label{eqB.25}
  \int_{\mathscr{H}^n}v_2^{0\bar{\alpha}}Z_0{\Phi}Z_{\bar{\alpha}}{\Phi}dV
  =-\frac{5n^2}{6(n+1)}(4\pi)^n\mathfrak{Q}\int_{-\infty}^{\infty}\int_0^{\infty}\frac{t^2-it(r^2+1)}{|t^2+i(1+r^2)|^{2n+4}}r^{2n+3}drdt.
  \end{align}
And by \eqref{eqB.8}, we get
  \begin{equation}\label{eqB.26}
  \int_{\mathscr{H}^n}v_2^{0\alpha}Z_0{\Phi}Z_{\alpha}{\Phi}dV=\int_{\mathscr{H}^n}v_2^{\bar{\alpha}0}Z_{\bar{\alpha}}{\Phi}Z_0{\Phi}dV=0.
  \end{equation}
So taking summation of \eqref{eqB.24}, \eqref{eqB.25} and \eqref{eqB.26}, we get
  \begin{align}
  \notag &\ \ \ \int_{\mathscr{H}^n}v_2^{a0}Z_a{\Phi}Z_0{\Phi}+v_2^{0a}Z_0{\Phi}Z_a{\Phi}dV\\
  \notag &=\int_{\mathscr{H}^n}v_2^{\alpha0}Z_{\alpha}{\Phi}Z_0{\Phi}+v_2^{\bar{\alpha}0}Z_{\bar{\alpha}}{\Phi}Z_0{\Phi}
  +v_2^{0\alpha}Z_0{\Phi}Z_{\alpha}{\Phi}+v_2^{0\bar{\alpha}}Z_0{\Phi}Z_{\bar{\alpha}}{\Phi}dV\\
  \notag &=-\frac{5n^2}{3(n+1)}(4\pi)^n\mathfrak{Q}\int_{-\infty}^{\infty}\int_0^{\infty}\frac{r^{2n+3}t^2}{|t^2+i(1+r^2)|^{2n+4}}drdt\\
  \notag &=-\frac{5n^2}{3(n+1)}(4\pi)^n\mathfrak{Q}N_1(2n+4,2n+3,2)=-\frac{5n^2}{6(n+1)(n-1)}\mathfrak{Q}N_1(2n+2,2n+1,0),
  \end{align}
  by \eqref{N1} and Lemma \ref{lem5.4}. So the second identity in \eqref{eq5.18} follows.

  By \eqref{eqB.8}, \eqref{N1}, \eqref{eqB.17}, the fifth identity in \eqref{eqB.1} and Lemma \ref{lem5.4}, we get
  \begin{align}
  \notag &\ \ \ \int_{\mathscr{H}^n}v_2^{00}Z_0{\Phi}Z_0{\Phi}dV=\int_{\mathscr{H}^n}\frac{4}{9}Q_{\gamma\lambda}^{\bar{\beta}}(q)Q_{\bar{\sigma}\bar{\mu}}^{\beta}(q)z^{\gamma}z^{\lambda}z^{\bar{\sigma}}z^{\bar{\mu}}\frac{n^2t^2}{|w+i|^{2n+4}}dV\\
  \notag &=\int_{-\infty}^{\infty}\int_0^{\infty}\frac{4n^2}{3(n+1)}(4\pi)^n\mathfrak{Q}\frac{r^{2n+3}t^2}{|t^2+i(1+r^2)|^{2n+4}}drdt\\
  \notag &=\frac{4n^2}{3(n+1)}(4\pi)^nN_1(2n+4,2n+3,2)\mathfrak{Q}=\frac{2n^2}{3(n+1)(n-1)}(4\pi)^nN_1(2n+2,2n+1,0)\mathfrak{Q}.
  \end{align}
  So the third identity in \eqref{eq5.18} follows.
\end{appendix}

  \bibliographystyle{plain}

 \end{document}